\documentclass{article}

\PassOptionsToPackage{sort}{natbib}

\usepackage[preprint, final]{neurips26}

\usepackage[utf8]{inputenc} %
\usepackage[T1]{fontenc}    %
\usepackage[unicode]{hyperref}  %
\usepackage{url}            %
\usepackage{booktabs}       %
\usepackage{amsfonts}       %
\usepackage{nicefrac}       %
\usepackage{microtype}      %
\usepackage{xcolor}         %
\usepackage{amsmath}
\usepackage{amssymb}
\usepackage{amsthm}
\usepackage{thmtools, thm-restate}
\usepackage{graphicx}
\usepackage{tikz}
\usepackage{nicematrix}
\usepackage{enumitem}
\usepackage{wrapfig}
\usepackage[nottoc]{tocbibind}

\usepackage[clean]{macro/main}
\usepackage{macro/theoremsave}

\definecolor{pastalred}{HTML}{e77c8d}
\newcommand{\declarecolor}[2]{\definecolor{#1}{RGB}{#2}\expandafter\newcommand\csname #1\endcsname[1]{\textcolor{#1}{##1}}}
\declarecolor{White}{255, 255, 255}
\declarecolor{Black}{0, 0, 0}
\declarecolor{Maroon}{128, 0, 0}
\declarecolor{Coral}{255, 127, 80}
\declarecolor{Red}{182, 21, 21}
\declarecolor{LimeGreen}{50, 205, 50}
\declarecolor{DarkGreen}{0, 80, 0}
\declarecolor{Purple}{146, 42, 158}
\declarecolor{Navy}{0, 0, 128}
\declarecolor{LightBlue}{84, 101, 202}
\hypersetup{
    pdftitle={Accelerating Min-Max Optimization via Power-Law Stepsizes},
    pdfauthor={Yue Wu; Weiqiang Zheng; Yang Cai; Haipeng Luo},
    pdfsubject={Dynamic power-law stepsize schedules for accelerated last-iterate convergence in min-max optimization},
    pdfkeywords={min-max optimization, extragradient, optimistic gradient, power-law stepsizes, last-iterate convergence},
    pdflang={en-US},
    colorlinks,
    citecolor=DarkGreen,
    linkcolor=LightBlue}
\usepackage[noabbrev, capitalise, nameinlink]{cleveref}
\usepackage[framemethod=TikZ]{mdframed}
\mdfsetup{%
roundcorner=4pt,
linewidth=1pt}

\declaretheorem[sibling=theorem]{lemma, corollary}
\newtheorem*{theorem*}{Theorem}

\declaretheorem[style=remark, numbered=no]{remark, observation}

\declaretheorem[within=section, name=Theorem]{apptheorem}
\declaretheorem[sibling=apptheorem, name=Lemma]{applemma}

\declaretheorem[sibling=apptheorem, style=remark, name=Remark]{appremark}
\declaretheorem[sibling=apptheorem, style=definition, name=Definition]{appdefinition}
\declaretheorem[sibling=apptheorem, style=definition, name=Assumption]{appassumption}

\def\+#1{\mathcal{#1}}
\def\-#1{\mathbb{#1}}

\NiceMatrixOptions{
    notes/style=\alph{#1}
}

\newcommand{\Xcomment}[1]{{}}

\renewcommand{\part}[2]{\frac{\partial #1}{\partial #2}}

\title{Accelerating Min-Max Optimization \\ via Power-Law Stepsizes}

\date{June 1, 2026}
\author{%
  \mdseries
  \begin{minipage}{0.45\textwidth}
    \centering
    \textbf{Yue Wu} \\
    University of Southern California \\
    \texttt{wu.yue@usc.edu}
  \end{minipage}\hfill
  \begin{minipage}{0.45\textwidth}
    \centering
    \textbf{Weiqiang Zheng} \\
    Yale University \\
    \texttt{weiqiang.zheng@yale.edu}
  \end{minipage} \\[2.5em]
  \begin{minipage}{0.45\textwidth}
    \centering
    \textbf{Yang Cai} \\
    Yale University \\
    \texttt{yang.cai@yale.edu}
  \end{minipage}\hfill
  \begin{minipage}{0.45\textwidth}
    \centering
    \textbf{Haipeng Luo} \\
    University of Southern California \\
    \texttt{haipengl@usc.edu}
  \end{minipage}
}

\begin{document}

\maketitle

\begin{abstract}
We revisit the convergence guarantees of the Extragradient (EG) method for unconstrained \draftcolor{cyan}{biaffine} min-max optimization. It is known that EG with a fixed stepsize achieves a $\Theta(T^{-1/2})$ last-iterate convergence rate, which is slower than the optimal $\mathcal{O}(T^{-1})$ rate attainable by incorporating additional mechanisms such as anchoring.
Motivated by recent advances showing that dynamic stepsizes alone can significantly accelerate gradient descent, we ask whether dynamic stepsizes can similarly accelerate the last-iterate convergence of EG.

We present the first positive result in this direction. Specifically, we provide a deterministic dynamic stepsize schedule that accelerates the convergence rate of EG to $\mathcal{O}(T^{-2/3+\varepsilon})$ for any $\varepsilon > 0$. We also show that this rate is nearly tight when the extrapolation and update steps of EG use the same stepsize. We then show that allowing different stepsizes for the extrapolation and update steps further improves the convergence rate to the near-optimal $\mathcal{O}(T^{-1+\varepsilon})$.
Our analysis reduces stepsize scheduling to an optimization problem, whose solution leads to a stepsize schedule that follows (a discretization of) a power-law distribution.
Our proposed stepsize schedules and analysis extend to other methods, such as Optimistic Gradient (OG), and suggest broader applicability to general min-max optimization problems.
\end{abstract}

\section{Introduction}\label{h1:intro}
Extragradient (EG)~\citep{korpelevich1976extragradient} is a simple iterative algorithm
for solving min-max optimization problems of the form $\min_{x \in \-R^n}\max_{y \in \-R^m} \ell (x,y)$,
where in each iteration $t$, an extrapolation step (with stepsize $\gamma_t$) is followed by an update step (with stepsize $\eta_t$):
\begin{align}
    \bigg\{\begin{aligned}
        x_{t+\frac 12} &= x_t-\gamma_t\nabla_x \ell(x_t,y_t),\\
        y_{t+\frac 12} &= y_t+\gamma_t\nabla_y \ell(x_t,y_t),
    \end{aligned}\qquad
    \bigg\{\begin{aligned}
        x_{t+1} &= x_t-\eta_t\nabla_x \ell(x_{t+\frac 12},y_{t+\frac 12}),\\
        y_{t+1} &= y_t+\eta_t\nabla_y \ell(x_{t+\frac 12},y_{t+\frac 12}),
    \end{aligned} \quad t=0,1,2,\dots
    \label{eq:EG}
\end{align}
The extrapolation step in EG is crucial, as the vanilla gradient descent-ascent
(GDA) 
algorithm diverges even for a bilinear $\ell$. 
Despite being proposed 50 years ago, EG (and its variants) remains a standard method for solving large-scale problems in modern optimization, machine learning, and game theory, due to its efficiency, simplicity, and scalability. It is thus of both theoretical and practical interest to analyze the convergence of EG and to tune the parameters to achieve the fastest possible convergence.

In particular, the goal of this work is to understand the fastest convergence rate of EG that can be achieved through stepsize tuning alone. We revisit this problem in the fundamental setting of unconstrained \draftcolor{purple}{biaffine} min-max optimization:
\begin{align}\label{eq:minmax}
    \min_{x \in \-R^n}\max_{y \in \-R^m} \left\{\ell (x,y):= x^\trans Ay + p^\trans y + x^\trans q \right\}, 
\end{align}
where $A \in \-R^{n \times m}$, $p \in \-R^m$, and $q \in \-R^n$. 
A classic result is that, with a suitable constant stepsize schedule $\{\gamma_t =\eta_t = \eta\}$, EG converges to the set of solutions and has an $\calO(T^{-1})$ ergodic convergence rate in terms of the duality gap~\citep{korpelevich1976extragradient, facchinei2003finite}. However, ergodic convergence guarantees only apply to the averaged iterate, rather than to the actual iterate produced by the algorithm. 
This distinction is important in applications where iterates correspond to complex objects, such as outputs of large neural networks, for which averaging may be prohibitively expensive. 
We therefore focus on the non-ergodic last-iterate convergence rate with respect to the gradient norm, which provides a more direct convergence guarantee.  

The rate of last-iterate convergence for EG remained open for a long time until recent work revealed that EG with a constant stepsize achieves a last-iterate convergence rate of $\calO(T^{-\frac{1}{2}})$ with respect to the gradient norm and the restricted gap function~\citep{gorbunov2022extragradient, cai2022finite}. This rate is tight for EG with a constant stepsize: \citet{golowich2020last} showed that, for a class of time-invariant algorithms that includes EG with a constant stepsize, $\Omega(T^{-\frac{1}{2}})$ convergence rate is unavoidable for solving \eqref{eq:minmax}. However, this class does not include EG with dynamic stepsizes, leaving open the possibility that stepsize tuning alone can lead to faster last-iterate convergence.

We also note that recent advances in optimization have achieved the accelerated $\calO(T^{-1})$ last-iterate convergence rate by augmenting EG with additional mechanisms, such as anchoring, momentum, or H-duality. We review this line of work in \Cref{sec:related works}. This rate is optimal among first-order methods for solving \eqref{eq:minmax}~\citep{yoon2021accelerated}. 
This raises a natural question: are such additional mechanisms necessary for acceleration? In other words,
\medskip
\begin{mdframed}[roundcorner=2.5pt,linewidth=0.6pt]
\centering\textit{
    Can we accelerate the convergence of EG by optimizing the stepsizes alone?}
\end{mdframed}

We note that similar questions have been studied for the fundamental gradient descent (GD) algorithm in smooth convex optimization. The classic result shows that the convergence rate of GD with a constant stepsize is $\calO(T^{-1})$, while Nesterov's momentum accelerates GD to an optimal $\calO(T^{-2})$ rate~\citep{nesterov1983method,nemirovskij1983problem}. 
Acceleration using dynamic stepsizes alone dates back to the work of \citet{young1953richardson}, who used Chebyshev stepsizes for convex quadratics, and has recently culminated in the state-of-the-art rate of $T^{-\log_2 \rho} \approx T^{-1.2716}$ (where $\rho=1+\sqrt{2}$) for convex optimization using the Silver stepsize schedule~\citep{grimmer2024provably, altschuler2025accelerationI, altschuler2025acceleration}.
Related schedules have been applied to GDA in min-max optimization problems via the periodically negative Slingshot schedule by \citet{shugart2025negative/preprint}.
We review this line of work in \Cref{sec:related works}.
As summarized in \Cref{tab:parallel}, there is a striking historical and structural parallel between the convergence of GD and EG.
To the best of our knowledge, accelerating these algorithms via stepsize tuning only remain unexplored before this paper. %

\begin{table}[bht]
    \centering
    \renewcommand{\arraystretch}{1.2}
    \caption{Parallel between stepsize tuning in quadratic convex minimization and \draftcolor{purple}{biaffine} min-max optimization. The contributions of this paper are in bold.}
    \label{tab:parallel}
    \begin{NiceTabular}{@{} l l l @{}}
        \toprule
        \textbf{Problem} & \textbf{Convex Optimization} & \textbf{Min-Max Optimization} \\
        \midrule
        Setting & Quadratic & \draftcolor{purple}{Biaffine} \\
        \addlinespace
        Base algorithm & Gradient Descent (GD) & Extragradient (EG) \\
        \addlinespace
        Fixed stepsize rate & $\calO(T^{-1})$ & $\calO(T^{-1/2})$ \\
        \addlinespace
        Dynamic stepsize rate & Silver\tabularnote{Not limited to quadratics; also works for \draftcolor{cyan}{general} smooth convex functions.}: $\calO(T^{-1.27})$ & \textbf{Ours (single):} $\calO(T^{-2/3+\eps})$ \\
         & Chebyshev, Arcsine\tabularnote{Not limited to quadratics; also works for \draftcolor{cyan}{$1$-d} or {separable} smooth convex functions.}: $\calO(T^{-2})$ & \textbf{Ours (double):} $\calO(T^{-1+\eps})$ \\
        \addlinespace
        SOTA (augmentation) & Nesterov (Momentum): $\calO(T^{-2})$ & EAG (Anchoring): $\calO(T^{-1})$ \\
        \addlinespace
        Lower bound & $\Omega(T^{-2})$ & $\Omega(T^{-1})$ \\
        \bottomrule
    \end{NiceTabular}
\end{table}

\paragraph{Contributions}
In this paper, we present the first last-iterate acceleration result for EG obtained solely by stepsize tuning. We show that time-varying, non-monotone stepsize schedules can accelerate EG beyond the constant-stepsize rate and, with a suitable choice of stepsizes, lead to near-optimal last-iterate convergence. We first focus on the common setting where the extrapolation step and the update step of EG share the same stepsize, i.e., $\gamma_t = \eta_t$. In this setting, we
derive a dynamic stepsize schedule that accelerates the convergence rate of EG to $\mathcal{O}(T^{-\frac{2}{3}+\varepsilon})$ for any $\varepsilon > 0$.

\begin{informalthm}{thm:main-eg}
    For any $\eps > 0$, there exists a deterministic stepsize schedule $\{\gamma_t, \eta_t\}_{t\in \dN}$ with $\gamma_t = \eta_t$ such that EG achieves a last-iterate convergence rate of $\calO(T^{-\frac{2}{3}+\eps})$ for solving \eqref{eq:minmax}.
\end{informalthm}

\begin{figure}[bt]
    \centering
    \includegraphics[width=0.86\linewidth]{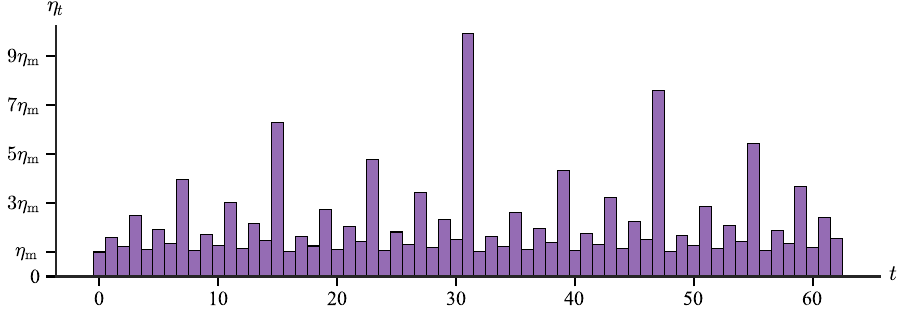}
    \caption{The discretized power-law stepsize schedule $\{\eta_0, \eta_1, \eta_2, \dots\}$ to achieve $\calO(T^{-0.66})$ in \cref{thm:main-eg}. Values are shown as multiples of a base stepsize $\etam$. Only the first 63 terms are shown.}
    \label{fig:vdc-schedule}
\end{figure}

The stepsize sequence has a base-2 structure and follows approximately a power-law distribution. A prefix of this sequence is visualized in \cref{fig:vdc-schedule}.
Moreover, we show that our analysis is nearly tight:
\begin{informalthm}{thm:lb-eg}
    For any $\eps > 0$, any stepsize schedule $\{\gamma_t, \eta_t\}_{t\in \dN}$ with $\gamma_t = \eta_t$, and any sufficiently large $T$, there exists a bilinear min-max optimization problem in which EG suffers a last-iterate convergence rate of $\Omega(T^{-\frac{2}{3}-\eps})$.
\end{informalthm}

Importantly, this lower bound does not preclude faster convergence when extrapolation and update steps are allowed to use different stepsizes.
We extend our analysis to this more flexible setting and show that differentiating the two stepsizes further accelerates convergence, yielding a near-optimal rate among first-order methods. Previously, such a rate was known only for variants of EG augmented with additional mechanisms such as anchoring.
\begin{informalthm}{thm:main-dseg}
    For any $\eps > 0$, there exists a deterministic stepsize schedule $\{\gamma_t, \eta_t\}$ that allows $\gamma_t$ and $\eta_t$ to differ such that EG achieves a last-iterate convergence rate of $\calO(T^{-1+\eps})$.
\end{informalthm}

We performed numerical experiments to verify our convergence rates; as \cref{fig:gradient-norm-compare} shows, our schedules closely follow the predicted rate. We also note that our proposed analysis and stepsize schedule are applicable to other methods, for example Optimistic Gradient (OG), another popular method for min-max optimization.
We believe that our work also sheds light on how to achieve acceleration via stepsize tuning alone for more general optimization problems.

\begin{figure}[tb]
    \centering
    \includegraphics[width=\linewidth]{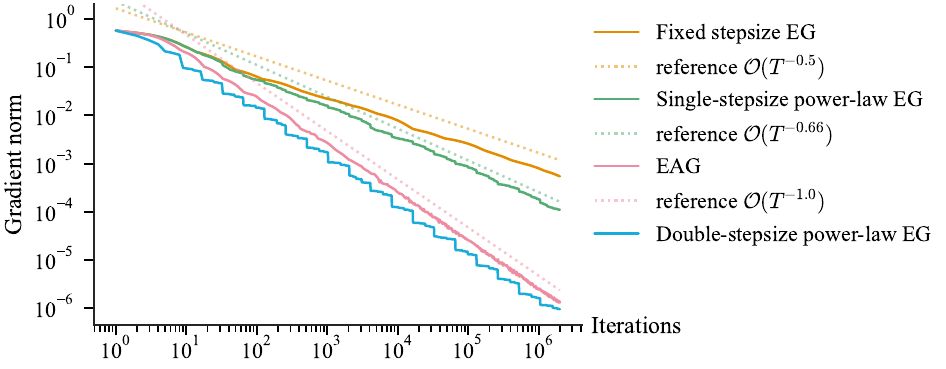}
    \caption{Last-iterate convergence of EG with a fixed stepsize, single-stepsize EG with power-law stepsize with $\beta=1/0.66$, the EAG algorithm of \cite{yoon2021accelerated}, and double-stepsize EG with power-law with $\beta=1/0.99$.
    X-axis: iteration number $T$; Y-axis: worst-case gradient norm $\GN(z_T)$ across random instances. Both axes are displayed in log scale.
    The three dashed lines are the reference slopes predicted by theory.
    }
    \label{fig:gradient-norm-compare}
\end{figure}

\paragraph{Analysis Overview}
We first consider the case where the extrapolation and update steps share the same stepsize, i.e., $\eta_t = \gamma_t$. 
Our analysis starts by reducing the problem of optimizing the stepsize schedule to the following  min-max optimization problem:
\begin{align}
    \minimize_{\{\eta_t\}_{t\in [T]}}\quad \max_{a\in [0,L]} \cbrm[\bigg]{
        a \prod_{t=0}^{T-1} \sqrt{1-\eta_t^2 a^2+\eta_t^4 a^4}
    }. %
    \label{eq:intro eg-minmax}
\end{align}
Here $L$ is the smoothness constant of the $\ell$ function in \eqref{eq:minmax}. We show that the value of \eqref{eq:intro eg-minmax} tightly characterizes the best last-iterate convergence rate achievable by EG for any fixed schedule (\Cref{lem:opt}).

Directly optimizing \eqref{eq:intro eg-minmax} is difficult. We instead replace the empirical distribution of the stepsizes with a probability distribution and analyze the convergence in geometric-mean. For fixed stepsizes, the instances with small singular values are the hard case, which requires the use of heavy-tailed distributions that occasionally take large steps. Similar behavior has appeared in accelerated schedules of Gradient Descent in convex optimization~\citep{altschuler2025accelerationI,altschuler2024acceleration/preprint}, both of which we observe to be approximately power-law distributions. We show that an appropriate power-law distribution achieves an $\mathcal{O}(T^{-2/3+\varepsilon})$ geometric-mean rate.

To obtain a deterministic anytime guarantee, we discretize its quantiles using the van der Corput sequence (\Cref{fig:vdc-schedule}), which removes the requirement of knowing the time horizon. A wavelet-based analysis controls the discretization error uniformly over all $T$, yielding \Cref{thm:main-eg}.

On the other hand, we prove a nearly matching lower bound via weak duality. 
Specifically, we construct a dual distribution over $a \in [0,L]$, also based on a power law, and show that for any stepsize schedule, when $a$ is sampled from this distribution, the geometric mean of the objective in \eqref{eq:intro eg-minmax} is at least $\Omega(T^{-2/3-\varepsilon})$ for any $\varepsilon > 0$. 
Consequently, for every stepsize schedule, there exists a choice of $a$ that forces the convergence rate to be no better than $\Omega(T^{-2/3-\varepsilon})$ (\Cref{thm:lb-eg}). 
This shows that the power-law stepsize schedule in \Cref{thm:main-eg} is nearly tight.

Our analysis for double-stepsize EG follows the same general logic, with one additional ingredient. 
We simplify the stepsize optimization problem by imposing a fixed ratio between the extrapolation and update stepsizes, $\gamma_t = \eta_t/\rho$. 
This reduction allows us to reduce to another optimization problem similar to \eqref{eq:intro eg-minmax}, and again derive power-law stepsize schedules, now achieving a near-optimal $O(T^{-1+\varepsilon})$ last-iterate convergence rate for any $\varepsilon > 0$ (\Cref{thm:main-dseg}). 
This nearly matches an existing $\Omega(T^{-1})$ lower bound for all gradient-based algorithms.

A preliminary analysis is also derived for the Optimistic Gradient (OG) algorithm in \cref{ah:og}, showing that OG achieves a randomized $\calO(T^{-0.577})$ convergence rate using power-law stepsizes. The result suggests that our method generalizes beyond the EG algorithm and serves as a broader acceleration method for min-max optimization.

\subsection{Related Works}\label{sec:related works}
There is a vast literature on solving min-max optimization problems. We review closely related work on EG and related algorithms and refer readers to the textbooks~\citep{facchinei2003finite, Ryu22Large} for a comprehensive review.

\paragraph{Extragradient} The Extragradient (EG)~\citep{korpelevich1976extragradient}, and the related Optimistic Gradient (OG) method~\citep{popov1980modification}, are classic methods for solving min-max optimization. The OG algorithm is also popular for online learning and game theory~\citep{chiang2012online, rakhlin2013optimization}. The average-iterate convergence of EG and OG has been studied extensively~\citep{nemirovski2004prox, nesterov2007dual, daskalakis2018training, mokhtari2020unified, mokhtari2020convergence}. More recently, there has been extensive attention on last-iterate convergence from the machine learning and game theory community~\citep{daskalakis2018training,daskalakis2019last,lei2021last}, culminating in the tight $\calO(T^{-\frac{1}{2}})$ rate for EG~\citep{gorbunov2022extragradient, cai2022finite} and OG~\citep{gorbunov2022last, cai2022finite}, matching lower bounds in~\citep{golowich2020last,golowich2020tight} for constant-stepsize algorithms. We also remark that EG and OG also have \emph{instance-dependent} linear last-iterate convergence for bilinear problems \eqref{eq:minmax}~\citep{liang2019interaction, wei2021last}.

Most of these works focus on the common case where the extrapolation and update steps of EG share the same stepsize, and the case with different stepsizes has been studied under the name double-stepsize EG~\citep{hsieh2020explore} and EG+~\citep{diakonikolas2021efficient}. The same idea has been applied to OG in \citep{hsieh2021adaptive}. We also note that although there is a line of work applying time-varying stepsize schedules to EG~\citep{rakhlin2013optimization, bach2019universal, antonakopoulos2019adaptive, antonakopoulos2021adaptive, hsieh2021adaptive, stonyakin2022generalized}, these works focus mostly on using adaptive and monotonic stepsizes for parameter-agnostic convergence and do not give accelerated convergence rates. To the best of our knowledge, accelerating the last-iterate convergence rate of EG or OG via stepsize optimization remains unexplored.

\paragraph{Accelerated Algorithms for Min-Max Optimization} The $\calO(T^{-\frac{1}{2}})$ rate is tight for constant-stepsize EG, but not tight for all first-order methods. Using Halpern iteration and anchoring~\citep{halpern1967fixed, lieder2021convergence}, \citet{diakonikolas2020halpern} obtain an accelerated $\Tilde{\calO}(T^{-1})$ rate, which was improved to the optimal $\calO(T^{-1})$ last-iterate convergence rate by~\citet{yoon2021accelerated} using the Extra Anchored Gradient (EAG) algorithm in the unconstrained setting. Following~\citep{yoon2021accelerated}, a line of work presents accelerated algorithms for certain nonconvex-nonconcave problems~\citep{lee2021fast}, generalizations in the constrained setting~\citep{kovalev2022first,cai2024accelerated}, and single-call variants~\citep{cai2023accelerated, cai2023doubly, sedlmayer2023fast}. On the other hand, \citet{tran2022connection} connects anchoring to Nesterov's momentum, while \citet{yoon2024optimal} propose different accelerated algorithms for anchoring-based algorithms through H-duality~\citep{kim2023time}.
These
accelerated algorithms all use additional mechanisms, while our work focuses on acceleration based on stepsize optimization alone.

\paragraph{Acceleration by Dynamic Stepsizes in Smooth Convex Optimization} %
The early work of~\citet{young1953richardson} accelerates GD from $\calO(T^{-1})$ to $\calO(T^{-2})$ for convex quadratics using Chebyshev stepsizes. A recent work of~\citet{altschuler2024acceleration/preprint} uses randomized Arcsine stepsizes to fully accelerate separable convex functions. For general smooth convex optimization, \citet{grimmer2024provably} present constant acceleration of GD using frequent long stepsizes, and \citet{altschuler2025accelerationI, altschuler2025acceleration} present the state-of-the-art rate $\calO(T^{-1.2716})$ using the Silver stepsizes, slower than the optimal $\calO(T^{-2})$ rate. These results have been generalized to acceleration in terms of gradient norm~\citep{zhang2024accelerated/preprint, grimmer2025accelerated} and the proximal gradient descent algorithms~\citep{bok2025accelerating}. It is worth noting that all these stepsize schedules require the knowledge of the horizon and thus do not provide an anytime convergence rate. This problem~\citep{kornowski2024open} is partially solved by~\citet{zhang2025anytime} with a suboptimal $\calO(T^{-1.119})$ anytime rate. 
In contrast, our power-law stepsize schedule for EG achieves anytime and near-optimal convergence rate. \citet{altschuler2024acceleration/preprint} propose the randomized Arcsine stepsize schedule that fully accelerates \emph{separable} smooth strongly-convex functions (see also \Cref{h2:result-eg} for a detailed discussion). In contrast, we can achieve near-optimal acceleration using deterministic stepsize schedules.

\paragraph{Stepsize-only acceleration for GDA} %
Most closely related to our setting, \citet{shugart2025negative/preprint} show that GDA can be made to converge using their Slingshot stepsize schedules,
and obtain the optimal $\calO(T^{-1})$ last-iterate gradient-norm rate for several settings, including smooth bilinear problems.
Their schedules are based on Chebyshev, time-varying, asymmetric, and most notably, periodically negative.
Our results concern a different base algorithm and stepsize regime:
we accelerate the EG family of algorithms using positive stepsizes symmetrically for the $x$ and $y$ components. We additionally obtain an anytime, instance-oblivious sequence from a positive power-law distribution using QMC derandomization.
Their work offer complementary results on stepsize-only acceleration in min-max optimization.

\section{Preliminary}\label{h1:prelim}

\paragraph{Notation} Throughout this paper, we use $\dN$ to denote the set of non-negative integers $\{0,1,2,\dots\}$. We define $[N]=\{0, 1, 2, \dots, N-1\}$. We use $\norm{\cdot}=\norm{\cdot}_2$ to denote the $L_2$ norm for vectors, $\norm{v}=\sqrt{\sum_i v_i^2}$, and the induced operator norm for matrices, $\norm{M}=\sup_{\norm{v}=1} \norm{Mv}$. The notation $(x_1,\dots,x_n)$ represents an $n\times 1$ column vector.

We use the notation $\Pareto(\xm,\alpha)$ for the Pareto distribution with scale parameter $\xm$ and shape parameter $\alpha$; if $X\sim \Pareto(\xm, \alpha)$, then $\Pr[X\geq x]=\rbrm[\big]{\frac{\xm}{x}}^\alpha$ for all $x\geq \xm$.
If $\calD$ is a distribution over $\dR$, we define $Q_\calD : [0, 1)\to \dR$ as the quantile function of $\calD$, or formally $Q_\calD(k)=F_\calD^{-1}(k)$, where $F_\calD$ is the CDF of $\calD$ and $F^{-1}_\calD$ is its generalized inverse. For a distribution $\calD$, the notation $X_1, \dots, X_n \sim \calD$ denotes that each random variable $X_i$ has marginal distribution $\calD$. This notation makes no assumptions regarding the joint distribution or independence. The \emph{geometric mean} of a random variable $X$ supported on $(0, +\infty)$ is defined as $\exp(\EE[\ln X])$.

\paragraph{Setup}%

Suppose we have a function $\ell : \calX\times\calY\to \dR$ on two Euclidean spaces $\calX$ and $\calY$ with dimensions $n$ and $m$. A \emph{min-max optimization} (or \emph{minimax game}) is an optimization problem specified by $(\calX,\calY,\ell)$, of the form
\begin{align}\label{eq:minmax-problem}
    \minimize_{x \in \calX}\quad\maximize_{y \in \calY} \quad\ell(x,y).
\end{align}
The solutions of the problem are defined as \emph{saddle points} of $\ell$. A pair $(x^\star, y^\star)$ is a saddle point, or \emph{Nash equilibrium} of $\ell$, if 
$
    \forall x\in \calX, \forall y\in \calY, \ell(x^\star,y)\leq \ell(x^\star, y^\star)\leq \ell(x, y^\star).
$
For notational simplicity, we write $z=(x,y)\in \calZ=\calX\times \calY$.

When $\ell$ is differentiable, define the gradient operator
$G_\ell(x,y)=(\nabla_x \ell(x,y),-\nabla_y \ell(x,y))$.
We omit the subscript $\ell$ if the function is clear from context.
\draftcolor{cyan}{In the unconstrained setting $\+X= \-R^n$, $\+Y = \-R^m$,} the first-order condition $G_\ell(z)=0$ is necessary for $z$ to be a saddle point.

A function $\ell : \calX\times \calY\to \dR$ on two real vector spaces is biaffine if $\ell(\cdot, y)$ is affine for all $y\in \calY$, and $\ell(x,\cdot)$ is affine for all $x\in \calX$. If $\calX=\dR^n, \calY=\dR^m$, such a function can always be written as
\begin{equation*}
\ell(x,y)=x^\trans A y + p^\trans y + x^\trans q + c
,
\end{equation*}
where $A\in \dR^{n\times m}, p\in \dR^m,q\in \dR^n,$ and $c\in \dR$ are parameters. We additionally say that $\ell$ is bilinear if $p=0$, $q=0$, and $c=0$. The gradient operator for a biaffine function simplifies to $G_\ell(x,y)=(Ay+q,-A^\trans x-p)$.

While biaffinity/bilinearity is a strong condition, it is the natural starting point for analyzing acceleration in games, analogous to how quadratic functions are the standard starting point for analyzing gradient descent in convex optimization (where Arcsine or Silver stepsizes were first developed).

We say the function $\ell$ is $L$-smooth if the gradient operator $G_\ell$ is $L$-Lipschitz, or formally,
$
    \forall z_1,z_2\in \calX\times \calY, \norm{G_\ell(z_1)-G_\ell(z_2)}\leq L\norm{z_1-z_2}.
$
For biaffine functions, it is equivalent to $\norm{A}\leq L$.

It is often not possible to solve for $z^\star$ exactly, and we aim to find approximate solutions. We use the gradient-norm function $\GN_\ell(z)=\norm{G_\ell(z)}$ to evaluate how close a pair $z$ is to optimum. In the biaffine setting, this function is equal to zero only when $z$ is a saddle point, and is upper-bounded by $L\norm{z-z^\star}$. Similarly, we may omit the subscript $\ell$ in $\GN_\ell$ when it is clear from context.

The standard evaluation of suboptimality in min-max optimization, the duality gap, is often $\infty$ in unconstrained problems.
We use $\GN_\ell(\cdot)$ as a proxy of the duality gap. Note that, for biaffine $\ell$, $\GN_\ell(z)$ is equal to the duality gap when the comparison set is the unit ball around $z^\star$:
\begin{align*}
    \GN_\ell(z)=\max_{z'=(x',y') : \norm{z'-z^\star}\leq 1}\cbr{\ell(x,y')
     -\ell(x',y)}.
\end{align*}

\paragraph{Extragradient}

The Extragradient (EG) algorithm~\citep{korpelevich1976extragradient} for solving a min-max optimization problem $(\calX,\calY,\ell)$ is defined as in \eqref{eq:EG}.
If we write $z_t=(x_t,y_t)$, EG can be concisely represented with the gradient operator $G_\ell$ as
\begin{align}
    z_{t+\frac12}&=z_t-\gamma_t G_\ell(z_t), & %
    z_{t+1}&=z_t-\eta_t G_\ell(z_{t+\frac 12}).\label{eq:eg-def}
\end{align}
We assume that $z_0$ is initialized at most $R$ away from $z^\star$, or formally $\norm{z_0-z^\star}\leq R$.

\paragraph{Van der Corput sequence}

The van der Corput sequence is a one-dimensional low-discrepancy sequence over the unit interval $[0, 1)$, widely used in numerical integration and quasi-Monte Carlo methods~\citep{niederreiter1992random}.
Specifically, 
for any integer $b\geq 2$, define the base-$b$ radical inverse function as
$
    \VDC_b(n) = \frac{d_0(n)}{b} + \frac{d_1(n)}{b^2} + \frac{d_2(n)}{b^3}+\cdots,
$
for all integer $n\in \dN$ with base-$b$ expansion 
    $n = d_0(n) + d_1(n)b + d_2(n) b^2 + \cdots$,
where $d_j(n) \in \{0, 1, \dots, b-1\}$. Intuitively, $\VDC_b(n)$ mirrors the base-$b$ digit expansion of $n$ to the right of the radix point.
For example, $\VDC_{10}(816)=0.618$, and $\VDC_2(10)=(0.0101)_2=\frac{5}{16}$ because $10=(1010)_2$.

The van der Corput sequence in base $b$ is defined as $\{\vdc_n=\VDC_b(n)\}_{n\geq 0}$. Throughout this paper, we exclusively focus on the base-2 van der Corput sequence and denote it simply as $\vdc_n = \VDC_2(n)$. The first few elements corresponding to indices $n \in \{0, 1, 2, \dots, 5\}$ are $0, \frac{1}{2}, \frac{1}{4}, \frac{3}{4}, \frac{1}{8}, \frac{5}{8}$, respectively.

\section{Main results: acceleration by power-law stepsizes}\label{h1:result}

In this section, we provide formal statements and informal overviews of our main results.
To understand the last-iterate convergence of EG, we track the growth of the gradient norm throughout the trajectory. A decomposition argument proves a tight inequality that completely captures the convergence of EG in a simple formula, as shown below.

\begin{restatable}{lemma}{lemOpt}\label{lem:opt}
    Let $(\dR^n, \dR^m, \ell)$ be an $L$-smooth biaffine min-max optimization problem with a saddle point $z^\star=(x^\star, y^\star)$. The trajectory of EG in this problem starting from $z_0$ satisfies the following:
    \begin{align}
        \GN(z_T)
        \leq
            \norm{z_0-z^\star}
            \max_{a\in [0,L]} \cbrm[\bigg]{
                a \prod_{t=0}^{T-1} \sqrt{1+\eta_t(\eta_t-2\gamma_t) a^2+\eta_t^2\gamma_t^2 a^4}
            }.
        \label{eq:lem-opt}
    \end{align}
    Furthermore, this inequality is tight: for any integer $T$ and any sequence of $\{(\gamma_t, \eta_t)\}_{t\in [T]}$,
    if $a^\star$ is a maximizer in the RHS of \eqref{eq:lem-opt}, then
    the inequality \eqref{eq:lem-opt} holds with equality on the problem $(\dR, \dR, \ell(x,y)=a^\star xy)$.
\end{restatable}

The proof of this lemma uses affine translation and SVD, and is deferred to \cref{ah:proof-decompose}.
This lemma abstracts away all the details of the algorithm, including the trajectory and the geometry of the function $\ell$, and reduces the problem of stepsize tuning into the problem of optimizing the right-hand side of \eqref{eq:lem-opt}.
In the rest of this section, we will utilize this tight reduction and present our results in both directions.

\subsection{Analysis of single-stepsize Extragradient}\label{h2:result-eg}

In this section, we focus on optimizing the RHS of \eqref{eq:lem-opt} with the traditional, single-stepsize setting, $\gamma_t=\eta_t$. We are then solving the following optimization problem:
\begin{align}
    \minimize_{\{\eta_t\}_{t\in [T]}}\;\maximize_{a\in[0,L]}\;
    \calR_T(\{\eta_t\}_{t\in [T]}, a) \defeq
    {
        a \prod_{t=0}^{T-1} \sqrt{1-\eta_t^2 a^2+\eta_t^4 a^4}
    }. %
    \label{eq:eg-minmax}
\end{align}

If the stepsize is fixed as $\eta_t\equiv \eta$, traditional analysis of EG chooses $\eta<\frac{1}{L}$ to ensure that the contraction ratio $\sqrt{1-a^2\eta^2+a^4\eta^4}<1$ for all $\abs{a}\leq L$.
Informally, when $a$ is close to zero, the contraction ratio can be arbitrarily close to 1.
In particular, when $a=\frac{1}{\eta\sqrt T}$, we have $\calR_T(\eta,a)=a(1-a^2\eta^2+a^4\eta^4)^{T/2}>\frac{1}{\eta\sqrt T}(1-\frac 1T)^{T/2}\approx \frac{1}{\eta\sqrt{eT}}$. This shows that EG with a fixed stepsize cannot optimize beyond $\calO(T^{-1/2})$.

To optimize the last-iterate convergence further, we analyze the use of time-varying stepsizes.
We observe that the objective $\calR_T(\{\eta_t\}_{t\in [T]}, a)$ is invariant under permutations of the stepsize sequence; it depends solely on the multiset, or \emph{empirical distribution}, of the stepsizes.
Rather than solving the discrete-time optimization for finite $T$, we relax the problem by replacing the empirical distribution with a general probability distribution $\calE$ over $\dR_+$. Define $\calR_T(\calE, a)$ as the \emph{geometric mean} of $\calR_T(\{\eta_t\}_{t\in [T]}, a)$ if each stepsize is sampled from $\calE$, or formally, $\calR_T(\calE, a)$ is the value that satisfies
\begin{align*}
    \ln \calR_T(\calE, a)
    & =
    \EE_{\eta_0,\dots,\eta_{T-1}\sim \calE}
    \sbrm[\Bigg]{
    \ln\rbrm[\bigg]{
        a \prod_{t=0}^{T-1} \sqrt{1-\eta_t^2 a^2+\eta_t^4 a^4}
    }} \\
    & =
    {
        \ln a + \frac T2\EE_{\eta\sim \calE}\sbrm[\big]{\ln(1-\eta^2a^2+\eta^4a^4)}
    }.\yestag\label{eq:calL-def}
\end{align*}
If $\calE=\frac1T\sum_{t=0}^{T-1}\delta_{\eta_t}$ is the empirical distribution of $\{\eta_t\}_{t\in [T]}$, the RHS of \eqref{eq:calL-def} coincides with $\ln \calR_T(\{\eta_t\}_{t\in [T]},a)$. Therefore, this formulation is indeed a generalization.
The continuous formulation allows us to leverage analytical tools to design distributions and analyze the contraction, before eventually discretizing them back into practical deterministic sequences.

We employ an ansatz that $\calE$ is a power-law distribution, or more specifically, $\calE=\Pareto(\etam/L, \beta)$.
This educated guess is based on the intuition that $\calE$ must control the worst-case choice of $a$, which, as we have seen in the fixed-stepsize setting, can be arbitrarily small;
therefore, the best-response $\argmin_{\eta} \ln(1-\eta^2a^2+\eta^4a^4)=\frac{1}{\sqrt2 a}$ for a given $a$ can be arbitrarily large, requiring a scale-free distribution.
The choice of a power-law distribution is also inspired by recent advances in accelerating gradient descent (GD) for convex optimization, where stepsize hedging has been shown to similarly improve convergence rates. More specifically, we draw inspiration from two results:
\begin{itemize}[leftmargin=*]
    \item {\bfseries Silver Stepsize Schedule}: \citet{altschuler2025acceleration} propose a discrete sequence of stepsizes that accelerates GD to achieve a convergence rate of $\calO(T^{-\log_2\rho})=\calO(T^{-1.27})$, where $\rho=1+\sqrt2$ is the silver ratio.
    In comparison, fixed-stepsize GD achieves $\calO(T^{-1})$. Stepsizes of scale $\rho^k$ are utilized with a frequency of $2^{-k}$; the tail of this empirical distribution satisfies $\Pr[\eta\geq x]=\Theta(x^{-\log_\rho 2})$ for large $x$, implicitly forming a discrete power-law distribution.
    \item {\bfseries Arcsine Stepsize Distribution}: \cite{altschuler2024acceleration/preprint} show that for $m$-strongly-convex, $M$-smooth separable functions, the optimal convergence rate can be achieved by sampling $\eta^{-1}\sim \Arcsin(m,M)$. If we na\"ively extrapolate their result to $m\to 0$, the distribution satisfies $\Pr[\eta\geq x]=\frac2\pi\arcsin\sqrt{(Mx)^{-1}}=\Theta(x^{-1/2})$ for large $x$, again approximating a power-law.
\end{itemize}

It turns out that using a Pareto distribution $\calE$ indeed improves the geometric-mean convergence rate to $\calO(T^{-2/3+\eps})$, as shown in the following lemma:

\begin{restatable}{lemma}{lemEgCont}\label{lem:eg-cont}
    For any $\beta\in (\frac32,2)$, there exists $\etam,C\in \dR_+$, such that the distribution $\calE=\Pareto(\etam/L, \beta)$ satisfies
    $
        \calR_T(\calE, a)\leq CLT^{-1/\beta}
    $ for all $a\in (0, L]$.
    As a corollary, for any $L$-smooth biaffine min-max optimization $(\dR^n, \dR^m, \ell)$ with a saddle point $z^\star$, %
    if EG is initiated at $z_0$ with $\norm{z_0-z^\star}\leq R$ and samples its stepsizes from $\calE$, we have for all $T$,
    $
        \EE_{\eta_0,\dots,\eta_{T-1}\sim \calE}[\ln \GN(z_T)]\leq \ln\sbrm[\big]{CLRT^{-1/\beta}}.
    $
\end{restatable}

The full proof of this lemma is deferred to \cref{ah:eg-cont}. 
As a high-level overview,
the expectation in the RHS of \eqref{eq:calL-def} under our ansatz of the Pareto distribution becomes
\begin{align*}
    \EE_{\eta\sim \calE}\sbrm[\big]{\ln(1-\eta^2a^2+\eta^4a^4)}
    & =
    \int_{\eta=\etam/L}^{\infty}\ln(1-\eta^2a^2+\eta^4a^4)\cdot(-1)\dd\rbrm[\Big]{\frac{\etam}{L\eta}}^\beta \\
    & = \frac{\beta}{2}\rbrm[\Big]{ \frac{\etam a}{L} }^{\beta} \int_{(\etam a/L)^2}^{\infty} \ln(1-t+t^2)t^{-\beta/2-1}\dd{t},\yestag\label{eq:pareto-contraction-int}
\end{align*}
where we used a substitution $t=\eta^2a^2$.
The value of \eqref{eq:pareto-contraction-int} represents the expected log contraction of EG in one iteration, and thus should ideally be negative.
A careful analysis of this Mellin-type integration shows that it can be negative only when $\beta>\frac32$, and $\etam$ is a small constant that depends on $\beta$. In this case, the entire term in \eqref{eq:pareto-contraction-int} grows at a rate of $-\Theta((a/L)^\beta)$, where $\Theta(\cdot)$ hides constants that depend on $\beta$; substituting it into \eqref{eq:calL-def},  we see that the worst-case $a^\star=\Theta(L T^{-1/\beta})$, which implies $\calR_T(\calE, a)=\calO(LT^{-1/\beta})$. The second claim of the lemma is a direct implication of \cref{lem:opt}.

\paragraph{Derandomization}{}

The $\calO(T^{-1/\beta})$ rate of \Cref{lem:eg-cont} only holds in geometric mean.
To obtain a deterministic bound, \draftcolor{cyan}{a natural idea is} to use the quantiles of the distribution as stepsizes $\{\eta_t\}_{t\in [T]}=\{Q_\calE\rbrm[\big]{\frac tT}\}_{t\in [T]}$. \draftcolor{cyan}{However, this approach requires knowing the time horizon $T$ and does not give an any-time convergence guarantee.}
In this work, we obtain a stronger \draftcolor{cyan}{any-time convergence} result that holds deterministically for the last iterate uniformly for all $T$.
Instead of using quantiles indexed by the uniform grid $\{\frac tT\}_{t\in [T]}$,
we utilize the quasirandom van der Corput sequence --- informally speaking, a quasirandom sequence that approaches the uniform grid for all $T$.
The technical result of quantile discretization is summarized in the following lemma.

\begin{restatable}{lemma}{lemDiscretization}
\label{lem:discretization}
    For any $\beta\in(\frac 32, \frac 85)$, let $\calE$ be the distribution specified in \cref{lem:eg-cont}.
    For $k\in [0, 1)$, we use $Q_\calE(k)$ to denote the $k$-th quantile of $\calE$, that is, $Q_\calE(k)=(1-k)^{-1/\beta}\etam/L$.
    There exists a universal constant $C'$ such that, if $\eta_t=Q_\calE(\vdc_t)$, where $\{\vdc_t\}_{t\in \dN}$ is the base-2 van der Corput sequence, then for any $T>0$ and any $a\in [0,L]$, 
    \begin{align*}
        \frac1T
        \sum_{t=0}^{T-1} \ln(1-\eta_t^2 a^2 + \eta_t^4 a^4)
        \leq
        \EE_{\eta\sim \calE}[\ln(1-\eta^2a^2+\eta^4 a^4)]
        + \frac{C'}{T}.
    \end{align*}
    Consequently, 
    $
        \calR_T(\{\eta_t\}_{t\in [T]}, a)
        \leq
        \exp(C'/2)\calR_T(\calE, a).
    $
\end{restatable}

At a high level, we are approximating the integral in \eqref{eq:pareto-contraction-int} with a finite sum using a discrete sequence $\{\vdc_t\}_{t\in [T]}$, which is exactly quasi-Monte Carlo (QMC).
Except at the singular point $\eta\to \infty$, the function satisfies a certain Besov-type smoothness condition, which implies that QMC converges at a rate of $\calO(1/T)$.
The singular point is handled using a monotonicity condition produced by the bias of the van der Corput sequence.
The proof is available in \cref{ah:eg-cont}, as a special case of \cref{thm:pareto-qmc-discretize}.

We are finally ready to present our first main theorem. It is proved by combining \cref{lem:discretization} with \cref{lem:opt,lem:eg-cont}.

\begin{formalrestatethm}{thm:main-eg}\label{thm:main-eg}
    For any $\eps\in (0, \frac{1}{24})$, there exists $\etam>0$ and $\beta=\rbrm[\big]{\frac23-\eps}^{-1}$ such that the following holds:
    For any $L$-smooth biaffine min-max optimization $(\dR^n, \dR^m, \ell)$ with a saddle point $z^\star = (x^\star, y^\star)$, if EG is initiated at $z_0=(x_0,y_0)$ with $\norm{z_0-z^\star}\leq R$ and uses stepsizes $\gamma_t = \eta_t = (1-\vdc_t)^{-1/\beta}\etam/L$, the $\vdc_t$-quantile of $\Pareto(\etam/L, \beta)$, then for all $T$,
    \begin{align*}
        \GN(z_T)\leq \calO(LRT^{-\frac 23+\eps}).
    \end{align*}
    The $\calO(\cdot)$ notation hides a constant that depends only on $\eps$.
\end{formalrestatethm}

\paragraph{Lower bound of single-stepsize Extragradient} %

Having established the $\calO(T^{-2/3+\eps})$ upper bound achieved by our power-law schedule, it is natural to ask: is the power-law schedule optimal? Can we further improve the rate by using a different stepsize schedule? The following theorem shows that the answer is negative for the standard, single-stepsize EG algorithm.

\begin{formalrestatethm}{thm:lb-eg}\label{thm:lb-eg}
    For any $\eps>0$, there exists a constant $C>0$ such that the following holds. For any sufficiently large $T$ and any distribution $\calE$ supported on $\dR_+$ of stepsizes, there exists a $1$-smooth biaffine min-max optimization $(\dR, \dR, \ell)$ with a saddle point $z^\star=(0,0)$, such that if EG is initiated at $z_0$ with $\norm{z_0}=R$ and samples its stepsizes from $\gamma_i=\eta_i\sim \calE$ for $i=0,\dots,T-1$, then
    \begin{align*}
        \EE[\ln \GN(z_T)]\geq \ln\sbrm[\Big]{CRT^{-\frac 23-\eps}}.
    \end{align*}
    Further, for any sufficiently large $T$ and any length-$T$ sequence of stepsizes $\{\gamma_t=\eta_t\}_{t\in [T]}$, there exists a $1$-smooth biaffine min-max optimization problem in which EG suffers
    \begin{align*}
        \GN(z_T)\geq CRT^{-\frac 23-\eps}.
    \end{align*}
\end{formalrestatethm}

The proof to \cref{thm:lb-eg} is deferred to \cref{ah:lower-bound} and relies on viewing the stepsize scheduling problem in \eqref{eq:eg-minmax} as a zero-sum meta-game. In this game, Algorithm chooses a stepsize schedule, and Nature adversarially chooses a parameter $a\in [0, L]$ that parameterizes the instance.
We find that, if Nature samples $a$ from a power-law-type distribution, the geometric mean of the gradient norm cannot be lower than $\Omega(T^{-1/\beta})$ for any $\beta<\frac 32$.
Consequently, no deterministic or randomized stepsize schedule can guarantee a rate polynomially faster than $\calO(T^{-\frac23})$.

\subsection{Analysis of double-stepsize Extragradient}\label{h2:result-dseg}

We have seen in the previous parts that power-law stepsizes in the standard single-stepsize EG algorithm achieve a near-tight convergence $\calO(T^{-2/3+\eps})$. We now show that this barrier can be bypassed by differentiating the two stepsizes of EG. In the spirit of accelerating purely through stepsize scheduling, we allow $\gamma_t \neq \eta_t$ without changing the updates in EG.

Analogous to \eqref{eq:eg-minmax}, we have a subtly different optimization problem.
By applying the same Pareto ansatz as in \cref{h2:result-eg}, it turns out that we can indeed break the $\calO(T^{-2/3})$ barrier.
We consider a schedule where we sample $\eta_t$ from a Pareto distribution, and set a constant ratio between the two stepsizes, $\gamma_t=\eta_t/\rho$. Our previous results in \cref{h2:result-eg} cover the case when $\rho=1$. 

It turns out that when $\rho\to 0^+$, we can obtain even better bounds by using a smaller $\beta>1$. The analysis follows our previous methodology closely. We defer the proof of the following theorem to \cref{ah:dseg} to avoid repetition.

\begin{formalrestatethm}{thm:main-dseg}\label{thm:main-dseg}
    For any $\eps\in (0, \frac{1}{8})$, there exist constants $\beta=(1-\eps)^{-1}, \rho, \etam, C>0$, such that for any $L$-smooth biaffine min-max optimization instance $(\dR^n, \dR^m, \ell)$ with a saddle point $z^\star$, if EG is initiated at $z_0$ with $\norm{z_0-z^\star}\leq R$, and uses stepsizes $\eta_t$ sampled from $\calE=\Pareto(\etam/L, \beta)$ and $\gamma_t=\eta_t/\rho$, the following holds for all $T$:
    \begin{align*}
        \EE_{\eta_0,\dots,\eta_{T-1}\sim \calE}[\ln \GN(z_T)]\leq \ln\sbrm[\Big]{CLRT^{-1+\eps}}.
    \end{align*}
    Further, if $\{\vdc_i\}_{i\in \dN}$ is the base-2 van der Corput sequence, $\eta_t =  (1-\vdc_t)^{-1/\beta} \etam/L$ is the $\vdc_t$-th quantile of $\calE$, and again $\gamma_t=\eta_t/\rho$, then we have for all $T$,
    \begin{align*}
        \GN(z_T)\leq \calO(LRT^{-1+\eps}),
    \end{align*}
    where the $\calO(\cdot)$ notation hides a constant that depends only on $\eps$.
\end{formalrestatethm}

\begin{remark}
    The $\calO(T^{-1+\eps})$ rate is almost optimal. \citet[Theorem 4]{yoon2021accelerated} provides a lower bound of $\GN(z_T)\geq \frac{LR}{2\floor{T/2}+1}=\Omega(LRT^{-1})$ for all deterministic gradient-based algorithms. They also argue informally that the same lower bound applies to all randomized algorithms using techniques of \citep{woodworth2016tight}.
\end{remark}

\section{Experiments}\label{h1:experiment}

In order to empirically illustrate our main results, we conduct numerical experiments with our stepsize schedules. We randomly generate instances of biaffine functions 
and simulate EG on them with our stepsize schedules. To optimize for constant factors, we use the truncated stepsize distributions specified in \cref{lem:eg-mixture} and \cref{lem:dseg-mixture}, and apply the van der Corput sequence to obtain a deterministic stepsize sequence. Details of the experimental setup and further numerical results are deferred to \cref{ah:experiment-more}.

The numerical results are shown in \cref{fig:gradient-norm-compare}. We see that, over the sampled instances and plotted horizons, the slope of $\GN(z_t)$ of our single-stepsize schedule with $\beta=100/66$ follows a $\calO(T^{-0.66})$ trend line, and our double-stepsize schedule with $\beta=100/99$ behaves very close to $\calO(T^{-1})$.
We include the EAG algorithm of \citet{yoon2021accelerated} for comparison, which converges at the optimal rate of $\calO(T^{-1})$.

\section{Conclusion and discussion}\label{h1:conclusion}

In this paper, we proposed stepsize schedules that accelerate the last-iterate convergence of Extragradient. We show that, in the case of \emph{unconstrained biaffine} min-max optimization problems, carefully designed nonmonotone stepsize schedules accelerate EG to  nearly optimal rate without adding anchoring, momentum, or other algorithmic mechanisms. We introduced an unexplored distribution family, namely power-law distributions, to the literature of stepsize tuning. We also developed a novel technique of using the van der Corput sequence to discretize a distribution to obtain a deterministic, non-adaptive stepsize sequence, linking the worlds of numerical integration and stepsize scheduling.

The analysis framework applies to both single-stepsize and double-stepsize EG, achieving $\order(T^{-2/3+\eps})$ and $\order(T^{-1+\eps})$ worst-case convergence rates, respectively, beating the $\order(T^{-1/2})$ rate achieved by fixed-stepsize EG. Traditional analysis of EG also focuses on instance-dependent linear convergence rates, which our algorithm also achieves; see \cref{ah:linear-rate} for further discussion.

Several questions remain open. 
The most important is whether this stepsize-only acceleration extends beyond our setting, for example to convex-concave min-max problems or general monotone variational inequalities. Our present analysis relies on the exact reduction in \cref{lem:opt}, so new ideas would be needed when the linear structure is unavailable.

Another direction is extending this stepsize tuning to other algorithms. For instance, we provide an initial theoretical analysis for using power-law schedules to accelerate the Optimistic Gradient (OG) method in \cref{ah:og}; however, our numerical experiments suggest that it empirically outperforms the theory's prediction, and developing a tighter analysis for OG remains open.

\begin{ack}
YC is supported by the NSF award CCF-2342642. WZ is supported by the NSF award CCF-2342642 and
a research fellowship from the Center for Algorithms, Data, and Market Design at Yale (CADMY).
HL and YW are supported by NSF awards IIS-1943607 and CCF-2551666.

Generative AI technology, including OpenAI Codex, Google Gemini, and Claude Code, is utilized at several stages of this work.
In developing the theoretical results, AI was used to explore the theory and initially draft the proofs for the special function integrations of \cref{lem:m-func} and \ref{lem:m-func-og}, as well as the discretization of \cref{ah:discretization}.
In the numerical experiments, AI was employed to generate utility scripts and testing code, while the main algorithms were implemented exclusively by hand.
Finally, for manuscript preparation, AI also provided writing assistance and proofreading.
\end{ack}

\bibliography{ref}

\begin{thebibliography}{65}
\providecommand{\natexlab}[1]{#1}
\providecommand{\url}[1]{\texttt{#1}}
\expandafter\ifx\csname urlstyle\endcsname\relax
  \providecommand{\doi}[1]{doi: #1}\else
  \providecommand{\doi}{doi: \begingroup \urlstyle{rm}\Url}\fi

\bibitem[Altschuler and Parrilo(2024)]{altschuler2024acceleration/preprint}
Jason~M. Altschuler and Pablo~A. Parrilo.
\newblock Acceleration by random stepsizes: {{Hedging}}, equalization, and the
  arcsine stepsize schedule.
\newblock ArXiv preprint, 2412.05790, December 2024.

\bibitem[Altschuler and
  Parrilo(2025{\natexlab{a}})]{altschuler2025acceleration}
Jason~M. Altschuler and Pablo~A. Parrilo.
\newblock Acceleration by stepsize hedging {{II}}: {{Silver}} stepsize schedule
  for smooth convex optimization.
\newblock \emph{Mathematical Programming}, 213\penalty0 (1-2):\penalty0
  1105--1118, September 2025{\natexlab{a}}.
\newblock ISSN 0025-5610, 1436-4646.
\newblock \doi{10.1007/s10107-024-02164-2}.

\bibitem[Altschuler and
  Parrilo(2025{\natexlab{b}})]{altschuler2025accelerationI}
Jason~M. Altschuler and Pablo~A. Parrilo.
\newblock Acceleration by stepsize hedging: Multi-step descent and the silver
  stepsize schedule.
\newblock \emph{Journal of the ACM}, 72\penalty0 (2):\penalty0 1--38,
  2025{\natexlab{b}}.

\bibitem[Antonakopoulos et~al.(2019)Antonakopoulos, Belmega, and
  Mertikopoulos]{antonakopoulos2019adaptive}
Kimon Antonakopoulos, Veronica Belmega, and Panayotis Mertikopoulos.
\newblock An adaptive mirror-prox method for variational inequalities with
  singular operators.
\newblock \emph{Advances in Neural Information Processing Systems}, 32, 2019.

\bibitem[Antonakopoulos et~al.(2021)Antonakopoulos, Belmega, and
  Mertikopoulos]{antonakopoulos2021adaptive}
Kimon Antonakopoulos, Veronica Belmega, and Panayotis Mertikopoulos.
\newblock Adaptive extra-gradient methods for min-max optimization and games.
\newblock In \emph{International Conference on Learning Representations}, 2021.

\bibitem[Azizian et~al.(2020)Azizian, Scieur, Mitliagkas, {Lacoste-Julien}, and
  Gidel]{azizian2020accelerating}
Wa{\"i}ss Azizian, Damien Scieur, Ioannis Mitliagkas, Simon {Lacoste-Julien},
  and Gauthier Gidel.
\newblock Accelerating smooth games by manipulating spectral shapes.
\newblock In \emph{Proceedings of the {{Twenty Third International Conference}}
  on {{Artificial Intelligence}} and {{Statistics}}}, pages 1705--1715. PMLR,
  June 2020.

\bibitem[Bach and Levy(2019)]{bach2019universal}
Francis Bach and Kfir~Y. Levy.
\newblock A universal algorithm for variational inequalities adaptive to
  smoothness and noise.
\newblock In \emph{Conference on learning theory}, pages 164--194. PMLR, 2019.

\bibitem[Bateman(1954)]{bateman1954integral}
Harry Bateman.
\newblock \emph{Tables of Integral Transforms}, volume~1.
\newblock McGraw-Hill Book Company, New York, 1954.
\newblock Compiled by the Bateman Manuscript Project and edited by Arthur
  Erd{\'e}lyi.

\bibitem[Bok and Altschuler(2025)]{bok2025accelerating}
Jinho Bok and Jason~M. Altschuler.
\newblock Accelerating proximal gradient descent via silver stepsizes.
\newblock In \emph{The Thirty Eighth Annual Conference on Learning Theory},
  pages 421--453. PMLR, 2025.

\bibitem[Cai and Zheng(2023{\natexlab{a}})]{cai2023accelerated}
Yang Cai and Weiqiang Zheng.
\newblock Accelerated single-call methods for constrained min-max optimization.
\newblock \emph{International Conference on Learning Representations (ICLR)},
  2023{\natexlab{a}}.

\bibitem[Cai and Zheng(2023{\natexlab{b}})]{cai2023doubly}
Yang Cai and Weiqiang Zheng.
\newblock Doubly optimal no-regret learning in monotone games.
\newblock In \emph{International Conference on Machine Learning}, pages
  3507--3524. PMLR, 2023{\natexlab{b}}.

\bibitem[Cai et~al.(2022)Cai, Oikonomou, and Zheng]{cai2022finite}
Yang Cai, Argyris Oikonomou, and Weiqiang Zheng.
\newblock Finite-time last-iterate convergence for learning in multi-player
  games.
\newblock In \emph{Advances in Neural Information Processing Systems
  (NeurIPS)}, 2022.

\bibitem[Cai et~al.(2024)Cai, Oikonomou, and Zheng]{cai2024accelerated}
Yang Cai, Argyris Oikonomou, and Weiqiang Zheng.
\newblock Accelerated algorithms for constrained nonconvex-nonconcave min-max
  optimization and comonotone inclusion.
\newblock In \emph{Forty-first International Conference on Machine Learning},
  2024.

\bibitem[Chiang et~al.(2012)Chiang, Yang, Lee, Mahdavi, Lu, Jin, and
  Zhu]{chiang2012online}
Chao-Kai Chiang, Tianbao Yang, Chia-Jung Lee, Mehrdad Mahdavi, Chi-Jen Lu, Rong
  Jin, and Shenghuo Zhu.
\newblock Online optimization with gradual variations.
\newblock In \emph{Conference on Learning Theory}, pages 6--1. JMLR Workshop
  and Conference Proceedings, 2012.

\bibitem[Daskalakis and Panageas(2019)]{daskalakis2019last}
Constantinos Daskalakis and Ioannis Panageas.
\newblock Last-iterate convergence: Zero-sum games and constrained min-max
  optimization.
\newblock In \emph{10th Innovations in Theoretical Computer Science Conference
  (ITCS)}, 2019.

\bibitem[Daskalakis et~al.(2018)Daskalakis, Ilyas, Syrgkanis, and
  Zeng]{daskalakis2018training}
Constantinos Daskalakis, Andrew Ilyas, Vasilis Syrgkanis, and Haoyang Zeng.
\newblock Training {GANs} with optimism.
\newblock In \emph{International Conference on Learning Representations
  (ICLR)}, 2018.

\bibitem[Diakonikolas(2020)]{diakonikolas2020halpern}
Jelena Diakonikolas.
\newblock Halpern iteration for near-optimal and parameter-free monotone
  inclusion and strong solutions to variational inequalities.
\newblock In \emph{Conference on Learning Theory (COLT)}, 2020.

\bibitem[Diakonikolas et~al.(2021)Diakonikolas, Daskalakis, and
  Jordan]{diakonikolas2021efficient}
Jelena Diakonikolas, Constantinos Daskalakis, and Michael Jordan.
\newblock Efficient methods for structured nonconvex-nonconcave min-max
  optimization.
\newblock \emph{International Conference on Artificial Intelligence and
  Statistics}, 2021.

\bibitem[Faber(1910)]{faber1910uber}
Georg Faber.
\newblock {{\"Uber}} die {{Orthogonalfunktionen}} des {{Herrn Haar}} [{{On
  Mr}}. {{Haar}}'s orthogonal functions].
\newblock \emph{Jahresbericht der Deutschen Mathematiker-Vereinigung},
  19:\penalty0 104--112, 1910.

\bibitem[Facchinei and Pang(2003)]{facchinei2003finite}
Francisco Facchinei and Jong-Shi Pang.
\newblock \emph{Finite-dimensional variational inequalities and complementarity
  problems}.
\newblock Springer, 2003.

\bibitem[Golowich et~al.(2020{\natexlab{a}})Golowich, Pattathil, and
  Daskalakis]{golowich2020tight}
Noah Golowich, Sarath Pattathil, and Constantinos Daskalakis.
\newblock Tight last-iterate convergence rates for no-regret learning in
  multi-player games.
\newblock \emph{Advances in neural information processing systems (NeurIPS)},
  2020{\natexlab{a}}.

\bibitem[Golowich et~al.(2020{\natexlab{b}})Golowich, Pattathil, Daskalakis,
  and Ozdaglar]{golowich2020last}
Noah Golowich, Sarath Pattathil, Constantinos Daskalakis, and Asuman Ozdaglar.
\newblock Last iterate is slower than averaged iterate in smooth convex-concave
  saddle point problems.
\newblock In \emph{Conference on Learning Theory (COLT)}, 2020{\natexlab{b}}.

\bibitem[Gorbunov et~al.(2022{\natexlab{a}})Gorbunov, Loizou, and
  Gidel]{gorbunov2022extragradient}
Eduard Gorbunov, Nicolas Loizou, and Gauthier Gidel.
\newblock Extragradient method: $\mathcal{O}(1/k)$ last-iterate convergence for
  monotone variational inequalities and connections with cocoercivity.
\newblock In \emph{International Conference on Artificial Intelligence and
  Statistics (AISTATS)}, pages 366--402. PMLR, 2022{\natexlab{a}}.

\bibitem[Gorbunov et~al.(2022{\natexlab{b}})Gorbunov, Taylor, and
  Gidel]{gorbunov2022last}
Eduard Gorbunov, Adrien Taylor, and Gauthier Gidel.
\newblock Last-iterate convergence of optimistic gradient method for monotone
  variational inequalities.
\newblock In \emph{Advances in Neural Information Processing Systems},
  2022{\natexlab{b}}.

\bibitem[Grimmer(2024)]{grimmer2024provably}
Benjamin Grimmer.
\newblock Provably faster gradient descent via long steps.
\newblock \emph{SIAM Journal on Optimization}, 34\penalty0 (3):\penalty0
  2588--2608, 2024.

\bibitem[Grimmer et~al.(2025)Grimmer, Shu, and Wang]{grimmer2025accelerated}
Benjamin Grimmer, Kevin Shu, and Alex~L Wang.
\newblock Accelerated objective gap and gradient norm convergence for gradient
  descent via long steps.
\newblock \emph{INFORMS Journal on Optimization}, 7\penalty0 (2):\penalty0
  156--169, 2025.

\bibitem[Haar(1910)]{haar1910zur}
Alfred Haar.
\newblock {Zur Theorie der orthogonalen Funktionensysteme: Erste Mitteilung}.
\newblock \emph{Mathematische Annalen}, 69\penalty0 (3):\penalty0 331--371,
  September 1910.
\newblock ISSN 0025-5831, 1432-1807.
\newblock \doi{10.1007/BF01456326}.

\bibitem[Halpern(1967)]{halpern1967fixed}
Benjamin Halpern.
\newblock Fixed points of nonexpanding maps.
\newblock \emph{Bulletin of the American Mathematical Society}, 73\penalty0
  (6):\penalty0 957--961, 1967.

\bibitem[Hsieh et~al.(2020)Hsieh, Iutzeler, Malick, and
  Mertikopoulos]{hsieh2020explore}
Yu-Guan Hsieh, Franck Iutzeler, J{\'e}r{\^o}me Malick, and Panayotis
  Mertikopoulos.
\newblock Explore aggressively, update conservatively: Stochastic extragradient
  methods with variable stepsize scaling.
\newblock \emph{Advances in Neural Information Processing Systems},
  33:\penalty0 16223--16234, 2020.

\bibitem[Hsieh et~al.(2021)Hsieh, Antonakopoulos, and
  Mertikopoulos]{hsieh2021adaptive}
Yu-Guan Hsieh, Kimon Antonakopoulos, and Panayotis Mertikopoulos.
\newblock Adaptive learning in continuous games: Optimal regret bounds and
  convergence to nash equilibrium.
\newblock In \emph{Conference on Learning Theory}, pages 2388--2422. PMLR,
  2021.

\bibitem[Ibrahim et~al.(2020)Ibrahim, Azizian, Gidel, and
  Mitliagkas]{ibrahim2020linear}
Adam Ibrahim, Wa{\"\i}ss Azizian, Gauthier Gidel, and Ioannis Mitliagkas.
\newblock Linear lower bounds and conditioning of differentiable games.
\newblock In \emph{Proceedings of the 37th {{International Conference}} on
  {{Machine Learning}}}, pages 4583--4593. PMLR, November 2020.

\bibitem[Kashin(1989)]{kashin1989orthogonal}
Boris~S. Kashin.
\newblock \emph{Orthogonal Series}.
\newblock American Mathematical Society, Providence, R.I., USA, 1989.
\newblock ISBN 978-0-8218-4527-1.

\bibitem[Kim et~al.(2023)Kim, Ozdaglar, Park, and Ryu]{kim2023time}
Jaeyeon Kim, Asuman Ozdaglar, Chanwoo Park, and Ernest Ryu.
\newblock Time-reversed dissipation induces duality between minimizing gradient
  norm and function value.
\newblock \emph{Advances in Neural Information Processing Systems},
  36:\penalty0 23389--23440, 2023.

\bibitem[Kornowski and Shamir(2024)]{kornowski2024open}
Guy Kornowski and Ohad Shamir.
\newblock Open problem: Anytime convergence rate of gradient descent.
\newblock In \emph{The Thirty Seventh Annual Conference on Learning Theory},
  pages 5335--5339. PMLR, 2024.

\bibitem[Korpelevich(1976)]{korpelevich1976extragradient}
Galina~M. Korpelevich.
\newblock The extragradient method for finding saddle points and other
  problems.
\newblock \emph{Ekonomika i Matematicheskie Metody}, 12:\penalty0 747--756,
  1976.

\bibitem[Kovalev and Gasnikov(2022)]{kovalev2022first}
Dmitry Kovalev and Alexander Gasnikov.
\newblock The first optimal algorithm for smooth and
  strongly-convex-strongly-concave minimax optimization.
\newblock \emph{Advances in Neural Information Processing Systems},
  35:\penalty0 14691--14703, 2022.

\bibitem[Lee and Kim(2021)]{lee2021fast}
Sucheol Lee and Donghwan Kim.
\newblock Fast extra gradient methods for smooth structured
  nonconvex-nonconcave minimax problems.
\newblock In \emph{Annual Conference on Neural Information Processing Systems
  (NeurIPS)}, 2021.

\bibitem[Lei et~al.(2021)Lei, Nagarajan, and Panageas]{lei2021last}
Qi~Lei, Sai~Ganesh Nagarajan, and Ioannis Panageas.
\newblock Last iterate convergence in no-regret learning: constrained min-max
  optimization for convex-concave landscapes.
\newblock In \emph{International Conference on Artificial Intelligence and
  Statistics}, 2021.

\bibitem[Liang and Stokes(2019)]{liang2019interaction}
Tengyuan Liang and James Stokes.
\newblock Interaction matters: A note on non-asymptotic local convergence of
  generative adversarial networks.
\newblock In \emph{The 22nd International Conference on Artificial Intelligence
  and Statistics}, pages 907--915. PMLR, 2019.

\bibitem[Lieder(2021)]{lieder2021convergence}
Felix Lieder.
\newblock On the convergence rate of the halpern-iteration.
\newblock \emph{Optimization Letters}, 15\penalty0 (2):\penalty0 405--418,
  2021.

\bibitem[Mokhtari et~al.(2020{\natexlab{a}})Mokhtari, Ozdaglar, and
  Pattathil]{mokhtari2020unified}
Aryan Mokhtari, Asuman Ozdaglar, and Sarath Pattathil.
\newblock A unified analysis of extra-gradient and optimistic gradient methods
  for saddle point problems: {{Proximal}} point approach.
\newblock In \emph{Proceedings of the {{Twenty Third International Conference}}
  on {{Artificial Intelligence}} and {{Statistics}}}, pages 1497--1507. PMLR,
  June 2020{\natexlab{a}}.

\bibitem[Mokhtari et~al.(2020{\natexlab{b}})Mokhtari, Ozdaglar, and
  Pattathil]{mokhtari2020convergence}
Aryan Mokhtari, Asuman~E Ozdaglar, and Sarath Pattathil.
\newblock Convergence rate of $\mathcal{O}(1/k)$ for optimistic gradient and
  extragradient methods in smooth convex-concave saddle point problems.
\newblock \emph{SIAM Journal on Optimization}, 30\penalty0 (4):\penalty0
  3230--3251, 2020{\natexlab{b}}.

\bibitem[Nemirovski(2004)]{nemirovski2004prox}
Arkadi Nemirovski.
\newblock Prox-method with rate of convergence {{$\mathcal{O}(1/t)$}} for
  variational inequalities with lipschitz continuous monotone operators and
  smooth convex-concave saddle point problems.
\newblock \emph{SIAM Journal on Optimization}, 15\penalty0 (1):\penalty0
  229--251, 2004.

\bibitem[Nemirovskij and Yudin(1983)]{nemirovskij1983problem}
Arkadij~Semenovi{\v{c}} Nemirovskij and David~Borisovich Yudin.
\newblock \emph{Problem complexity and method efficiency in optimization}.
\newblock Wiley-Interscience, 1983.

\bibitem[Nesterov(2007)]{nesterov2007dual}
Yurii Nesterov.
\newblock Dual extrapolation and its applications to solving variational
  inequalities and related problems.
\newblock \emph{Mathematical Programming}, 109\penalty0 (2):\penalty0 319--344,
  2007.

\bibitem[Nesterov(1983)]{nesterov1983method}
Yurii~Evgen'evich Nesterov.
\newblock A method of solving a convex programming problem with convergence
  rate {{$\mathcal{O}\bigl(\frac{1}{k^2}\bigr)$}}.
\newblock In \emph{Doklady Akademii Nauk}, volume 269, pages 543--547. Russian
  Academy of Sciences, 1983.

\bibitem[Niederreiter(1992)]{niederreiter1992random}
Harald Niederreiter.
\newblock \emph{Random Number Generation and Quasi-Monte Carlo Methods}.
\newblock {{CBMS-NSF Regional Conference Series}} in {{Applied Mathematics}}.
  {Society for Industrial and Applied Mathematics}, January 1992.
\newblock ISBN 978-0-89871-295-7.
\newblock \doi{10.1137/1.9781611970081}.

\bibitem[Popov(1980)]{popov1980modification}
Leonid~Denisovich Popov.
\newblock A modification of the arrow-hurwicz method for search of saddle
  points.
\newblock \emph{Mathematical notes of the Academy of Sciences of the USSR},
  28\penalty0 (5):\penalty0 845--848, 1980.

\bibitem[Rakhlin and Sridharan(2013)]{rakhlin2013optimization}
Sasha Rakhlin and Karthik Sridharan.
\newblock Optimization, learning, and games with predictable sequences.
\newblock \emph{Advances in Neural Information Processing Systems}, 2013.

\bibitem[Ryu and Yin(2022)]{Ryu22Large}
Ernest~K. Ryu and Wotao Yin.
\newblock \emph{Large-Scale Convex Optimization via Monotone Operators}.
\newblock Cambridge University Press, 2022.

\bibitem[Schauder(1927)]{schauder1927zur}
Julius Schauder.
\newblock {Zur Theorie stetiger Abbildungen in Funktionalr\"aumen [On the
  theory of continuous mappings in function spaces]}.
\newblock \emph{Mathematische Zeitschrift}, 26\penalty0 (1):\penalty0 47--65,
  December 1927.
\newblock ISSN 1432-1823.
\newblock \doi{10.1007/BF01475440}.

\bibitem[Sedlmayer et~al.(2023)Sedlmayer, Nguyen, and Bot]{sedlmayer2023fast}
Michael Sedlmayer, Dang-Khoa Nguyen, and Radu~Ioan Bot.
\newblock A fast optimistic method for monotone variational inequalities.
\newblock In \emph{International Conference on Machine Learning}, pages
  30406--30438. PMLR, 2023.

\bibitem[Shugart and Altschuler(2025)]{shugart2025negative/preprint}
Henry Shugart and Jason~M. Altschuler.
\newblock Negative stepsizes make gradient-descent-ascent converge, May 2025.

\bibitem[Stonyakin et~al.(2022)Stonyakin, Gasnikov, Dvurechensky, Titov, and
  Alkousa]{stonyakin2022generalized}
Fedor Stonyakin, Alexander Gasnikov, Pavel Dvurechensky, Alexander Titov, and
  Mohammad Alkousa.
\newblock Generalized mirror prox algorithm for monotone variational
  inequalities: {{Universality}} and inexact oracle.
\newblock \emph{Journal of Optimization Theory and Applications}, 194\penalty0
  (3):\penalty0 988--1013, September 2022.
\newblock ISSN 1573-2878.
\newblock \doi{10.1007/s10957-022-02062-7}.

\bibitem[Tran-Dinh(2022)]{tran2022connection}
Quoc Tran-Dinh.
\newblock The connection between {{Nesterov's}} accelerated methods and halpern
  fixed-point iterations.
\newblock \emph{arXiv preprint arXiv:2203.04869}, 2022.

\bibitem[Triebel(2010)]{triebel2010bases}
Hans Triebel.
\newblock \emph{Bases in Function Spaces, Sampling, Discrepancy, Numerical
  Integration}, volume~11 of \emph{{{EMS Tracts}} in {{Mathematics}}}.
\newblock EMS Press, 1 edition, June 2010.
\newblock ISBN 978-3-03719-085-2.
\newblock \doi{10.4171/085}.

\bibitem[Ul'yanov(1972)]{ulyanov1972representation}
Pyotr~Lavrentyevich Ul'yanov.
\newblock Representation of functions by series and classes {$\phi(L)$}.
\newblock \emph{Uspekhi Matematicheskikh Nauk}, 27\penalty0 (2 (164)):\penalty0
  3--52, April 1972.
\newblock ISSN 0036-0279.
\newblock \doi{10.1070/RM1972v027n02ABEH001370}.

\bibitem[{van der Corput}(1935)]{vandercorput1935verteilungsfunktionen}
Johannes~G. {van der Corput}.
\newblock Verteilungsfunktionen.
\newblock \emph{Proceedings of the Section of Sciences (Koninklijke Akademie
  van Wetenschappen te Amsterdam)}, 38\penalty0 (8):\penalty0 813--821,
  September 1935.

\bibitem[Wei et~al.(2021)Wei, Lee, Zhang, and Luo]{wei2021last}
Chen-Yu Wei, Chung-Wei Lee, Mengxiao Zhang, and Haipeng Luo.
\newblock Last-iterate convergence of decentralized optimistic gradient
  descent/ascent in infinite-horizon competitive markov games.
\newblock In \emph{Conference on learning theory}, pages 4259--4299. PMLR,
  2021.

\bibitem[Woodworth and Srebro(2016)]{woodworth2016tight}
Blake~E Woodworth and Nati Srebro.
\newblock Tight complexity bounds for optimizing composite objectives.
\newblock \emph{Advances in neural information processing systems}, 29, 2016.

\bibitem[Yoon and Ryu(2021)]{yoon2021accelerated}
TaeHo Yoon and Ernest~K. Ryu.
\newblock Accelerated algorithms for smooth convex-concave minimax problems
  with $\mathcal{O}(1/k^2)$ rate on squared gradient norm.
\newblock In \emph{Proceedings of the 38th {{International Conference}} on
  {{Machine Learning}}}, pages 12098--12109. PMLR, July 2021.

\bibitem[Yoon et~al.(2024)Yoon, Kim, Suh, and Ryu]{yoon2024optimal}
Taeho Yoon, Jaeyeon Kim, Jaewook~J. Suh, and Ernest~K. Ryu.
\newblock Optimal acceleration for minimax and fixed-point problems is not
  unique.
\newblock In \emph{Proceedings of the 41st International Conference on Machine
  Learning}, pages 57244--57314, 2024.

\bibitem[Young(1953)]{young1953richardson}
David Young.
\newblock On {{Richardson}}'s method for solving linear systems with positive
  definite matrices.
\newblock \emph{Journal of Mathematics and Physics}, 32\penalty0
  (1-4):\penalty0 243--255, 1953.

\bibitem[Zhang and Jiang(2024)]{zhang2024accelerated/preprint}
Zehao Zhang and Rujun Jiang.
\newblock Accelerated gradient descent by concatenation of stepsize schedules,
  October 2024.
\newblock ArXiv preprint, 2410.12395.

\bibitem[Zhang et~al.(2025)Zhang, Lee, Du, and Chen]{zhang2025anytime}
Zihan Zhang, Jason Lee, Simon Du, and Yuxin Chen.
\newblock Anytime acceleration of gradient descent.
\newblock In \emph{The Thirty Eighth Annual Conference on Learning Theory},
  pages 5991--6013. PMLR, 2025.

\end{thebibliography}

\appendix
\crefalias{section}{appendix}
\crefalias{subsection}{appendix}
\crefalias{subsubsection}{appendix}

\tableofcontents

\section{Lemmas involving elementary math and calculus}\label{ah:elem}

\begin{applemma}\label{lem:ln-plus-monomial}
    For $a,b,c>0$, we have
    \begin{align*}
        \max_{x\in (0,\infty)} \cbrm[\big]{- a x^b + c \ln x} = -\frac{c\ln(ab/c)+c}{b},
    \end{align*}
    where the maximum is achieved at $x=(ab/c)^{-1/b}$.
\end{applemma}
\begin{proof}
    Define $f(x)=-a x^b+c\ln x$. Using the first-order condition,
    \begin{align*}
        f'(x)=\frac cx - ab x^{b-1},
    \end{align*}
    and setting $f'(x^*)=0$ gives
    \begin{align*}
        \frac c{x^*} - ab (x^*)^{b-1} = 0 \implies x^*=(ab/c)^{-1/b}.
    \end{align*}
    It is easy to verify that $f'(x)>0$ for $0<x<x^*$ and $f'(x)<0$ for $x>x^*$. Therefore, \phantom{\qedhere}
    \begin{align*}
        \max_{x\in (0,\infty)}f(x)=f(x^*)=- a (ab/c)^{-1} + c\ln\sbrm[\big]{ (ab/c)^{-1/b} } = -\frac{c\ln(ab/c)+c}{b}.\mqed
    \end{align*}
\end{proof}

\begin{applemma}\label{lem:ceil-sub-1}
    If $x\in \dR,y\in [0, 1)$, then $\abs{\ceil{x-y}-x} \leq 1$.
\end{applemma}
\begin{proof}\phantom{\qedhere}
    \begin{align*}
        \ceil{x-y}-x&\leq \ceil{x}-x<1,\\
        \ceil{x-y}-x&\geq\ceil{x-1}-x\geq x-1-x=-1.\mqed
    \end{align*}
\end{proof}

\begin{applemma}\label{lem:m-func} For $0<s<1$ and $-\pi<\theta<\pi$, we have
    \begin{align*}
        \int_0^{\infty} \ln(1+2t\cos \theta+t^2)t^{-s-1}\dd{t}=\frac{2\pi \cos(\theta s)}{s\sin(\pi s)}.
    \end{align*}
    As a special case, when $\theta=2\pi/3$, we have
    \begin{align*}
        \int_0^{\infty} \ln(1-t+t^2)t^{-s-1}\dd{t}=\frac{2\pi \cos(2\pi s/3)}{s\sin(\pi s)}.
    \end{align*}
\end{applemma}

This equality appears in \cite[p. 316, Eq. 6.4(27)]{bateman1954integral}.

\begin{proof}
We use Feynman's Trick. Define $J(\theta)=\int_0^{\infty} \ln(1+2t\cos \theta+t^2)t^{-s-1}dt$, and differentiate it with respect to $\theta$,
\begin{align*}
    J'(\theta) 
    = \int_{0}^{\infty} \frac{-2t\sin\theta}{1+2t\cos\theta+t^2} t^{-s-1} \dd{t} 
    = -2\sin\theta \int_{0}^{\infty} \frac{t^{-s}}{1+2t\cos\theta+t^2} \dd{t}.
\end{align*}
We can decompose the integrand using partial fractions: with $e^{i\theta} + e^{-i\theta} = 2\cos\theta$ and $e^{i\theta} - e^{-i\theta} = 2i\sin\theta$, we have:
\begin{align*}
\frac{1}{1+2t\cos\theta+t^2} 
= \frac{1}{(t+e^{i\theta})(t+e^{-i\theta})} 
= \frac{1}{2i\sin\theta} \rbr{ \frac{1}{t+e^{-i\theta}} - \frac{1}{t+e^{i\theta}} }.
\end{align*}
Therefore, substituting this back into the integration yields
\begin{align*}
    J'(\theta) 
    & = -2\sin\theta \cdot \frac{1}{2i\sin\theta} \rbr{ \int_{0}^{\infty} \frac{t^{-s}}{t+e^{-i\theta}} \dd{t} - \int_{0}^{\infty} \frac{t^{-s}}{t+e^{i\theta}} \dd{t} } \\
    & = i \rbr{ \int_{0}^{\infty} \frac{t^{-s}}{t+e^{-i\theta}} \dd{t} - \int_{0}^{\infty} \frac{t^{-s}}{t+e^{i\theta}} \dd{t} }. \yestag\label{eq:lem-m-func-j'}
\end{align*}

For any $z\in \dC$, we have
\begin{align*}
    \int_{0}^{\infty} \frac{t^{-s}}{t+z} \dd{t}=\int_0^\infty \frac{(zu)^{-s}}{zu+z}z\dd{u}=z^{-s}\int_0^\infty \frac{u^{-s}}{u+1}\dd{u}=\frac{\pi z^{-s}}{\sin(\pi s)},
\end{align*}
where we substitute $t=zu$ and use the standard Mellin transform $\int_0^\infty \frac{u^{-s}}{u+1}\dd{u}=\frac{\pi}{\sin(\pi s)}$. Plugging this back into \eqref{eq:lem-m-func-j'}, and we have
\begin{align*}
    J'(\theta) 
    = i \rbr{ \frac{\pi e^{is\theta}}{\sin(\pi s)} - \frac{\pi e^{-is\theta}}{\sin(\pi s)} } 
    = \frac{i\pi}{\sin(\pi s)} \rbr{ e^{is\theta} - e^{-is\theta} }
    = -\frac{2\pi \sin(\theta s)}{\sin(\pi s)}.
\end{align*}
Integrating both sides with respect to $\theta$, and we can recover $J(\theta)$ as
\begin{align*}
    J(\theta) = \int -\frac{2\pi \sin(s\theta)}{\sin(\pi s)} \dd{\theta} = \frac{2\pi \cos(s\theta)}{s \sin(\pi s)} + C(s).
\end{align*}

To determine the constant of integration $C(s)$, we evaluate $J(\theta)$ at $\theta=0$:
\begin{align*}
    J(0)
    & = \int_0^\infty \ln(1+2t+t^2)t^{-s-1}\dd{t} = 2\int_0^\infty \ln(1+t)t^{-s-1}\dd{t} \\
    & = \frac{2}{-s}\int_0^\infty \ln(1+t)\dd{t^{-s}} \\
    & = \frac{2}{-s} \underbrace{\rbr{\ln(1+t)\cdot t^{-s}}\big|_0^\infty}_{{}=0}
        {} + \frac{2}{s}\int_0^\infty t^{-s}\dd{\,\ln(1+t)} \tag{integration by parts}\\
    & = \frac{2}{s}\int_0^\infty \frac{t^{-s}}{1+t}\dd{t} = \frac{2\pi}{s\sin(\pi s)},
\end{align*}
where, in the last step, we again used the aforementioned standard Mellin transform. Therefore,
\begin{align*}
    C(s)=\frac{2\pi}{s\sin(\pi s)}-\frac{2\pi \cos(0)}{s \sin(\pi s)}=0.
\end{align*}

Thus we have proved the formula for $J(\theta)$. The special case of $\theta=-2\pi/3$ can be derived by simply plugging the value in.
\end{proof}

\begin{applemma}\label{lem:dseg-function-convexity}
    When $0<\rho<\frac{3-\sqrt{7}}{2}$, the function $f(z)=\ln(1+(\rho-2)z+z^2)$ is concave on $[0, \frac12]$, and satisfies $f'(0)=\rho-2$.
\end{applemma}
\begin{proof}
    We only need to prove $f''(z)\leq 0$. We have
    \begin{align*}
        f'(z)&=\frac{(\rho-2)+2z}{1+(\rho-2)z+z^2},\\
        f''(z)&=\frac{2}{1+(\rho-2)z+z^2}-\frac{[(\rho-2)+2z]^2}{[1+(\rho-2)z+z^2]^2}\\
        &= \frac{-2 z^2-2 (\rho-2) z-\rho ^2+4 \rho -2}{[1+(\rho-2)z+z^2]^2}.
    \end{align*}
    The second claim is obtained by simply substituting $z=0$ into $f'(z)$. For the first claim, solving the quadratic equation $-2 z^2-2 (\rho-2) z-\rho ^2+4 \rho -2=0$ gives two solutions $\{z_1,z_2\}=\cbrm[\big]{\frac12 \rbrm[\big]{2-\rho\pm\sqrt{4\rho-\rho^2}}}$; then the solution of the quadratic inequality $-2 z^2-2 (\rho-2) z-\rho ^2+4 \rho -2\leq 0$ is the region $(-\infty, z_1]\cup [z_2, +\infty)$.

    We show that $[0, \frac12] \subseteq (-\infty, z_1]$, or equivalently, $z_1\geq \frac12$. Indeed, the function $\rho\mapsto z_1=\frac12 \rbrm[\big]{2-\rho-\sqrt{4\rho-\rho^2}}$ is monotonically decreasing on $(0, \frac{3-\sqrt{7}}{2})$, because both $\rho\mapsto\rho$ and $\rho\mapsto4\rho-\rho^2$ are monotonically increasing. Also, when $\rho=\frac{3-\sqrt{7}}{2}$, direct calculation shows $z_1=\frac12$. The claim is thus proved.
\end{proof}

\begin{applemma}\label{lem:zeta-bound}
    The Riemann zeta function
    \begin{align*}
        \zeta(z)=\sum_{n=1}^\infty \frac{1}{n^z}=\frac{1}{1^z}+\frac{1}{2^z}+\frac{1}{3^z}+\cdots
    \end{align*}
    is monotonically decreasing on $(1, \infty)$ and satisfies $\zeta(z)<1+\frac 1{z-1}$.
\end{applemma}
\begin{proof}
    Monotonicity comes from the monotonicity of each term in the summation. The upper bound is proved with an integral,\phantom{\qedhere}
    \begin{align*}
        \zeta(z)
        & = 1+\sum_{n=2}^{\infty}\frac{1}{n^z} \\
        & < 1 + \sum_{n=2}^\infty \int_{n-1}^n \frac{1}{x^z}\dd{x} \\
        & = 1 + \int_{1}^\infty \frac{1}{x^z}\dd{x} = 1 + \frac{1}{z-1}.\mqed
    \end{align*}
\end{proof}

\begin{applemma}\label{lem:og-first-order-0}
    For $t>0$, we have
    \begin{align*}
        \ln\frac{4t^2+1+\sqrt{(4t^2+1)^2-4t}}{2}\geq -t.
    \end{align*}
\end{applemma}
\begin{proof}
    Let $f(t)=\ln\frac{4t^2+1+\sqrt{(4t^2+1)^2-4t}}{2}+t$; we only need to prove $f(t)>0$ for $t>0$. It is clear that $f(0)=0$; we claim that $f'(t)\geq 0$ for $t>0$. Taking the derivative,
    \begin{align*}
        f'(t)
        & = \frac{\frac{\dd}{\dd{t}}(4t^2+1+\sqrt{(4t^2+1)^2-4t})/2}{4t^2+1+\sqrt{(4t^2+1)^2-4t}/2}+1\\
        & = \frac{8t+\frac{16t(1+4t^2)-4}{2\sqrt{(4t^2+1)^2-4t}}}{4t^2+1+\sqrt{(4t^2+1)^2-4t}}+1\\
        & = \frac{12t^2-1+\sqrt{(4t^2+1)^2-4t}}{2t\sqrt{(4t^2+1)^2-4t}}+1\\
        & = \frac{12t^2-1+(2t+1)\sqrt{(4t^2+1)^2-4t}}{2t\sqrt{(4t^2+1)^2-4t}}.
    \end{align*}
    The denominator is clearly nonnegative. If $12t^2-1\geq 0$, the numerator is also clearly nonnegative; now assume $12t^2-1<0$. We need to prove
    \begin{align*}
        && 12t^2-1+(2t+1)\sqrt{(4t^2+1)^2-4t} & \geq 0 \\
        \Longleftrightarrow && (2t+1)\sqrt{(4t^2+1)^2-4t} & \geq 1-12t^2 \\
        \Longleftrightarrow && (2t+1)^2\rbrm[\big]{(4t^2+1)^2-4t} & \geq (1-12t^2)^2 \\
        \Longleftrightarrow && 64 t^6+64 t^5+48 t^4+16 t^3-4 t^2+1 & \geq 144 t^4-24 t^2+1 \\
        \Longleftrightarrow && 4t^2 (16 t^4+16 t^3-24 t^2+4 t+5) & \geq 0\\
        \Longleftrightarrow && 4t^2 \rbrm[\big]{
                (4t^2-1)^2+t(4t-2)^2+4
            }& \geq 0,
    \end{align*}
    and the last statement is true. Therefore, we have $f'(t)\geq 0$ for all $t>0$. Together with the fact that $f(0)=0$, this proves $f(t)\geq 0$ for all $t>0$, which is equivalent to the statement of the lemma.
\end{proof}

\begin{applemma}\label{lem:m-func-og}
    For $0<s<1$, the integral
    \begin{align*}
        &\int_{0}^\infty \ln\rbrm[\Big]{
            \frac{4t^2+1+\sqrt{(4t^2+1)^2-4t}}{2}
        } t^{-s-1}\,\dd{t}\\
        {}={} & \begin{aligned}[t]
            &\frac{2^s\pi}{s\sin(\pi s/2)}
            \ghg43\rbrm[\Big]{
                -\frac{s}{2},\frac{s}{6},\frac{s+2}{6},\frac{s+4}{6};
                1, \frac12, \frac12;
                \frac{27}{64}
            }\\
            &{}-
            \frac{2^{s-3}(s+1)\pi}{\cos(\pi s/2)}
            \ghg43\rbrm[\Big]{
                \frac{1-s}{2},\frac{s+3}{6},\frac{s+5}{6},\frac{s+7}{6};
                1, \frac32, \frac32;
                \frac{27}{64}
            },
            \end{aligned}
    \end{align*}
    where $\ghg43$ is the generalized hypergeometric function
    \begin{align*}
        \ghg{p}{q}(a_1,\dots,a_p;b_1,\dots,b_q;z)
        = \sum_{n=0}^\infty \frac{\rpow{a_1}{n}\cdots\rpow{a_p}{n}}{\rpow{b_1}{n}\cdots\rpow{b_q}{n}}\frac{z^n}{n\fact},
    \end{align*}
    and $\rpow{a}{n}$ is the rising factorial power
    \begin{align*}
        \rpow{a}{n}=a(a+1)\cdots(a+n-1).
    \end{align*}
\end{applemma}
\begin{proof}
    We define 
    \begin{align*}
        f(t)=\frac{4t^2+1+\sqrt{(4t^2+1)^2-4t}}{2}
        =\frac{4t^2+1}{2}\rbrm[\Bigg]{
            1+\sqrt{
                1-\frac{4t}{(1+4t^2)^2}
            }
        }.
    \end{align*}

    Let $u=\frac{4t}{(4t^2+1)^2}$. For all $t\geq 0$, from $4t^2+1\geq 1$ and AM-GM, we have $u\leq \frac{4t}{4t^2+1}\leq 1$. Therefore, we separate the logarithm
    \begin{align}
        \ln f(t)=\ln(4t^2+1)+\ln\rbrm[\Big]{
            \frac{1+\sqrt{1-u}}{2}
        }.\label{eq:ln-ft-expand}
    \end{align}
    Let $g(u)=\ln\rbrm[\big]{ \frac{1+\sqrt{1-u}}{2} }$. We have $g'(u)=\frac{1}{1+\sqrt{1-u}}\frac{\dd}{\dd{u}}(1+\sqrt{1-u})=\frac{1}{2u}-\frac{1}{2u\sqrt{1-u}}$. From the generalized binomial theorem, we have
    \begin{align*}
        (1-u)^{-1/2}
        &=\sum_{n=0}^\infty {-1/2\choose n} (-u)^n \\
        &=\sum_{n=0}^\infty \frac{(-1/2)(-3/2)\cdots(-1/2-n+1)}{n\fact} (-1)^n u^n \\
        &=\sum_{n=0}^\infty \frac{(2n)\fact}{(n\fact)^2 4^n} u^n.
    \end{align*}
    Now, substitute this series back into our expression for $g'(u)$ and extracting the $n=0$ term:
    \begin{align*}
        g'(u)=\frac1{2u}-\frac{1}{2u}\rbrm[\Big]{
            1+\sum_{n=1}^\infty \frac{(2n)\fact}{(n\fact)^2 4^n} u^n
        }=-\frac12\sum_{n=1}^\infty \frac{(2n)\fact}{(n\fact)^2 4^n} u^{n-1}.
    \end{align*}
    Integrate term-by-term and notice that $g(0)=\ln 1=0$, we have
    \begin{align*}
        g(u) 
        & = -\frac12\sum_{n=1}^\infty \frac{(2n)\fact}{(n\fact)^2 4^n } \int_{0}^u t^{n-1}\,\dd{t} \\
        & = -\sum_{n=1}^\infty \frac{(2n)\fact}{2n (n\fact)^2 4^n} u^n.
    \end{align*}
    Substituting the definition of $u$ gives
    \begin{align*}
        \ln\rbrm[\Bigg]{
            \frac{1+\sqrt{
                1-\frac{4t}{(1+4t^2)^2}
            }}{2}
        }
        = g\rbrm[\Big]{\frac{4t}{(4t^2+1)^2}}
        & = -\sum_{n=1}^\infty \frac{(2n)\fact}{2n (n\fact)^2 4^n} \rbrm[\Big]{\frac{4t}{(4t^2+1)^2}}^n \\
        & = -\sum_{n=1}^\infty \frac{(2n)\fact}{2n (n\fact)^2} \frac{t^n}{(4t^2+1)^{2n}}
    \end{align*}
    Define $I$ as the integration we try to evaluate in this lemma.
    From \eqref{eq:ln-ft-expand} and integrating term-by-term on the sum above, we have
    \begin{align*}
        I & =
        \int_0^{\infty} \sbr{\ln(4t^2+1) + \ln\rbrm[\Bigg]{
            \frac{1+\sqrt{
                1-\frac{4t}{(1+4t^2)^2}
            }}{2}
        }}t^{-s-1} \,\dd{t} \\
        & = 
        \underbrace{\int_0^{\infty} \ln(4t^2+1) t^{-s-1}\,\dd{t}}_{{} \defeq I_0}
        -\sum_{n=1}^\infty
        \frac{(2n)\fact}{2n (n\fact)^2} \underbrace{\int_0^{\infty} \frac{t^{n-s-1}}{(4t^2+1)^{2n}}\,\dd{t}}_{{}\defeq I_n}. \yestag\label{eq:Is-expand-integrals}
    \end{align*}
    We evaluate $I_0$ first. Using the substitution $u=2t$, $\dd{t}=\frac12 \dd{u}$, we have
    \begin{align*}
        I_0
        & = \int_0^{\infty} \ln(1+u^2) (u/2)^{-s-1}\frac12\dd{u}\\
        & = 2^s \int_0^{\infty} \ln(1+u^2) u^{-s-1}\,\dd{u} \\
        & = 2^s \frac{2\pi \cos(\pi s/2)}{s \sin(\pi s)} \tag{\cref{lem:m-func}, $\theta=\pi/2$}\\
        & = \frac{2^s\pi}{s \sin(\pi s/2)},
    \end{align*}
    where in the last step we use the trigonometric identity $\sin(2\theta)=2\sin\theta\cos\theta$ with $\theta=\frac{\pi s}{2}$.

    We then evaluate $I_n$. Using the substitution $v=4t^2$, $t=\frac{v^{1/2}}{2}$ and $\dd{t}=\frac14 v^{-1/2}\,\dd{v}$, we have
    \begin{align*}
        I_n
        & = \int_0^{\infty} \frac{\rbrm[\big]{\frac{v^{1/2}}{2}}^{n-s-1}}{(v+1)^{2n}}\frac14 v^{-1/2}\,\dd{v}\\
        & = 2^{s-n-1} \int_0^\infty \frac{v^{(n-s)/2-1}}{(1+v)^{2n}}\dd{v} \\
        & = 2^{s-n-1} \Beta\rbrm[\Big]{\frac{n-s}{2},\frac{3n+s}{2}} \\
        & = 2^{s-n-1} \frac{
            \Gamma\rbrm[\big]{\frac{n-s}{2}}
            \Gamma\rbrm[\big]{\frac{3n+s}{2}}
        }{
            \Gamma\rbrm[\big]{2n}
        }
    \end{align*}
    where $\Beta(x,y)=\int_0^\infty \frac{t^{x-1}}{(1+t)^{x+y}}\dd{t}=\frac{\Gamma(x)\Gamma(y)}{\Gamma(x+y)}$ is the Beta function.

    Substituting these integrals back into \eqref{eq:Is-expand-integrals}, we have
    \begin{align*}
        I
        & = 
        \frac{2^s\pi}{s \sin(\pi s/2)}
        -\sum_{n=1}^\infty
        \frac{(2n)\fact}{2n (n\fact)^2} 2^{s-n-1} \frac{
            \Gamma\rbrm[\big]{\frac{n-s}{2}}
            \Gamma\rbrm[\big]{\frac{3n+s}{2}}
        }{
            \Gamma\rbrm[\big]{2n}
        } \\
        & = \frac{2^s\pi}{s \sin(\pi s/2)}
        -\sum_{n=1}^\infty
        \frac{1}{(n\fact)^2} 2^{s-n-1}{
            \Gamma\rbrm[\Big]{\frac{n-s}{2}}
            \Gamma\rbrm[\Big]{\frac{3n+s}{2}}
        }
    \end{align*}
    Coincidentally, if we extend the sum to $n=0$, the term evaluates to
    \begin{align*}
        \frac{1}{(0\fact)^2} 2^{s-0-1}{
            \Gamma\rbrm[\Big]{\frac{0-s}{2}}
            \Gamma\rbrm[\Big]{\frac{0+s}{2}}
        }
        & =2^{s-1}\Gamma(-s/2)\Gamma(s/2) \\
        & = 2^{s-1}\Gamma(1-s/2)\Gamma(s/2)\cdot \frac{\Gamma(-s/2)}{\Gamma(1-s/2)} \\
        & = 2^{s-1}\cdot \frac{\pi}{\sin(\pi s/2)}\cdot \frac{1}{-s/2} = -I_0,
    \end{align*}
    where we used the identities $\Gamma(z)\Gamma(1-z)=\frac{\pi}{\sin(\pi z)}$ and $\Gamma(z+1)=z\Gamma(z)$. So, the $I_0$ term is absorbed into the summation as the $n=0$ term, and
    \begin{align}
        I
        & = 
        -\sum_{n=0}^\infty
        \underbrace{\frac{1}{(n\fact)^2} 2^{s-n-1}{
            \Gamma\rbrm[\Big]{\frac{n-s}{2}}
            \Gamma\rbrm[\Big]{\frac{3n+s}{2}}
        }}_{{}\defeq T_n}.\label{eq:Is-series}
    \end{align}
    The series converges at a reasonably fast rate (expanding using Stirling's formula shows that terms grow at a rate of $(\sqrt{27/64})^n\approx 0.65^n$). However, we will further simplify it by rewriting it as hypergeometric functions.

    A series can be written as a generalized hypergeometric function only when the ratio between consecutive terms is a rational function in the index. Formally, let $G_k=\frac{\rpow{a_1}{k}\cdots\rpow{a_p}{k}}{\rpow{b_1}{k}\cdots\rpow{b_q}{k}}\frac{z^k}{k\fact}$ be the $k$-th term in the definition of $\ghg{p}{q}$, and we have
    \begin{align*}
        \frac{G_{k+1}}{G_n}=\frac{(k+a_1)\cdots(k+a_p)}{(k+b_1)\cdots(k+b_q)}\frac{z}{(k+1)},
    \end{align*}
    which is a rational function in $k$. In the converse, all rational functions in $k$ can be represented this way, because of the fundamental theorem of algebra.

    By observing \eqref{eq:Is-series}, we see that the ratio $T_{n+1}/T_n$ cannot be easily simplified, as it is nontrivial to relate $\Gamma\rbrm[\big]{z+\frac12}$ to $\Gamma(z)$. However, if we consider $T_{n+2}/T_n$, we need only to relate $\Gamma(z+1)$ to $\Gamma(z)$, for which we can use $\Gamma(z+1)=z\Gamma(z)$. Specifically, we have
    \begin{align*}
        \frac{T_{n+2}}{T_n}
        & =
            \frac{(n\fact)^2}{((n+2)\fact)^2}
            \frac{2^{s-n-3}}{2^{s-n-1}}
            \frac{\Gamma\rbrm[\big]{\frac{n-s}{2}+1}}{\Gamma\rbrm[\big]{\frac{n-s}{2}}}
            \frac{\Gamma\rbrm[\big]{\frac{3n+s}{2}+3}}{\Gamma\rbrm[\big]{\frac{3n+s}{2}}} \\
        & = \frac{1}{(n+1)^2(n+2)^2}
            \cdot
            \frac{1}{4}
            \cdot 
            \frac{n-s}{2}
            \rbrm[\Big]{
                \frac{3n+s}{2}
            }\rbrm[\Big]{
                \frac{3n+s}{2}+1
            }\rbrm[\Big]{
                \frac{3n+s}{2}+2
            } \\
        & = \frac{
            (n-s)
            \rbrm[\big]{n+\frac{s}{3}}
            \rbrm[\big]{n+\frac{s+2}{3}}
            \rbrm[\big]{n+\frac{s+4}{3}}
        }{
            (n+1)^2
            (n+2)^2
        }\cdot \underbrace{\frac14 \cdot \frac12 \cdot \rbrm[\Big]{\frac32}^3}_{{}=\frac{27}{64}}.
    \end{align*}
    It is easy to see that the ratio is a rational function in $n$. Therefore, we split the sum into even and odd entries,
    \begin{align*}
        I
        & = 
        -\sum_{n=0}^\infty T_n
        = -\sum_{k=0}^\infty T_{2k}
        -\sum_{k=0}^\infty T_{2k+1}
    \end{align*}
    For the sum of $\{-T_{2k}\}_{k\in \dN}$, the ratio between consecutive terms is
    \begin{align*}
        \frac{-T_{2(k+1)}}{-T_{2k}}
        & = \frac{
            (2k-s)
            \rbrm[\big]{2k+\frac{s}{3}}
            \rbrm[\big]{2k+\frac{s+2}{3}}
            \rbrm[\big]{2k+\frac{s+4}{3}}
        }{
            (2k+1)^2
            (2k+2)^2
        }\cdot \frac{27}{64} \\
        & = \frac{
            \rbrm[\big]{k-\frac{s}{2}}
            \rbrm[\big]{k+\frac{s}{6}}
            \rbrm[\big]{k+\frac{s+2}{6}}
            \rbrm[\big]{k+\frac{s+4}{6}}
        }{
            \rbrm[\big]{k+\frac{1}{2}}^2
            \rbrm{k+1}
        }
        \frac{
            27/64
        }{k+1},
    \end{align*}
    and the initial term is
    \begin{align*}
        -T_0=I_0=\frac{2^s\pi}{s \sin(\pi s/2)}.
    \end{align*}
    Therefore, the even sum is
    \begin{align*}
        -\sum_{k=0}^\infty T_{2k}
        & = \sum_{k=0}^\infty (-T_0) \frac{
            \rpow{\rbrm[\big]{-\frac{s}{2}}}{k}
            \rpow{\rbrm[\big]{\frac{s}{6}}}{k}
            \rpow{\rbrm[\big]{\frac{s+2}{6}}}{k}
            \rpow{\rbrm[\big]{\frac{s+4}{6}}}{k}
        }{
            \rpow{\rbrm[\big]{\frac{1}{2}}}{k}
            \rpow{\rbrm[\big]{\frac{1}{2}}}{k}
            \rpow{1}{k}
        }
        \frac{
            (27/64)^k
        }{k\fact} \\
        & = (-T_0)\cdot \ghg43\rbrm[\Big]{
                -\frac{s}{2},\frac{s}{6},\frac{s+2}{6},\frac{s+4}{6};
                \frac12, \frac12, 1;
                \frac{27}{64}
            } \\
        & = \frac{2^s\pi}{s \sin(\pi s/2)} \ghg43\rbrm[\Big]{
                -\frac{s}{2},\frac{s}{6},\frac{s+2}{6},\frac{s+4}{6};
                \frac12, \frac12, 1;
                \frac{27}{64}
            }.
    \end{align*}
    Similarly, for the sum of $\{-T_{2k+1}\}_{k\in \dN}$, the ratio between consecutive terms is
    \begin{align*}
        \frac{-T_{2(k+1)+1}}{-T_{2k+1}}
        & = \frac{
            (2k+1-s)
            \rbrm[\big]{2k+1+\frac{s}{3}}
            \rbrm[\big]{2k+1+\frac{s+2}{3}}
            \rbrm[\big]{2k+1+\frac{s+4}{3}}
        }{
            (2k+1+1)^2
            (2k+1+2)^2
        }\cdot \frac{27}{64} \\
        & = \frac{
            \rbrm[\big]{k+\frac{1-s}{2}}
            \rbrm[\big]{k+\frac{s+3}{6}}
            \rbrm[\big]{k+\frac{s+5}{6}}
            \rbrm[\big]{k+\frac{s+7}{6}}
        }{
            \rbrm{k+1}
            \rbrm[\big]{k+\frac{3}{2}}^2
        }
        \frac{
            27/64
        }{k+1},
    \end{align*}
    and the initial term is
    \begin{align*}
        -T_1
        & = -\frac{1}{(1\fact)^2} 2^{s-1-1}{
            \Gamma\rbrm[\Big]{\frac{1-s}{2}}
            \Gamma\rbrm[\Big]{\frac{3+s}{2}}
        } \\
        & = -2^{s-2}{
            \Gamma\rbrm[\Big]{\frac{1-s}{2}}
            \Gamma\rbrm[\Big]{\frac{1+s}{2}}
            \cdot
            \frac
                {\Gamma\rbrm[\big]{\frac{3+s}{2}}}
                {\Gamma\rbrm[\big]{\frac{1+s}{2}}}
        } \\
        & = -2^{s-2}{
            \frac{\pi}{\sin(\pi(1-s)/2)}
            \cdot
            \frac{1+s}{2}
        } \\
        & = -2^{s-3}{
            \frac{\pi(1+s)}{\cos(\pi s/2)}
        }
    \end{align*}
    where, again, we used the identities $\Gamma(z)\Gamma(1-z)=\frac{\pi}{\sin(\pi z)}$ and $\Gamma(z+1)=z\Gamma(z)$; we also used the trigonometric identity $\sin(\pi/2-\theta)=\cos\theta$. Therefore, the odd sum is
    \begin{align*}
        -\sum_{k=0}^\infty T_{2k+1}
        & = \sum_{k=0}^\infty (-T_1) \frac{
            \rpow{\rbrm[\big]{\frac{1-s}{2}}}{k}
            \rpow{\rbrm[\big]{\frac{s+3}{6}}}{k}
            \rpow{\rbrm[\big]{\frac{s+5}{6}}}{k}
            \rpow{\rbrm[\big]{\frac{s+7}{6}}}{k}
        }{
            \rpow{1}{k}
            \rpow{\rbrm[\big]{\frac{3}{2}}}{k}
            \rpow{\rbrm[\big]{\frac{3}{2}}}{k}
        }
        \frac{
            (27/64)^k
        }{k\fact} \\
        & = (-T_1)\cdot \ghg43\rbrm[\Big]{
                \frac{1-s}{2},\frac{s+3}{6},\frac{s+5}{6},\frac{s+7}{6};
                1, \frac32, \frac32;
                \frac{27}{64}
            } \\
        & = -{
            \frac{2^{s-3}(1+s)\pi}{\cos(\pi s/2)}
            } \ghg43\rbrm[\Big]{
                \frac{1-s}{2},\frac{s+3}{6},\frac{s+5}{6},\frac{s+7}{6};
                1, \frac32, \frac32;
                \frac{27}{64}
            }.
    \end{align*}
     Combining the even and odd sums finishes the proof.
\end{proof}

\section{Supporting proofs on decomposition and reformulation as optimization}\label{ah:proof-decompose}

In this part of the appendix, we will prove our reduction from the stepsize scheduling problem of general biaffine min-max optimizations to the optimization problem in \cref{lem:opt}. 

Assume
\begin{equation}
\ell(x,y)=x^\trans A y + p^\trans y + x^\trans q + c
=
\rbrm[\Big]{\begin{matrix}
    x\\1
\end{matrix}}^\trans
\rbrm[\bigg]{\begin{matrix}
    A&q\\p^\trans&c
\end{matrix}}
\rbrm[\Big]{\begin{matrix}
    y\\1
\end{matrix}}\label{eq:def-biaffine'}
\end{equation}
is a biaffine function.
If a saddle point $z^\star=(x^\star,y^\star)$ exists for $\ell$, the first-order condition $G(z^\star)=0$ implies $p=-A^\trans x^\star$ and $q=-Ay^\star$, which simplifies the function to $\ell(x,y)=(x-x^\star)^\trans A (y-y^\star)+c'$. It is easy to see that the first-order condition $G(z^\star)=0$ is also sufficient for $z^\star$ to be a saddle point when $\ell$ is biaffine.

We will first prove the equality form of \cref{lem:opt} for $n=m=1$ and bilinear $\ell$.

\begin{applemma}\label{lem:opt-1x1}
    Let $(\dR, \dR, \ell(x,y)=axy)$ be a bilinear min-max optimization problem. The trajectory of EG in this problem satisfies the following:
    \begin{align}
        \GN(z_T)
        =
            \norm{z_0}
                \cdot \abs{a} \prod_{t=0}^{T-1} \sqrt{1+\eta_t(\eta_t-2\gamma_t) a^2+\eta_t^2\gamma_t^2 a^4}.
        \label{eq:eg-1x1}
    \end{align}
\end{applemma}
\begin{remark}
    This lemma also holds for any 1d Euclidean vector spaces. Formally, assume $(\calX, \calY, \ell)$ is a bilinear min-max optimization problem defined on real inner product spaces $\calX,\calY$ with $\dim\calX=\dim\calY=1$. Then, the trajectory of EG satisfies \eqref{eq:eg-1x1},
    where $a$ is $\ell(u,v)$ for arbitrary vectors $u\in \calX,v\in \calY$ such that $\norm{u}=\norm{v}=1$. This claim can be simply shown with isomorphisms $\calX\iso\dR$ and $\calY\iso\dR$ that maps $u$ and $v$ to $1$ respectively.
\end{remark}
\begin{proof}[Proof of \cref{lem:opt-1x1}]
    As the loss function is bilinear, the point $z^\star=(0,0)$ is a saddle point.
    
    Let $z_t=(x_t, y_t)$ be the trajectory for $t\in \dN$.
    From the definition $\ell(x,y)=axy$, we know that 
    $G_{\ell}(x,y)=(a y, -a x)=\smimatrix{0&a\\-a&0}\smimatrix{x\\y}$ is linear; we slightly abuse the notation to use $G_\ell$ to also denote the matrix $\smimatrix{0&a\\-a&0}$. Expanding the EG iterates \eqref{eq:eg-def}, we have
    \begin{align*}
        z_{t+\frac12} & = z_t-\gamma_t G_\ell z_t = (I-\gamma_t G_\ell) z_t,\\
        z_{t+1} & = z_t-\eta_t G_\ell z_{t+\frac12} = z_t - \eta_t G_\ell(I-\gamma_t G_\ell) z_t = (I - \eta_t G_\ell + \eta_t \gamma_t G_\ell^2) z_t.\yestag\label{eq:opt-1x1-expand}
    \end{align*}
    Define the polynomial $f_t(x)=1-\eta_t x+\eta_t \gamma_t x^2$; we know that $z_{t+1}=f_t(G_\ell)z_t$. Applying this to all $t$ and telescoping the product yields
    \begin{align*}
        z_T = \rbrm[\bigg]{\prod_{t=0}^{T-1}f_t(G_\ell)}z_0
            = \rbrm[\bigg]{\prod_{t=0}^{T-1}f_t}(G_\ell)z_0.
    \end{align*}
    The product notation is well defined because polynomials of the same matrix commute. Left-multiplying $G_\ell$ on both sides, we have
    \begin{align*}
        G_\ell z_T = G_\ell \rbrm[\bigg]{\prod_{t=0}^{T-1}f_t}(G_\ell)z_0
        = p_T(G_\ell) z_0,
    \end{align*}
    where we define $p_T(x)=x\cdot \prod_{t=0}^{T-1}f_t(x)$.
    
    Define the matrix $J=\smimatrix{0&1\\-1&0}$. We can verify that $J^2=-I$. We know that $\Span\anglem{I,J}$ is a commutative division algebra with the following canonical isomorphism from $\dC$:
    \begin{align*}
        \iota : \qquad \dC &\iso \Span\anglem{I,J}\\
        x+yi &\mapsto xI+yJ = \smimatrix{x&y\\-y&x}
    \end{align*}
    Since $J^\trans=-J$, we know that $\iota(w)^\trans = \iota(\overline{w})$ for all $w\in \dC$. This implies that $\iota(w)^\trans \iota(w)=\iota(\overline w)\iota(w)=\iota(\overline{w}w)=\iota(\abs{w}^2)=\abs{w}^2I$; therefore, for any vector $v\in \dR^2$, we have $\norm{\iota(w)v}^2=v^\trans \iota(w)^\trans \iota(w) v = v^\trans \rbrm[\big]{\abs{w}^2 I}v=\abs{w}^2 \norm{v}^2$; in other words, $\norm{\iota(w)v}=\abs{w}\norm{v}$.
    
    Observe that $G_\ell=\iota(ai)$. Since $\iota$ is an algebra homomorphism, it commutes with polynomials, and specifically the polynomial $p_t$. Combining with \eqref{eq:opt-1x1-expand}, we have
    \begin{align*}
        \GN_\ell(z_T)=\norm{G_\ell z_T} 
        & = \norm{p_T(G_\ell)z_0} \\
        & = \norm{p_T(\iota(ai))z_0} \\
        & = \norm{\iota(p_T(ai))z_0} \\
        & = \abs{p_T(ai)}\cdot\norm{z_0} \\
        & = \absm[\bigg]{
            ai\cdot \prod_{t=0}^{T-1}f_t(ai)
        }\norm{z_0} \yestag\label{eq:gn-to-poly}
    \end{align*}
    Now, $\abs{ai}=\abs{a}$, and direct computation shows
    \begin{align*}
        \abs{f_t(ai)}
        & = \abs{1-\eta_t ai+\eta_t \gamma_t a^2 i^2} \\
        & = \abs{1-\eta_t \gamma_t a^2 - \eta_t ai} \\
        & = \sqrt{(1-\eta_t\gamma_t a^2)^2+(-\eta_t a)^2} \\
        & = \sqrt{1+\eta_t(\eta_t-2\gamma_t) a^2+\eta_t^2\gamma_t^2 a^4}.
    \end{align*}
    Substituting these back into \eqref{eq:gn-to-poly} finishes the proof.
\end{proof}

We now show that any biaffine min-max optimization can be decomposed to direct sums of \cref{lem:opt-1x1} instances. We first show that EG on \emph{biaffine} instances can be reduced to EG on \emph{bilinear} instances; then in \cref{lem:decompose} we will show that any bilinear instance can be decomposed to $1\times 1$ instances.

\begin{applemma}\label{lem:decompose-translation}
    Let $\ell$ define a $L$-smooth biaffine function on $\dR^n \times \dR^m$ with a saddle point $z^\star=(x^\star, y^\star)$. Then, there exists an isometric affine bijection $P : \dR^n \times \dR^m \iso \dR^n \times \dR^m$ and an $L$-smooth bilinear function $\ell' : \dR^n\times \dR^m \to \dR$, such that:
    \begin{itemize}
        \item  Let $\{z_t\}_{t \in \dN} = \{(x_t, y_t)\}_{t \in \dN}$ be any trajectory generated by EG with a stepsize sequence $\{\eta_t, \gamma_t\}_{t \in \dN}$ on the instance defined by $\ell$. Then, $\{P(z_t)\}_{t\in \dN}$ is a valid trajectory generated by EG with the same stepsize sequence on the instance defined by $\ell'$.
        \item $P$ preserves the gradient norm,
        \begin{align}
            \GN_\ell(z)=\GN_{\ell'}(P(z)). \label{eq:translate-gn}
        \end{align}
    \end{itemize}
\end{applemma}
\begin{proof}
    Because $(x^\star, y^\star)$ is a saddle point of $\ell$, we can rewrite the loss function as
    \begin{align}
        \ell(x,y)=(x-x^\star)^\trans A (y-y^\star) + c'\label{eq:biaffine-ne}
    \end{align}
    for some $c'\in \dR$. Define
    \begin{align*}
        \ell'(x,y)& \defeq x^\trans A y, \\
        P(x,y)& \defeq(x-x^\star, y-y^\star),
    \end{align*}
    The smoothness of $\ell'$ inherits that of $\ell$ because $\norm{A}\leq L$.
    The mapping $P$ is clearly isometric, biaffine, and bijective. Since $P$ is a translation, its total derivative is just the identity matrix $I$. We also have $P(z+\eta g)=z+\eta g-z^\star=P(z)+\eta g$.

    From \eqref{eq:biaffine-ne}, it is clear that $\ell(x,y)=\ell'(P(x,y))+c'$; taking the derivative with respect to $x$ and $y$ eliminates the constant $c'$ and shows $G_\ell(x,y)=G_{\ell'}(P(x,y))$. This immediately proves \eqref{eq:translate-gn}.

    Therefore, $P$ preserves affine operations (adding a multiple of a vector to a point) and calls to the gradient operator $G$. As EG only makes use of affine operations and calls to $G$, we know that $P$ preserves EG trajectories as well. Formally, if $\{z_t\}_{t\geq 0} = \{(x_t, y_t)\}_{t\geq 0}$ is a trajectory of EG on the instance $\ell$, we know that
    \begin{align*}
        z_{t+\frac12}&=z_t-\gamma_t G_\ell(z_t), & %
        z_{t+1}&=z_t-\eta_t G_\ell(z_{t+\frac 12}).
    \end{align*}
    applying $P$ on both sides of both equations gives
    \begin{align*}
        P(z_{t+\frac12})&=P(z_t-\gamma_t G_\ell(z_t))=P(z_t)-\gamma_t G_\ell(z_t)=P(z_t)-\gamma_t G_{\ell'}(P(z_t)), \\
        P(z_{t+1})&=P(z_t-\eta_t G_\ell(z_{t+\frac 12}))
        =P(z_t)-\eta_t G_\ell(z_{t+\frac 12})
        =P(z_t)-\eta_t G_{\ell'}(P(z_{t+\frac 12})).
    \end{align*}
    The final equations have the same structure as the definition of EG \eqref{eq:eg-def}. Therefore, $\{P(z_t)\}_{t\geq 0}$ is also a trajectory of EG on the instance $\ell'$.
\end{proof}

\begin{applemma}\label{lem:decompose}
    Let $\ell$ define a $L$-smooth bilinear function on $\dR^n \times \dR^m$. Then, there exists a number $r\in \dN$, a direct sum decomposition $\dR^n \times \dR^m = \calZ_0 \oplus (\calX_1\times \calY_1)\oplus\cdots\oplus (\calX_r\times\calY_r)$, and a family of bilinear functions $\{\ell_i : \calX_i \times \calY_i \to \dR\}_{i=1,\dots,r}$, such that:
    \begin{itemize}
        \item $\dim{\calX_1}=\dim{\calY_1}=\cdots=\dim{\calX_r}=\dim{\calY_r}=1$.
        \item For any $z=(x,y)\in \dR^n\times \dR^m$, we can decompose it with $z=z!0+(x!1,y!1)+\cdots+(x!r,y!r)$, where $z!0\in \calZ_0, x!1\in \calX_1,y_1\in \calY_1,\dots,x!r\in \calX_r, y!r\in \calY_r$. Furthermore, this decomposition preserves gradient norm, as
        \begin{align}
            \GN_\ell(z)^2=\sum_{i=1}^r \GN_{\ell_i}(z!i)^2, \label{eq:decompose-gn}
        \end{align}
        in which we write $z!i=(x!i,y!i)$.
        \item  For any trajectory $\{z_t\}_{t \in \dN} = \{(x_t, y_t)\in \dR^n\times \dR^m\}_{t \in \dN}$ generated by EG with a stepsize sequence $\{\eta_t, \gamma_t\}_{t \in \dN}$ on the instance $(\dR^n,\dR^m,\ell)$,
        if we decompose it as defined above, then we have
        \begin{enumerate}
            \item $z!0_t=z!0_0$ for all $t$;
            \item for each $i\in\{1,\dots,r\}$, $\{z!i_t\}_{t\in \dN}$ is a trajectory generated by EG with the same stepsize sequence on the instance $(\calX_i, \calY_i, \ell_i)$.
        \end{enumerate}
    \end{itemize}
\end{applemma}

\begin{proof}
    As $\ell$ is bilinear, we define $A$ to be its matrix, i.e., $\ell(x,y)=x^\trans A y$.
    Let $r\leq\min\{m,n\}$ be the rank of $A$. Let $A=\sum_{i=1}^r \sigma_i u_i v_i^\trans$ be the singular value decomposition (SVD) of $A$, and $\{u_1,\dots,u_n\}, \{v_1,\dots,v_m\}$ be two sets of orthonormal bases of $\dR^n$ and $\dR^m$, respectively. The singular values satisfy $\sigma_i\in (0, L]$.

    Define the following subspaces of $\dR^n\times \dR^m$: 
    \begin{itemize}
        \item $\calZ_0\defeq\Span\anglem{u_{r+1},\dots,u_n} \times \Span\anglem{v_{r+1},\dots,v_m}$.
        \item For each $i=1,2,\dots,r$, define $\calZ_i=\calX_i\times\calY_i$, where $\calX_i=\Span\anglem{u_i}$ and $\calY_i=\Span\anglem{v_i}$.
    \end{itemize}
    From the definition, we can see that the dimensionality conditions are satisfied. Because $\{u_1,\dots,u_n\}$ and $\{v_1,\dots,v_m\}$  are bases of $\dR^n$ and $\dR^m$, the following direct sum holds:
    \begin{align}
        \dR^n \times \dR^m= \calZ_0 \oplus\calZ_1\oplus\cdots\oplus\calZ_r\label{eq:invariant-subspace-decomposition}
    \end{align}
    The sum is also orthogonal, because the bases are orthogonal.

    We claim that \eqref{eq:invariant-subspace-decomposition} is an invariant subspace decomposition of $G_\ell$. We first show that $\calZ_i$ is an invariant subspace for $i=1,\dots,r$. If $(x,y)\in \calZ_i$, then $x\in \calX_i$ and $y\in \calY_i$, which means $x=pu_i$ and $y=qv_i$, which means
    \begin{align*}
        Ay=qAv_i&=q\rbrm[\Big]{\sum_{j=1}^r \sigma_ju_jv_j^\trans}v_i=q\sigma_i u_i v_i^\trans v_i=q\sigma_i u_i,\\
        A^\trans x=pA^\trans u_i&=p\rbrm[\Big]{\sum_{j=1}^r \sigma_jv_ju_j^\trans}u_i=p\sigma_i v_i u_i^\trans u_i=p\sigma_i v_i,
    \end{align*}
    due to orthogonality. From the definition of $G_\ell$, we see that
    \begin{align}
        G_{\ell}(x,y)=G_\ell(pu_i,qv_i)=(Ay,-A^\trans x)=(\sigma_i q u_i, -\sigma_i p v_i)\in \calX_i\times \calY_i=\calZ_i.\label{eq:Gl-subspace}
    \end{align}
    Further, this proves that  $G_\ell\big|_{\calZ_i}$, the restriction of $G_\ell$ to $\calZ_i$, is represented with the matrix $\smimatrix{0&\sigma_i\\-\sigma_i&0}$ with respect to the basis $\{(u_i,0),(0,v_i)\}$. This matrix is precisely the matrix of $G_{\ell_i}$ if we define $\ell_i : \calX_i\times\calY_i\to\dR, \ell_i(pu_i,qv_i)=pq\sigma_i$; so $G_\ell\big|_{\calZ_i}=G_{\ell_i}$.

    We then show that $\calZ_0$ is also invariant. If $i>r$, from the property of SVD, we know that
    \begin{align*}
        Av_i=0 & \implies G_\ell(0,v_i)=(Av_i,-A^\trans 0)=0,\\
        u_i^\trans A=0 & \implies G_\ell(u_i,0)=(A0, -A^\trans u_i)=0.
    \end{align*}
    These two equations imply that $G_\ell$ nullifies the spanning vectors of $\calZ_0$, which means $G_\ell(\calZ_0)=\{0\}\subseteq \calZ_0$. Therefore,
    \eqref{eq:invariant-subspace-decomposition}
    is an invariant subspace decomposition of the linear endomorphism $G_\ell$. If a vector $z$ decomposes into $z=z!0+z!1+\cdots+z!r=z_0+(x!1,y!1)+\cdots+(x!r,y!r)$, we have
    \begin{align*}
        G_\ell(z)
        & = G_\ell(z!0)+\sum_{i=1}^r G_\ell(z!i)\tag{linearity}\\
        & = \sum_{i=1}^r G_\ell(z!i) \tag{$G_\ell$ nullifies $\calZ_0$}\\
        & = \sum_{i=1}^r G_{\ell_i}(z!i). \tag{$G_\ell\big|_{\calZ_i}=G_{\ell_i}$}
    \end{align*}
    Since each $G_{\ell_i}(z!i)\in \calZ_i$, the last sum is orthogonal, and the norm of the sum decomposes as
    \begin{align*}
        \normm[\big]{G_\ell(z)}^2
        = \sum_{i=1}^r \normm[\big]{G_{\ell_i}(z!i)}^2,
    \end{align*}
    which is exactly \eqref{eq:decompose-gn}.
    
    As EG only makes use of linear operations (addition and scalar multiplication) and calls to $G_\ell$, the EG formulas decompose into each subspace. Formally, this is more easily verified with the expanded definition \eqref{eq:opt-1x1-expand} of EG,
    \begin{align*}
        z!0_{t+1}+\sum_{i=1}^r z!i_{t+1}
        & = z_{t+1} \\
        & = (I - \eta_t G_\ell + \eta_t \gamma_t G_\ell^2) z_t \\
        & = (I - \eta_t G_\ell + \eta_t \gamma_t G_\ell^2) \rbrm[\bigg]{
            z!0_{t}+\sum_{i=1}^r z!i_{t}
        } \\
        & = (I - \eta_t G_\ell + \eta_t \gamma_t G_\ell^2) z!0_{t} + 
        \sum_{i=1}^r (I - \eta_t G_\ell + \eta_t \gamma_t G_\ell^2) z!i_{t}\\
        & = z!0_{t} + 
        \sum_{i=1}^r (I - \eta_t G_\ell + \eta_t \gamma_t G_\ell^2) z!i_{t} \tag{$G_\ell z!0_{t}=0$}\\
        & = z!0_{t} + 
        \sum_{i=1}^r (I - \eta_t G_{\ell_i} + \eta_t \gamma_t G_{\ell_i}^2) z!i_{t}. \tag{$G_\ell\big|_{\calZ_i}=G_{\ell_i}$}
    \end{align*}
    The final equation can be decomposed with \eqref{eq:invariant-subspace-decomposition}, and we have
    \begin{align*}
        z!0_{t+1}&=z!0_{t},\\
        \forall i\in \{1,\dots,r\},
        z!i_{t+1}&=(I - \eta_t G_{\ell_i} + \eta_t \gamma_t G_{\ell_i}^2) z!i_{t}.
    \end{align*}
    The final two statements of the lemma holds true because this holds for any $n$.
\end{proof}

We are finally ready to prove \cref{lem:opt}, restated here for convenience:
\lemOpt*
\begin{proof}
    Let $\{\eta_t,\gamma_t\}_{t\in \dN}$ be an arbitrary sequence of stepsizes. Assume $T\in \dN$. Define
    \begin{align*}
        H = \max_{a\in [0,L]} \cbrm[\bigg]{
            a \prod_{t=0}^{T-1} \sqrt{1+\eta_t(\eta_t-2\gamma_t) a^2+\eta_t^2\gamma_t^2 a^4}
        }.
    \end{align*}
    The maximum exists because the maximized function is continuous on $[0, L]$.

    We first apply \cref{lem:decompose-translation}. Assume $\{z_t\}_{t\in \dN}$ is a trajectory of EG on the instance $(\dR^n,\dR^m,\ell)$, we have $\{P(z_t)\}_{t\in \dN}$ is a trajectory of EG on $(\dR^n,\dR^m,\ell')$ where $\ell'$ is bilinear.

    We then apply the decomposition of \cref{lem:decompose} on $\{P(z_t)\}$ and decompose $\ell'$ into $r$ instances $\ell'_1,\dots,\ell'_r$, we have
    \begin{align*}
        P(z_t)=z!0_t+z!1_t+\cdots+z!r_t.
    \end{align*}
    Now, for each $i=1,2,\dots,r$, from \cref{lem:opt-1x1} we have
    \begin{align*}
        \GN_{\ell_i'}(z!i_T)=\normm[\big]{z!i_0}
                \cdot \sigma_i \prod_{t=0}^{T-1} \sqrt{1+\eta_t(\eta_t-2\gamma_t) \sigma_i^2+\eta_t^2\gamma_t^2 \sigma_i^4}\leq \normm[\big]{z!i_0}H.
    \end{align*}
    Squaring both sides, summing across $i=1,2,\dots,r$, and using \eqref{eq:decompose-gn}, we have
    \begin{align*}
        \GN_\ell(z_T)^2 
        & = \GN_{\ell'}(P(z_T))^2
         = \sum_{i=1}^r \GN_{\ell_i'}(z!i_T)^2\\
        & \leq \sum_{i=1}^r \normm[\big]{z!i_0}^2H^2 \\
        & \leq H^2 \cdot \rbrm[\Big]{
            \normm[\big]{z!0_0}^2
            +
            \sum_{i=1}^r \normm[\big]{z!i_0}^2
        } \\
        & = H^2 \normm[\big]{P(z_0)}^2 \\
        & = H^2 \normm[\big]{z_0-z^\star}^2.
    \end{align*}
    Taking the square root of both sides proves the first claim of the lemma.

    For the second claim on the tightness, let
    \begin{align*}
        a^\star = \argmax_{a\in [0,L]} \cbrm[\bigg]{
            a \prod_{t=0}^{T-1} \sqrt{1+\eta_t(\eta_t-2\gamma_t) a^2+\eta_t^2\gamma_t^2 a^4}
        }.
    \end{align*}
    Applying \cref{lem:opt-1x1} on the instance $(\dR,\dR,\ell_\star(x,y)=a^\star xy)$ directly proves the tightness claim.
\end{proof}

\section{Supporting proofs on power-law stepsizes in EG}\label{ah:eg-cont}

In this section, we provide an analysis of Pareto stepsize schedules for EG.

\subsection{On single-stepsize EG}

We show our proofs to single-stepsize EG first.
As a recap, we have
\begin{align*}
    \ln \calR_T(\calE, a)
    & \defeq
    \EE_{\eta_0,\dots,\eta_{T-1}\sim \calE}
    \sbrm[\Bigg]{
    \ln\rbrm[\bigg]{
        a \prod_{t=0}^{T-1} \sqrt{1-\eta_t^2 a^2+\eta_t^4 a^4}
    }} \\
    & =
    {
        \ln a + \frac T2\EE_{\eta\sim \calE}\sbrm[\big]{\ln(1-\eta^2a^2+\eta^4a^4)}
    },\tag{\ref{eq:calL-def}, again}\label{eq:calL-def'}
\end{align*}
and, when $\calE=\Pareto(\etam/L, \beta)$,
\begin{align*}
    \EE_{\eta\sim \calE}\sbrm[\big]{\ln(1-\eta^2a^2+\eta^4a^4)}
    & = \frac{\beta}{2}\rbrm[\Big]{ \frac{\etam a}{L} }^{\beta} \int_{(\etam a/L)^2}^{\infty} \ln(1-t+t^2)t^{-\beta/2-1}\dd{t},\tag{\ref{eq:pareto-contraction-int}, again}\label{eq:pareto-contraction-int'}
\end{align*}

Define the integral in the RHS as $I_\beta((\etam a/L)^2)$, where $I_\beta(s)=\int_{s}^{\infty} \ln(1-t+t^2)t^{-\beta/2-1}\dd{t}$. In order for the log-contraction-ratio to be smaller than 0 on average, we require $I_\beta((\etam a/L)^2)<0$.
The derivative $I'_\beta(s)=-\ln(1-s+s^2)s^{-\beta/2-1}$ shares the same sign as $-\ln(1-s+s^2)$; therefore, $I_\beta(s)$ is monotonically increasing on $[0, 1]$ and decreasing on $[1, \infty)$. Because $I_\beta(+\infty)=0$, any solution of $I_\beta(s)<0$ must lie in $[0, 1)$, and the existence of such a solution requires $I_\beta(0)<0$. The integration $I_\beta(0)$ is the Mellin transform of the function $\ln(1-t+t^2)$, and solved in \cref{lem:m-func} as
\begin{align*}
    I_\beta(0)=\int_0^\infty \ln(1-t+t^2) t^{-\beta/2-1}\dd{t}=\frac{4\pi \cos(\pi \beta/3)}{\beta\sin(\pi \beta/2)}.
\end{align*}
The inequality $I_\beta(0)<0$ holds when $\beta>\frac 32$. For all values of $\beta$ strictly larger than $\frac32$,
we have
\begin{align*}
    I_\beta(s)
    & = I_\beta(0)
        - \int_0^s \ln(1-t+t^2) t^{-\beta/2-1}\dd{t} \\
    & \leq I_\beta(0)
        + \int_0^s t \cdot t^{-\beta/2-1}\dd{t} \tag{$-\ln(1-t+t^2)\leq t$ whenever $t\geq 0$}\\
    & = I_\beta(0)+\int_0^s t^{-\beta/2}\dd{t} = I_\beta(0)+\frac{s^{1-\beta/2}}{1-\beta/2}.\yestag\label{eq:Ibeta-bound}
\end{align*}
Therefore, for any $\beta>\frac32$, we have $I_\beta(0)<0$, and it is possible to choose a sufficiently small $\etam$ such that $I_\beta((\etam a/L)^2)<0$ for all $a\in [0, L]$. We are now ready to prove the lemma, restated here:

\lemEgCont*
\begin{proof}
    Assume $\beta\in (\frac32, 2)$. Let $\etam=\rbrm[\big]{-(1-\frac\beta2)\frac{I_\beta(0)}{2}}^{1/(2-\beta)}>0$; the value satisfies the following equality:
    \begin{align}
        \frac{(\etam^2)^{1-\beta/2}}{1-\beta/2}=-\frac{I_\beta(0)}{2}.\label{eq:step:eg-cont-proof-first}
    \end{align}
    Therefore, for all $a\in [0, L]$, $(\etam a/L)^2\in [0, \etam^2]$, and we have $I_\beta((\etam a/L)^2)\leq I_\beta(\etam^2)\leq I_\beta(0)-\frac{I_\beta(0)}{2}=\frac12 I_\beta(0)<0$ from \eqref{eq:Ibeta-bound}. Then, from (\hyperref[eq:pareto-contraction-int']{\ref*{eq:pareto-contraction-int}}), we know that
    \begin{align*}
        \EE_{\eta\sim \calE}\sbrm[\big]{\ln(1-\eta^2a^2+\eta^4a^4)}
        \leq
        \frac{\beta}{2}\rbrm[\Big]{ \frac{\etam a}{L} }^{\beta}
        \rbrm[\Big]{\frac{I_\beta(0)}{2}}.
    \end{align*}
    Combining this with (\hyperref[eq:calL-def']{\ref*{eq:calL-def}}), we have
    \begin{align*}
        \ln \calR_T(\calE, a)
        & \leq {
            \ln L + \ln \frac{a}{L}
            +
            \frac T2 
            \frac{\beta}{2}\rbrm[\Big]{ \frac{\etam a}{L} }^{\beta}
            \rbrm[\Big]{\frac{I_\beta(0)}{2}}
        } \\
        & = \ln L + {
            \ln b
            +
            \frac T2 
            \frac{\beta}{2}\rbrm{ \etam b }^{\beta}
            \rbrm[\Big]{\frac{I_\beta(0)}{2}}
        } \tag{substitute $b=\frac aL$}\\
        & = \ln L + \rbrm[\Big]{
            \ln b
            +
            \frac {T \beta \etam^\beta I_\beta(0)}{8}
            b^\beta
        } \\
        & \leq
        \ln L - \frac{1}{\beta}\rbrm[\Big]{
            \ln \frac {-T \beta^2 \etam^\beta I_\beta(0)}{8} +1
        } \tag{$\star$}\label{step:lem-eg-cont-1st-order} \\
        & = \ln L - \frac{1}{\beta} \ln T - \frac{1}{\beta}\rbrm[\Big]{
            \ln \frac {\beta^2 \etam^\beta (-I_\beta(0))}{8} +1
        }.\yestag\label{eq:lem-eg-cont-final}
    \end{align*}
    In the above, we substitute $b=\frac aL$, and use \cref{lem:ln-plus-monomial} for the step \eqref{step:lem-eg-cont-1st-order}. The first statement of the lemma is proved if we take $C=\exp\sbrm[\big]{- \frac{1}{\beta}\rbrm[\big]{
            \ln \frac {\beta^2 \etam^\beta (-I_\beta(0))}{8} +1
        }}$.
    The second statement is a direct implication of \cref{lem:opt} and the definition of $\calR_T(\cdot, \cdot)$.
\end{proof}

The constant factor $C$ in the argument above is wildly suboptimal; to achieve a rate of $\calO(T^{-0.66})$ with $\beta=\frac{100}{66}=\frac 32+\frac{1}{66}$, the above proof sets $\etam\approx 0.00042$ and $C\approx 8331$. Symbolic computation suggests that $C\sim \eps^{-8/3}$.
The suboptimality comes from the fact that $\etam$ is too small; indeed, using any stepsize below $\eta_*=\frac{1}{\sqrt2 L}$ is unwise, because the contraction ratio suffers as
\begin{align*}
    1-\eta^2a^2+\eta^4 a^4> 1-\eta_*^2 a^2 + \eta_*^4 a^4
\end{align*}
for all $\eta<\eta_*$, and any $a\in (0, L]$. We can significantly improve the constant factor by truncating the Pareto distribution at $\eta_*$ and using a mixture between Pareto and a point mass (see  \cref{fig:vdc-schedule-mixture} for an illustration). This can be formulated in the following lemma.

\begin{figure}[tb]
    \centering
    \includegraphics[width=0.92\linewidth]{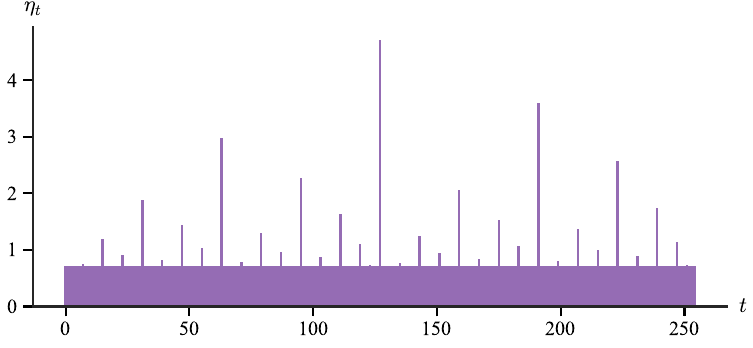}
    \caption{The discretized power-law stepsize schedule $\{\eta_0, \eta_1, \eta_2, \dots\}$ to achieve $\calO(T^{-0.66})$ in \cref{lem:eg-mixture}. Only the first 255 terms are shown. The majority of iterations have $\eta_t=\etam=1/\sqrt 2$.}
    \label{fig:vdc-schedule-mixture}
\end{figure}

\begin{applemma}\label{lem:eg-mixture}
    For any $\beta\in (\frac 32, 2)$, there exists $p\in (0, 1), C\in \dR_+$, and distribution $\calE=p\cdot \Pareto(\eta_*, \beta)+(1-p)\cdot \delta_{\eta_*}$ where $\eta_*=\frac{1}{\sqrt{2}L}$, such that
    \begin{align*}
        \calR_T(\calE, a)\leq CLT^{-1/\beta},
    \end{align*}
    for all $a\in (0, L]$.
\end{applemma}
\begin{proof}
    Since $\calE$ is a mixture, we have, similar to (\hyperref[eq:pareto-contraction-int']{\ref*{eq:pareto-contraction-int}}),
    \begin{align*}
    & \phantom{{}={}} \EE_{\eta\sim \calE}\sbrm[\big]{\ln(1-\eta^2a^2+\eta^4a^4)} \\
    & = p \EE_{\eta\sim \Pareto(\eta_*,\beta)}\sbrm[\big]{\ln(1-\eta^2a^2+\eta^4a^4)}
        + (1-p) \ln(1-\eta_*^2a^2+\eta_*^4a^4) \\
    & = p\cdot \frac{\beta}{2}\rbrm[\Big]{ \frac{\etam a}{L} }^{\beta} I_\beta((\etam a/L)^2) 
      + (1-p) \ln\sbrm[\Big]{1-\etam^2\rbrm[\big]{\frac aL}^2+\etam^4\rbrm[\big]{\frac aL}^4}
    \end{align*}
    where, for notational consistency, we define $\etam=\frac{1}{\sqrt 2}$. If we substitute $b=\frac aL$ and apply \eqref{eq:Ibeta-bound}, we have
    \begin{align*}
    & \phantom{{}={}} 
        \EE_{\eta\sim \calE}\sbrm[\big]{\ln(1-\eta^2a^2+\eta^4a^4)} \\
    & =
        p\cdot \frac{\beta}{2}\rbrm{ \etam b }^{\beta} I_\beta((\etam b)^2) 
        + (1-p) \ln\sbrm[\big]{1-\etam^2b^2+\etam^4b^4} \\
    & \leq 
        p\cdot \frac{\beta}{2}\rbrm{ \etam b }^{\beta} \rbrm[\Big]{
            I_\beta(0)
            +
            \frac{(\etam b)^{2-\beta}}{1-\beta/2}
        }
        + (1-p) \ln\sbrm[\big]{1-\etam^2b^2+\etam^4b^4} \\
    & \leq
        \frac{p\beta\cdot I_\beta(0)}{2}\rbrm{ \etam b }^{\beta}
        + p\cdot \frac{\beta}{2-\beta}\rbrm{ \etam b }^2
        + (1-p) \ln\sbrm[\big]{1-\etam^2b^2+\etam^4b^4}.
        \yestag\label{eq:eg-mixture-contraction}
    \end{align*}
    The function $\ln(1-z+z^2)$ is convex on $[0, \frac12]$, has a value of $0$ when $z=0$, and has a value of $-\ln \frac43<-\frac 14$ when $z=\frac12$; therefore, $\ln(1-z+z^2)\leq -\frac z2$ on $[0, \frac12]$. Because $\etam^2b^2\in [0, \frac12]$, we have
    \begin{align*}
        p\cdot \frac{\beta}{2-\beta}\rbrm{ \etam b }^2
        + (1-p) \ln\sbrm[\big]{1-\etam^2b^2+\etam^4b^4}
        \leq
        \rbrm[\Big]{
            p\cdot \frac{\beta}{2-\beta}-(1-p)\cdot \frac12
        } (\etam b)^2,
    \end{align*}
    and this value is nonpositive when $p\leq \frac{2-\beta}{2+\beta}$. Combining this with \eqref{eq:eg-mixture-contraction} and (\hyperref[eq:calL-def']{\ref*{eq:calL-def}}), we have
    \begin{align*}
        \ln\calR_T(\calE, a)
        & \leq
        \ln L + {
            \ln b
            +
            \frac T2 
            \frac{p\beta\cdot I_\beta(0)}{2}\rbrm{ \etam b }^{\beta}
        }\\
        & \leq
        \ln L - \frac{1}{\beta}\rbrm[\Big]{
            \ln \frac {T p \beta^2 \etam^\beta (- I_\beta(0))}{4} +1
        } \tag{$\star$}\label{step:lem-eg-mixture-1st-order}\\
        & \leq
        \ln L - \frac{1}{\beta}\ln T
        - \frac{1}{\beta}\rbrm[\Big]{
            \ln \frac {p \beta^2 \etam^\beta (- I_\beta(0))}{4} +1
        }.
    \end{align*}
    We again use \cref{lem:ln-plus-monomial} for the step \eqref{step:lem-eg-mixture-1st-order}. The lemma is proved when we take $p= \frac{2-\beta}{2+\beta}$ and $C=\exp\sbrm[\big]{- \frac{1}{\beta}\rbrm[\big]{
            \ln \frac {p \beta^2 \etam^\beta (- I_\beta(0))}{4} +1
        }}$.
\end{proof}

\begin{appremark}
    To achieve the same $\calO(T^{-0.66})$ rate with $\beta=\frac 32+\frac 1{66}$, \cref{lem:eg-mixture} sets $C\approx 11.6$, significantly lower than that of \cref{lem:eg-cont}. Symbolic computation suggests that $C\sim \eps^{-2/3}$.
\end{appremark}

We then show the result of discretizing the distribution into a stepsize sequence using the van der Corput sequence. As the technique for discretization is quite convoluted and independent of the rest of the paper, it is detailed separately in \cref{ah:discretization}.

\lemDiscretization*
\begin{proof}
    The first claim is a special case of \cref{thm:pareto-qmc-discretize}, obtained by setting $\psi=\pi/3$ and $p=1$.
    The second claim is obtained by plugging the first claim into the formulas of $\calR_T$ (see \eqref{eq:eg-minmax}, \eqref{eq:calL-def}).
\end{proof}

\subsection{On the lower bound of single-stepsize EG}\label{ah:lower-bound}

In this part of the appendix, we will prove \cref{thm:lb-eg}. The proof again uses the one-dimensional formulation in \eqref{eq:eg-minmax}. From \cref{lem:opt-1x1}, for a scalar bilinear instance $\ell(x,y)=a xy$ with $a\in [0,1]$, single-stepsize EG satisfies
\begin{align}
    \ln \GN(z_T)
    =
    \ln R+
    \ln a+\frac12\sum_{t=0}^{T-1}\ln(1-a^2\eta_t^2+a^4\eta_t^4),
    \label{eq:lb-log-gn-scalar}
\end{align}
provided that $\norm{z_0}=R$. Thus it suffices to prove a lower bound on the right-hand side.

We first show the minimax lower bound on this scalar expression.

\begin{applemma}\label{lem:eg-lb-meta}
    For any $\eps>0$, there exists a constant $C>0$ such that the following holds. For any $T\geq 3$ and any distribution $\calE$ supported on $\dR_+$,
    \begin{align*}
        \sup_{a\in (0,1]}
        \cbrm[\Big]{
            \ln a+\frac T2\EE_{\eta\sim \calE}\sbrm[\big]{\ln(1-a^2\eta^2+a^4\eta^4)}
        }
        \geq
        - \rbrm[\Big]{\frac23 + \eps} \ln T + \ln C 
        =
        \ln\sbrm[\big]{C T^{-\frac23-\eps}}.
    \end{align*}
\end{applemma}
\begin{proof}
    It suffices to prove the result for $0<\eps<\frac1{12}$. Let $\beta=(\frac23+\eps)^{-1}\in (\frac 43, \frac 32)$.

    Define $a_* = T^{-1/\beta}$. Consider a Pareto distribution $\calA$ with shape $\beta$ truncated on $[a_*, 1]$. Specifically, we have for $x\in [a_*, 1]$,
    \begin{align*}
        \dd{}\calA(x)=\frac 1V x^{-\beta-1}\dd{x},
    \end{align*}
    where $V$ is the normalizing constant $V=\int_{a_*}^1 a^{-\beta-1}\dd{a}=(a_\star^{-\beta}-1)/{\beta}=(T-1)/{\beta}$.

    In the next steps, we will prove that
    \begin{align*}
        \EE_{a\sim \calA}
        \sbrm[\Big]{
            \ln a+\frac T2\EE_{\eta\sim \calE}\sbrm[\big]{\ln(1-a^2\eta^2+a^4\eta^4)}
        }
        \geq
        \ln C - \rbrm[\Big]{\frac23 + \eps} \ln T.
    \end{align*}

    We first analyze $\EE_{a\sim \calA}[\ln a]$. We have
    \begin{align*}
        \EE_{a\sim \calA}[\ln a]
        & = \frac{1}{V} \int_{a_*}^{1} a^{-\beta-1}\ln a\,\dd{a} \\
        & = \frac{\beta}{a_*^{-\beta}-1}\rbrm[\Big]{-\frac{a^{-\beta}}{\beta}\ln a-\frac{a^{-\beta}}{\beta^2}}\Big|_{a_*}^{1} \tag{$\star$}\label{step:lb-int-by-parts}\\
        & = \frac{1}{a_*^{-\beta}-1}\rbrm[\Big]{-\frac{1}{\beta} +{a_*^{-\beta}}\ln a_*+\frac{a_*^{-\beta}}{\beta}} \\
        & = \frac{1}{\beta}-\frac{T}{(T-1)\beta}\ln T,\yestag\label{eq:ln-a-ln-T}
    \end{align*}
    where in \eqref{step:lb-int-by-parts} we used integration by parts:
    \begin{align*}
        \int a^{-\beta-1}\ln a\,\dd{a}
        & =\frac{1}{-\beta}\int \ln a\,\dd{}a^{-\beta} \\
        & = \frac{1}{-\beta}\rbrm[\Big]{a^{-\beta}\ln a - \int a^{-\beta} \dd{}\ln a} \\
        & = \frac{1}{-\beta}\rbrm[\Big]{a^{-\beta}\ln a - \int a^{-\beta-1} \dd{a}}=-\frac{a^{-\beta}}{\beta}\ln a-\frac{a^{-\beta}}{\beta^2}.
    \end{align*}

    Now we turn our attention to $\frac T 2 \EE_{a\in\calA}[\ln(1-a^2\eta^2+a^4\eta^4)]$. We have
    \begin{align*}
        \frac T 2 \EE_{a\in\calA}[\ln(1-a^2\eta^2+a^4\eta^4)]
        & = \frac{T}{2V} \int_{a_*}^1 \ln(1-a^2\eta^2+a^4\eta^4) a^{-\beta-1} \dd{a}\\
        & = \frac{T}{4V} \eta^\beta \int_{a_*^2 \eta^2}^{\eta^2}\ln(1-t+t^2) t^{-\beta/2-1} \dd{t},\yestag\label{eq:s-integral}
    \end{align*}
    where we substituted $t=a^2\eta^2$, $a=\eta^{-1}t^{1/2}$, and $a^{-\beta-1} \dd{a}=\rbr{\eta^{-1}t^{1/2}}^{-\beta-1}\cdot \rbr{\eta^{-1}\cdot \frac12 t^{-1/2}\dd{t}}=\frac12 \eta^\beta t^{-\beta/2-1}\dd{t}$.

    We use case-splitting on $\eta$ to lower-bound
    \begin{align*}
        S\defeq \eta^\beta \int_{a_*^2 \eta^2}^{\eta^2}\ln(1-t+t^2) t^{-\beta/2-1} \dd{t}
        = \eta^\beta \rbrm[\big]{
            I_\beta(a_*^2 \eta^2) - I_\beta(\eta^2)
        }.
    \end{align*}
    \begin{itemize}
        \item When $0<\eta\leq 1$, we use $\ln(1-t+t^2)\geq-t$ for $t\geq0$. Indeed, $\ln(1-t+t^2)+t$ vanishes at zero and has derivative $(t+t^2)/(1-t+t^2)\geq0$. Thus,
        \begin{align*}
            S&\geq-\eta^\beta\int_0^{\eta^2}t^{-\beta/2}\dd{t}
            =-\frac{2\eta^2}{2-\beta}
            \geq-\frac{2}{2-\beta}.
        \end{align*}
        \item When $\eta\geq \frac{1}{a_*}=T^{1/\beta}$, we know that $t=\eta^2 a^2\geq\eta^2 a_*^2\geq 1$, so $\ln(1-t+t^2)\geq 0$, and thus $S\geq 0$.
        \item The most complicated case is $1<\eta<\frac{1}{a_*}$. Since $\eta^2 a_*^2<1$, we know that $I_\beta(0)<I_\beta(\eta^2 a_*^2)$ due to the monotonicity of $I_\beta$ (see the discussion below \eqref{eq:pareto-contraction-int'}). Therefore,
        \begin{align*}
            S & > \eta^\beta \rbrm[\big]{
                I_\beta(0) - I_\beta(\eta^2)
            } \\
            & = \eta^\beta I_\beta(0) - \eta^\beta \int_{\eta^2}^{\infty} \ln(1-t+t^2) t^{-\beta/2-1} \dd{t}.\yestag\label{eq:lem-lb-int-s-lb}
        \end{align*}
        For $t\geq \eta^2>1$, we have $t-1>0$, so $\ln(1-t+t^2)<\ln t^2=2\ln t$, so
        \begin{align*}
            \int_{\eta^2}^{\infty} \ln(1-t+t^2) t^{-\beta/2-1} \dd{t}
            & \leq
            2 \int_{\eta^2}^{\infty} \ln t \cdot t^{-\beta/2-1} \dd{t} \\
            & = 2 \rbrm[\Big]{-\frac{t^{-\beta/2}}{\beta/2}\ln t-\frac{t^{-\beta/2}}{(\beta/2)^2}}\Big|_{t=\eta^2}^{t=+\infty} \\
            & = 2 \rbrm[\Big]{0+\frac{\eta^{-\beta}}{\beta/2}\cdot 2\ln \eta+\frac{\eta^{-\beta}}{(\beta/2)^2}}\\
            & = 8\eta^{-\beta}\rbrm[\Big]{\frac{1}{\beta}\ln \eta+\frac{1}{\beta^2}}.
        \end{align*}
        Substituting this back to \eqref{eq:lem-lb-int-s-lb} and using \cref{lem:ln-plus-monomial}, we have
        \begin{align*}
            S & \geq \eta^\beta I_\beta(0) - \frac{8}{\beta}\ln \eta-\frac{8}{\beta^2} \\
              & \geq \frac{8}{\beta^2}\sbrm[\Big]{
                \ln \frac{I_\beta(0)\beta^2}{8}+1
              } -\frac{8}{\beta^2} \\
              & = \frac{8}{\beta^2}\ln \frac{I_\beta(0)\beta^2}{8}.
        \end{align*}
    \end{itemize}
    Combining the three cases, define the negative constant
    \begin{align*}
        D_\beta=\min\cbrm[\Big]{-\frac{2}{2-\beta},\frac{8}{\beta^2}\ln\frac{I_\beta(0)\beta^2}{8}}<0.
    \end{align*}
    Here $I_\beta(0)>0$ for $\beta\in(\frac43,\frac32)$ by \cref{lem:m-func}. For every $\eta>0$, we have $S\geq D_\beta$. Substituting into \eqref{eq:s-integral} gives
    \begin{align*}
        \frac T2\EE_{a\sim\calA}[\ln(1-a^2\eta^2+a^4\eta^4)]
        =\frac{T}{4V}S
        \geq\frac{\beta T}{4(T-1)}D_\beta.
    \end{align*}
    Therefore,
    \begin{align*}
        & \phantom{{}={}}
        \EE_{a\sim \calA}
        \sbrm[\big]{
            \ln a+\frac T2\EE_{\eta\sim\calE}[\ln(1-a^2\eta^2+a^4\eta^4)]
        } \\
        & = \EE_{a\sim \calA}[\ln a] 
          + \frac T2\EE_{\eta\sim \calE, a\sim \calA}
          \sbrm[\big]{
            \ln(1-a^2\eta^2+a^4\eta^4)
          } \\
        & \geq 
            \frac{1}{\beta}-\frac{T}{\beta(T-1)}\ln T
            +
            \frac{\beta T}{4(T-1)}D_\beta \\
        & =
           -\frac{1}{\beta}\ln T
           +\frac{1}{\beta}\rbrm[\Big]{1-\frac{\ln T}{T-1}}
           +\frac{\beta T}{4(T-1)}D_\beta.
    \end{align*}    
    Since $T\geq3$, we have $\ln T\leq T-1$ and $T/(T-1)\leq3/2$. As $D_\beta<0$, this implies
    \begin{align*}
        & \phantom{{}={}}
        \EE_{a\sim \calA}
        \sbrm[\big]{
            \ln a+\frac T2\EE_{\eta\sim\calE}[\ln(1-a^2\eta^2+a^4\eta^4)]
        } \\
        & \geq -\frac{1}{\beta}\ln T+\frac{3\beta}{8}D_\beta.
    \end{align*}
    Define $C=\exp(3\beta D_\beta/8-1)>0$.
    Since $\sup_a\{\cdot\}\geq \EE_a[\cdot]$, we know that
    \begin{align*}
        &
        \sup_{a\in [a_*, 1]}
        \cbrm[\big]{
            \ln a+\frac T2\EE_{\eta\sim\calE}[\ln(1-a^2\eta^2+a^4\eta^4)]
        } \\
        & \geq
        \EE_{a\sim \calA}
        \sbrm[\big]{
            \ln a+\frac T2\EE_{\eta\sim\calE}[\ln(1-a^2\eta^2+a^4\eta^4)]
        } \\
        & > -\frac{1}{\beta}\ln T + \ln C
          = \ln\rbrm[\big]{CT^{-2/3-\eps}}.
    \end{align*}
    And thus the lemma is proven. In addition, because the supremum is strictly larger than $\ln(CT^{-2/3-\eps})$, there must be an $a\in [a_*, 1]\subseteq (0, 1]$ such that the expression is also strictly larger than $\ln(CT^{-2/3-\eps})$ as well.
\end{proof}

We are now ready to prove our main lower-bound theorem.

\begin{proof}[Proof of \cref{thm:lb-eg}]
    For any $\eps>0$, let $\eps_0=\min\{\eps,1/24\}$ and take the constant $C>0$ in \cref{lem:eg-lb-meta} with $\eps_0$. Since $T^{-2/3-\eps_0}\geq T^{-2/3-\eps}$ for $T\geq1$, this constant also gives the desired exponent $2/3+\eps$.

    Suppose $T\geq 3$ and $\calE$ is a distribution on $\calR_{+}$. From \cref{lem:eg-lb-meta}, there exists a number $a\in (0, 1]$ such that
    \begin{align*}
        \ln a+\frac T2\EE_{\eta\sim\calE}[\ln(1-a^2\eta^2+a^4\eta^4)]
        >
        \ln\rbrm[\big]{CT^{-2/3-\eps}}.
    \end{align*}
    By \eqref{eq:lb-log-gn-scalar} and linearity of expectation, the LHS is exactly $\EE[\ln \GN(z_T)]-\ln R$ if EG is initialized at $\norm{z_0}=R$, draws each stepsize $\gamma_i=\eta_i\sim \calE$ for all $i$, and runs on the instance $(\dR,\dR,\ell(x,y)=axy)$. This proves the first claim.

    For the second claim, let $T\geq 3$ and $\{\gamma_t=\eta_t\}_{t\in[T]}$ is an arbitrary stepsize sequence. Define $\calE_T=\frac1T\sum_{t=0}^{T-1}\delta_{\eta_t}$ is the empirical distribution of the stepsize sequence, and from \cref{lem:eg-lb-meta}, there exists a number $a\in (0, 1]$ such that
    \begin{align*}
        \ln\rbrm[\big]{CT^{-2/3-\eps}}
        & <
        \ln a+\frac T2\EE_{\eta\sim\calE_T}[\ln(1-a^2\eta^2+a^4\eta^4)] \\
        & = \ln a + \frac T2 \frac 1T \sum_{t=0}^{T-1}\ln(1-a^2\eta_t^2+a^4\eta_t^4),
    \end{align*}
    which is exactly $\ln\GN(z_T)-\ln R$ if EG initializes at $\norm{z_0}=R$, uses the stepsize sequence $\{\gamma_t=\eta_t\}_{t\in [T]}$, and runs on the instance $(\dR,\dR,\ell(x,y)=axy)$.
\end{proof}

\subsection{On double-stepsize EG}\label{ah:dseg}

Similar to \cref{h2:result-eg}, we consider an optimization problem related to minimizing \eqref{eq:lem-opt}. As we have discussed in \cref{h2:result-dseg}, we assume that $\gamma_t = \eta_t/\rho$ across all iterations $t$. Define  $\lambda_t=\sqrt{\gamma_t \eta_t}$; we have $\gamma_t=\lambda_t/\sqrt\rho$ and $\eta_t=\lambda_t \sqrt\rho$. In this reparametrization, we have
\begin{align*}
    1+\eta_t (\eta_t-2\gamma_t)a^2+\eta_t^2\gamma_t^2a^4
    =
    1+(\rho-2) \lambda_t^2a^2 + \lambda_t^4 a^4,
\end{align*}
and therefore minimizing \eqref{eq:lem-opt} is the same as the following optimization problem:

\begin{align}
    \minimize_{\rho,\{\lambda_t\}_{t\in [T]}}\; \maximize_{a\in [0,L]} \;
    \calR_{\rho,T}(\{\lambda_t\}_{t\in [T]},a)\defeq
    {
        a \prod_{t=0}^{T-1} \sqrt{1+(\rho-2)\lambda_t^2 a^2+\lambda_t^4 a^4}
    }.
    \label{eq:dseg-minmax}
\end{align}

It is clear that our previous $\calR_T$ for single-stepsize EG is a special case of $\calR_{\rho,T}$ with $\rho=1$. If $\lambda_0,\lambda_1,\dots$ are sampled from a distribution $\calE$, we again define $\calR_{\rho,T}(\calE, a)$ as the geometric mean, or formally,
\begin{align}
    \ln \calR_{\rho,T}(\calE, a)
    & \defeq
    \EE_{\lambda_0,\dots,\lambda_{T-1}\sim \calE}
    \sbrm[\Bigg]{
    \ln\rbrm[\bigg]{
        a \prod_{t=0}^{T-1} \sqrt{1+(\rho-2)\lambda^2a^2+\lambda^4a^4}
    }} \\
    & =
    {
        \ln a + \frac T2\EE_{\lambda\sim \calE}\sbrm[\big]{\ln(1+(\rho-2)\lambda^2a^2+\lambda^4a^4)}
    },\label{eq:calL-dseg}
\end{align}
and, when $\calE=\Pareto(\lambdam/L, \beta)$,
\begin{align}
    \EE_{\lambda\sim \calE}\sbrm[\big]{\ln(1+(\rho-2)\lambda^2a^2+\lambda^4a^4)}
    & = \frac{\beta}{2}\rbrm[\Big]{ \frac{\lambdam a}{L} }^{\beta} \int_{(\lambdam a/L)^2}^{\infty} \ln(1+(\rho-2)t+t^2)t^{-\beta/2-1}\dd{t},\label{eq:pareto-contraction-dseg}
\end{align}
where we use the substitution $t=\eta^2a^2$ in the last step. These two equations are double-stepsize versions of \eqref{eq:calL-def} and \eqref{eq:pareto-contraction-int}, respectively. We similarly define the integration in the RHS of \eqref{eq:pareto-contraction-dseg} as
\begin{align*}
    I_{\rho,\beta}(s)=\int_{s}^{\infty} \ln(1+(\rho-2)t+t^2)t^{-\beta/2-1}\dd{t},
\end{align*}
and we need to find $I_{\rho,\beta}(s)<0$ for sufficiently small $s$. Due to the monotonicity in $\rho$, this requires $\rho\to 0^+$; we assume $\rho\leq 2$, and thus $\ln(1+(\rho-2)t+t^2)$ is negative on $(0, 2-\rho)$ and positive everywhere else. The integral $I_{\rho,\beta}(0)$ is solved in \cref{lem:m-func}, as
\begin{align*}
    I_{\rho,\beta}(0)=\frac{4\pi \cos(\theta \beta/2)}{\beta\sin(\pi \beta/2)},
\end{align*}
where $2\cos\theta=(\rho-2)$, or $\theta=\arccos\frac{\rho-2}{2}$. When $\rho\to 0^+$, we have $\theta\to \pi^-$, and $I_{\rho,\beta}(0)=\frac{4\pi \cos(\theta \beta/2)}{\beta\sin(\pi \beta/2)}$ converges to $\frac{4\pi\cot(\pi\beta/2)}{\beta}$ from above. Solving $\frac{4\pi\cot(\pi\beta/2)}{\beta}$ yields $\beta>1$. Therefore, as a counterpart of \cref{lem:eg-cont}, we have

\begin{applemma}\label{lem:dseg-cont}
    For any $\beta\in (1, 2)$, there exists $\rho, \lambdam, C\in \dR_+$, and distribution $\calE=\Pareto(\lambdam/L, \beta)$, such that for all $a\in (0, L]$,
    \begin{align*}
        \calR_{\rho,T}(\calE,a)\leq CLT^{-1/\beta}.
    \end{align*}
    As a corollary, for any $L$-smooth biaffine min-max optimization instance $(\dR^n, \dR^m, \ell)$ with an NE $z^\star$, if EG is initiated at $z_0$ with $\norm{z_0-z^\star}\leq R$ and, samples $\lambda_0,\lambda_1,\dots$ from $\calE$, and uses $\gamma_t=\lambda_t/\sqrt\rho$, $\eta_t=\lambda_t\sqrt\rho$ as its stepsizes, its gradient norm will satisfy
    \begin{align*}
        \EE[\ln \GN(z_T)]\leq \ln\sbrm[\big]{CLRT^{-1/\beta}}.
    \end{align*}
\end{applemma}
\begin{proof}
    The proof will be very close to the proof of \cref{lem:eg-cont}, with modifications added to handle $\rho$.

    Assume $\beta\in (1, 2)$ and $\theta\in (\frac{2\pi}3, \pi)$ such that $\theta\beta>\pi$. Let $\rho=2+2\cos\theta\in (0, 1)$. 
    
    On the range $\beta\in (1,2)$, we know that $\frac{4\pi}{\beta\sin(\pi\beta/2)}>0$, so $I_{\rho,\beta}(0)$ shares the same sign as $\cos(\theta\beta/2)$; from the assumption, we have $\pi<\theta\beta<2\pi$, so $\cos(\theta\beta/2)<0$, and thus $I_{\rho,\beta}(0)<0$.

    For any $t\in (0, \frac12)$, we know that $\ln(1+(\rho-2)t+t^2)\geq \ln(1-2t+t^2)=2\ln(1-t)$. This function is concave on $(0, \frac12)$, with values $2\ln (1-0)=0$ and $2\ln (1-\frac12)=-2\ln 2$, so $2\ln(1-t)>-4\ln 2 \cdot t$. Therefore, for $s\in (0, \frac12]$, 
    \begin{align*}
        I_{\rho,\beta}(s)
        & = I_{\rho,\beta}(0)-\int_{0}^{s} \ln(1+(\rho-2)t+t^2)t^{-\beta/2-1}\dd{t} \\
        & \leq I_{\rho,\beta}(0)+\int_{0}^{s} (4\ln 2\cdot t) t^{-\beta/2-1}\dd{t} \\
        & = I_{\rho,\beta}(0)+4\ln 2\int_0^s t^{-\beta/2}\dd{t} = I_{\rho,\beta}(0)+\frac{4\ln 2}{1-\beta/2}s^{1-\beta/2}.\yestag\label{eq:Irhobeta-bound}
    \end{align*}

    Let $\lambdam=\min\cbrm[\Big]{\rbrm[\big]{(1-\frac\beta2)\frac{-I_{\rho,\beta}(0)}{8\ln 2}}^{1/(2-\beta)},\frac{1}{\sqrt2}}>0$; the value satisfies $\lambdam^2\leq \frac12$ and the following equality:
    \begin{align*}
        \frac{4\ln 2}{1-\beta/2}(\lambdam^2)^{1-\beta/2}\leq \frac{-I_{\rho,\beta}(0)}{2}.
    \end{align*}
    Therefore, for all $a\in [0, L]$, $(\lambdam a/L)^2\in [0, \lambdam^2]$, and thus we have 
    \begin{align*}
        I_{\rho,\beta}((\lambdam a/L)^2)\leq I_{\rho,\beta}(\lambdam^2)\leq I_{\rho,\beta}(0)-\frac{I_{\rho,\beta}(0)}{2}=\frac12 I_{\rho,\beta}(0)<0,
    \end{align*}
    where we used the monotonicity of $I_{\rho,\beta}(\cdot)$ on $(0, 2-\rho)\supseteq(0, \frac12)$ and \eqref{eq:Irhobeta-bound}. Then, we plug this into \eqref{eq:pareto-contraction-dseg} and obtain
    \begin{align*}
        \EE_{\lambda\sim \calE}\sbrm[\big]{\ln(1+(\rho-2)\lambda^2a^2+\lambda^4a^4)}
        & = \frac{\beta}{2}\rbrm[\Big]{ \frac{\lambdam a}{L} }^{\beta} \int_{(\lambdam a/L)^2}^{\infty} \ln(1+(\rho-2)t+t^2)t^{-\beta/2-1}\dd{t}\\
        & = \frac{\beta}{2}\rbrm[\Big]{ \frac{\lambdam a}{L} }^{\beta}I_{\rho,\beta}\rbrm[\big]{(\lambdam a/L)^2}\\
        & \leq \frac{\beta}{2}\rbrm[\Big]{ \frac{\lambdam a}{L} }^{\beta}\frac{I_{\rho,\beta}(0)}{2},
    \end{align*}
    which can be combined with \eqref{eq:calL-dseg} to get
    \begin{align*}
        \ln \calR_{\rho,T}(\calE, a)
        & \leq {
            \ln L + \ln \frac{a}{L}
            +
            \frac T2 
            \frac{\beta}{2}\rbrm[\Big]{ \frac{\lambdam a}{L} }^{\beta}
            \frac{I_{\rho,\beta}(0)}{2}
        } \\
        & = \ln L +  \rbrm[\Big]{
            \ln b
            +
            \frac T2 
            \frac{\beta}{2}\rbrm{ \lambdam b }^{\beta}
            \frac{I_{\rho,\beta}(0)}{2}
        } \tag{substitute $b=\frac aL$}\\
        & \leq \ln L - \frac{1}{\beta} \ln T - \frac{1}{\beta}\rbrm[\Big]{
            \ln \frac {\beta^2 \lambdam^\beta (-I_{\rho,\beta}(0))}{8} +1
        }.
    \end{align*}
    The last step is similar to the derivation in \eqref{eq:lem-eg-cont-final}. Choosing $\theta=\frac{\pi+\pi/\beta}{2}$, $\rho=2+2\cos\theta$, and $C=\exp\sbrm[\Big]{- \frac{1}{\beta}\rbrm[\big]{
            \ln \frac {\beta^2 \lambdam^\beta (-I_{\rho,\beta}(0))}{8} +1
        }}$
    finishes the proof.
\end{proof}

This lemma, together with \cref{lem:dseg-cont} and \cref{lem:opt}, proves the first statement of \cref{thm:main-dseg}; note that we need to define $\etam=\lambdam\sqrt{\rho}$ to obtain the exact parametrization of the theorem.

Again, this version of the lemma suffers from a large constant factor; to achieve $\calO(T^{-0.99})$, we choose $\beta=100/99$, and we have $\lambdam\approx 0.0084$ and $C\approx 3365$. Symbolic computation suggests that $C\sim \eps^{-2}$. This is, again, due to stepsize being too small.
To improve the constant factor, we again use a mixture between the Pareto distribution and a point mass.
\begin{applemma}\label{lem:dseg-mixture}
    For any $\beta\in (1, \frac 54)$, there exists $\rho\in (0, 1), p\in (0, 1), C\in \dR_+$, such that if $\lambda_0, \dots, \lambda_{T-1}$ are sampled from $\calE=p\cdot \Pareto(\lambda_*, \beta)+(1-p)\cdot \delta_{\lambda_*}$ where $\lambda_*=\frac{1}{\sqrt{2}L}$, we will have for all $a\in (0, L]$,
    \begin{align*}
        \calR_{\rho,T}(\calE, a)\leq CLT^{-1/\beta}.
    \end{align*}
\end{applemma}

Careful readers can spot that the lemma is almost identical to \cref{lem:eg-mixture}. Their proofs are almost identical as well; we present the proof here for completeness. %

\begin{proof}
    Since $\calE$ is a mixture, we have
    \begin{align*}
    & \phantom{{}={}} \EE_{\lambda\sim \calE}\sbrm[\big]{\ln(1+(\rho-2)\lambda^2a^2+\lambda^4a^4)} \\
    & = p \cdot \EE_{\lambda\sim \Pareto(\lambda_*,\beta)}\sbrm[\big]{\ln(1+(\rho-2)\lambda^2a^2+\lambda^4a^4)}
        + (1-p) \ln(1+(\rho-2)\lambda_*^2a^2+\lambda_*^4a^4) \\
    & = p \cdot \frac{\beta}{2}\rbrm[\Big]{ \frac{\lambdam a}{L} }^{\beta}I_{\rho,\beta}\rbrm[\big]{(\lambdam a/L)^2}
      + (1-p) \ln\sbrm[\Big]{1+(\rho-2)\lambdam^2\rbrm[\big]{\frac aL}^2+\lambdam^4\rbrm[\big]{\frac aL}^4}
    \end{align*}
    where, for notational consistency, we define $\lambdam=\frac{1}{\sqrt 2}$. Substituting $b=\frac aL$ and applying \eqref{eq:Irhobeta-bound}, we have
    \begin{align*}
    & \phantom{{}={}} 
        \EE_{\lambda\sim \calE}\sbrm[\big]{\ln(1-\draftcolor{cyan}{(\rho-2)}\lambda^2a^2+\lambda^4a^4)} \\
    & =
        p\cdot \frac{\beta}{2}\rbrm{ \lambdam b }^{\beta} I_{\rho,\beta}((\lambdam b)^2) 
        + (1-p) \ln\sbrm[\big]{1+(\rho-2)\lambdam^2b^2+\lambdam^4b^4} \\
    & \leq 
        p\cdot \frac{\beta}{2}\rbrm{ \lambdam b }^{\beta} \rbrm[\Big]{
            I_{\rho,\beta}(0)
            +
            \frac{4\ln 2}{1-\beta/2}(\lambdam b)^{2-\beta}
        }
        + (1-p) \ln\sbrm[\big]{1+(\rho-2)\lambdam^2b^2+\lambdam^4b^4} \\
    & \leq
        \frac{p\beta\cdot I_{\rho,\beta}(0)}{2}\rbrm{ \lambdam b }^{\beta}
        + p\cdot \frac{4\ln 2\cdot \beta}{2-\beta}\rbrm{ \lambdam b }^2
        + (1-p) \ln\sbrm[\big]{1+(\rho-2)\lambdam^2b^2+\lambdam^4b^4}.
        \yestag\label{eq:dseg-mixture-contraction}
    \end{align*}
    Assuming $\rho\leq\frac16<\frac{3-\sqrt{7}}{2}$; \Cref{lem:dseg-function-convexity} guarantees that the function $\ln(1+(\rho-2)z+z^2)$ is concave on $[0, \frac12]$, has a value of $0$ and a derivative of $\rho-2<-\frac{11}{6}$ when $z=0$; therefore, $\ln(1+(\rho-2)z+z^2)\leq -\frac{11}{6}z$ on $[0, \frac12]$. Because $\lambdam^2b^2\in [0, \frac12]$, we have
    \begin{align*}
        p\cdot \frac{4\ln 2\cdot \beta}{2-\beta}\rbrm{ \lambdam b }^2
        + (1-p) \ln\sbrm[\big]{1+(\rho-2)\lambdam^2b^2+\lambdam^4b^4}
        \leq
        \rbrm[\Big]{
            p\cdot \frac{4\ln 2\cdot \beta}{2-\beta}-(1-p)\cdot \frac{11}{6}
        } (\lambdam b)^2,
    \end{align*}
    and this value is nonpositive when $p\leq \frac{11(2-\beta)}{11(2-\beta)+24\ln 2\cdot \beta}$. Combining this with \eqref{eq:dseg-mixture-contraction} and \eqref{eq:calL-dseg}, we have
    \begin{align*}
        \ln \calR_{\rho,T}(\calE, a)
        & \leq
        \ln L + \max_{b\in (0, 1]} \cbrm[\Big]{
            \ln b
            +
            \frac T2 
            \frac{p\beta\cdot I_{\rho,\beta}(0)}{2}\rbrm{ \lambdam b }^{\beta}
        }\\
        & \leq
        \ln L - \frac{1}{\beta}\rbrm[\Big]{
            \ln \frac {T p \beta^2 \lambdam^\beta (- I_{\rho,\beta}(0))}{4} +1
        } \tag{$\star$}\label{step:lem-dseg-mixture-1st-order}\\
        & \leq
        \ln L - \frac{1}{\beta}\ln T
        - \frac{1}{\beta}\rbrm[\Big]{
            \ln \frac {p \beta^2 \lambdam^\beta (- I_{\rho,\beta}(0))}{4} +1
        }.
    \end{align*}
    Step \eqref{step:lem-dseg-mixture-1st-order} is again due to \cref{lem:ln-plus-monomial}. The lemma is proved when we take $\theta=\frac{2\pi}{3}+\frac{\pi}{3\beta}\in (\frac{14}{15}\pi, \pi)$, $\rho=2+2\cos\theta\in (0, 1/20)$, $p= \frac{11(2-\beta)}{11(2-\beta)+24\ln 2\cdot \beta}\in (0.28, 0.4)$ and $C=\exp\sbrm[\big]{- \frac{1}{\beta}\rbrm[\big]{
            \ln \frac {p \beta^2 \lambdam^\beta (- I_{\rho,\beta}(0))}{4} +1
        }}$.
\end{proof}

\begin{appremark}
    To achieve $\calO(T^{-0.99})$ rate with $\beta=1+\frac 1{99}$, \cref{lem:dseg-mixture} sets $C\approx 38.1$, significantly lower than that of \cref{lem:dseg-cont}. Symbolic computation suggests that $C\sim \eps^{-1}$.
\end{appremark}

The next lemma analyzes the error due to discretization and 
is analogous to \cref{lem:discretization}.

\begin{applemma}\label{lem:dseg-discretization}
    For any $\beta\in(1, \frac 87)$, let $\rho$ and $\calE$ be their values specified in \cref{lem:dseg-cont} or \cref{lem:dseg-mixture}.
    For $k\in [0, 1)$, We use $Q_\calE(k)$ to denote the $k$-th quantile of $\calE$.
    There exists a constant $C'$ such that, if $\lambda_t=Q_\calE(\vdc_t)$, where $\{\vdc_t\}_{t\in \dN}$ is the base-2 van der Corput sequence, then for any $T>0$, 
    \begin{align*}
        \frac1T
        \sum_{t=0}^{T-1} \ln(1+(\rho-2)\lambda_t^2 a^2 + \lambda_t^4 a^4)
        \leq
        \EE_{\lambda\sim \calE}[\ln(1+(\rho-2)\lambda^2a^2+\lambda^4 a^4)]
        + \frac{C'}{T}.
    \end{align*}
    Consequently, 
    $
        \calR_{\rho,T}(\{\lambda_t\}_{t\in [T]}, a)
        \leq
        e^{C'/2}
        \calR_{\rho,T}(\calE, a)
        .
    $
\end{applemma}
\begin{proof}
    The first claim is implied by \cref{thm:pareto-qmc-discretize} with $\psi=\pi-\theta=\pi-\arccos \frac{\rho-2}{2}$.

    The second claim is obtained by multiplying both sides by $T/2$, exponentiating, and using the definition of $\calR_{\rho,T}$ in \eqref{eq:calL-dseg}.
\end{proof}

This lemma together with \cref{lem:dseg-cont} proves the second statement of \cref{thm:main-dseg}.

\begin{appremark}
    In our discretization step, we introduce a multiplicative constant factor of $e^{C'/2}=\exp\sbrm[\big]{6(\ln(\sin\psi)^{-1})^2+\frac{53}{2}\ln(\sin\psi)^{-1}+1}$. When $\beta=100/99$, we choose $\psi=\pi/300$, and the constant factor evaluates to an astronomical $\approx1.2\times 10^{107}$. As our experiments (e.g., \cref{fig:gradient-norm-compare}) shows, this constant factor is purely an artifact of our analysis, and does not affect the empirical performance of our stepsize schedule.
\end{appremark}

\section{Quasi-Monte Carlo on $[0, 1)$ using the van der Corput sequence}\label{ah:discretization}

In this appendix section, we aim to provide a self-contained treatment on approximating an integral $\int_0^1 f(x) \dd{x}$ using a finite sum of discrete points $\frac{1}{n}\sum_{i=0}^{n-1} f(x_i)$, i.e., quasi-Monte Carlo, using the van der Corput sequence $x_i=\vdc_i$.
We will also apply the theory to discretize the function used in this paper.

\paragraph{Notation}
In this section, we follow the standard notation of real analysis. For a finite set $A$, $\abs{A}$ denotes its cardinality.
The $L_p$ norm of a sequence $x=\{x_n\}_{n\in \dN}$ is defined as $\norm{x}_{p}=\rbrm[\big]{\sum_{n=0}^{\infty}\abs{x_n}^p}^{1/p}$.

We use $A\to B$ to denote functions from $A$ to $B$. For an interval $S\subseteq\dR$,
$\calL_1(S)$ denotes the space of Lebesgue integrable functions on $S$, i.e., functions $f : S\to\dR$ satisfying $\int_S \abs{f(x)}\dd{x}<\infty$.
We denote by $C(S)$ the space of continuous functions on $S$ and
by $C^n(S)$ the space of functions that are $n$-times continuously differentiable on the interior of $S$.

The total variation of a function on an interval $[a,b]$ is defined as
\begin{align*}
    V_a^b(f)=\sup_{P\in \calP}\sum_{i=0}^{n_P-1}\abs{f(x_{i+1})-f(x_i)},
\end{align*}
where the supremum is taken over the set of partitions $\calP=\cbrm[\big]{P=\{x_0,\dots,x_{n_P}\}\mid a=x_0<x_1<\cdots<x_{n_P}=b, n_P \in\{1,2,\ldots\}}$. If one or both endpoint(s) of the interval is infinite, then the total variation is defined as the limit when that endpoint approaches infinity; formally, $V_a^{+\infty}(f)=\lim_{b\to+\infty} V_a^b(f)$.
The other infinite directions $V_{-\infty}^b(f)$ and $V_{-\infty}^{+\infty}(f)$ can be similarly defined.
A function is said to have \emph{bounded variation} on $[a,b]$ if $V_a^b(f)<\infty$.
If $f'$ is integrable, then $V_a^b(f)=\int_a^b \abs{f'(x)}\,\dd{x}$.

\subsection{Properties of the van der Corput sequence}

As a recap, we have the following definition.

\begin{appdefinition}
    For any integer $n \in \dN$, let its base-$b$ expansion (for a chosen integer base $b \geq 2$) be
    $n = \sum_{j=0}^\infty d_j(n) b^j$,
where $d_j(n) \in \{0, 1, \dots, b-1\}$ represents the $j$-th digit of $n$. We define:
\begin{itemize}
    \item The base-$b$ radical inverse of $n$, denoted by $\VDC_b(n)$, as the number obtained by mirroring the digits and putting them after the radix point,
    \begin{align*}
        \VDC_b(n) = \sum_{j=0}^\infty d_j(n) b^{-j-1}.
    \end{align*}
    \item The base-$b$ van der Corput sequence, denoted by $\VDC_b = \{\vdc_n\}_{n\in \dN}$, as $\vdc_n=\VDC_b(n)$.
\end{itemize}
\end{appdefinition}
We will assume $b=2$ for the entirety of this appendix. Equivalently, we can define the sequence recursively:
\begin{align*}
    \vdc_0&=0,\\
    \vdc_{2n}&=\frac{\vdc_n}2,\\
    \vdc_{2n+1}&=\frac{\vdc_n}2+\frac12.
\end{align*}
The equivalence can be proved by comparing the binary expansion of $n$ with those of $2n$ and $2n+1$.

We will use some properties of the sequence. First, we will show that this sequence is approximately asymptotically uniform; formally, the proportion of the sequence that is in an interval $[a,b)$ is roughly $b-a$. As the sequence is defined using binary digits, we will prove this for dyadic intervals.
\begin{appdefinition}\label{def:dyadic-interval}
    For any level $m\in \dN$ and any index $i\in \{0, 1, \dots, 2^m-1\}$, the $i$-th dyadic interval of length $2^{-m}$ is defined as $\calI_{m,i}=[i/2^{m}, (i+1)/2^{m})$.
\end{appdefinition}

\begin{applemma}\label{lem:vdc-proportion}
    For any level $m\in \dN$, any $0\leq i<2^m$, and any $n\in \dN$, the $n$-th element of the van der Corput sequence satisfies $\vdc_n\in \calI_{m,i}$ if and only if $n=2^m\VDC_2(i)+2^m q$ for some integer $q\geq 0$.
    
    Consequently, for $N\geq 1$,
    \begin{align*}
        \absm[\Big]{\frac{\abs{\cbr{0\leq n<N\mid \vdc_n\in \calI_{m,i}}}}{N}-\frac{1}{2^m}}\leq\frac 1N.
    \end{align*}
\end{applemma}
\begin{proof}
    The condition $\vdc_n\in \calI_{m,i}$ means that the first $m$ binary digits of $\vdc_n$ after the radix point are $d_{m-1}(i),\dots,d_0(i)$. Since $\vdc_n$ is obtained by reversing the binary digits of $n$, this is equivalent to saying that the last $m$ binary digits of $n$ are $d_{0}(i),\dots,d_{m-1}(i)$; formally, if $q=\floor{n/2^m}$, we have
    \begin{align*}
        n 
        & = 2^m q + \sum_{j=0}^{m-1} d_{m-1-j}(i) 2^j \\
        & = 2^m q + \sum_{j=0}^{m-1} d_{j}(i) 2^{m-1-j} \tag{substitute $j\leftarrow m-1-j$} \\
        & = 2^m q + 2^m \sum_{j=0}^{\infty} d_{j}(i) 2^{-1-j} \tag{$i$ has at most $m$ bits} \\
        & = 2^m q + 2^m \VDC_2(i).
    \end{align*}
    This also proved that $2^m\VDC_2(i)$ is an integer.

    Therefore, the proportion of numbers in the length-$N$ prefix that satisfy this condition is
    \begin{align*}
        \frac{\abs{\cbr{0\leq n<N\mid \vdc_n\in \calI_{m,i}}}}{N}
        & =
        \frac{\abs{\cbr{q\in \dN\mid 2^m q+2^m \VDC_2(i)<N}}}{N} \\
        & =
        \frac{\abs{\cbr{q\in \dN\mid 2^m q<N - 2^m \VDC_2(i)}}}{N} \\
        & = \frac{\ceil{N/2^m-\VDC_2(i)}}{N}\in \sbrm[\Big]{\frac{1}{2^m}-\frac{1}{N}, \frac{1}{2^m}+\frac{1}{N}}.
    \end{align*}
    The last statement is due to $\Phi_2(i)\in [0, 1)$ and \cref{lem:ceil-sub-1}.
\end{proof}

We will also make use of the fact that the van der Corput sequence is biased to the left, in the sense that the proportion of numbers below $t$ is at least $t$ for any prefix length. Formally,
\begin{applemma}\label{lem:vdc-bias}
    For every $N\geq 1$ and every $t\in [0, 1]$, we have
    \begin{align*}
        \abs{\cbr{0\leq n<N\mid \vdc_n<t}}\geq Nt.
    \end{align*}
    Equivalently, if $\phi!N_0<\cdots<\phi!N_{N-1}$ is the monotonically increasing permutation of $\{\vdc_n\}_{n\in [N]}$, then $\phi!N_k\leq \frac{k}{N}$ for $0\leq k<N$.
\end{applemma}
\begin{proof}
    Define $A(N,t)=\abs{\cbr{0\leq n<N\mid \vdc_n<t}}$.     
    We first prove $A(N,t)\geq Nt$ for all $N\geq 1$ and $t\in [0, 1]$ by an induction on $N$. The case $N=1$ is immediate.

    Assume the first claim holds for all $N'<N$. Write $N=2^m+r$ with $0\leq r<2^m$. Consider two cases:
    \begin{itemize}
        \item $t\in [0, \frac12]$. In this case, if $\vdc_n<t\leq \frac12$, then $n$ must be even. Thus, if we count all the even indices below $N$ with $\vdc_{2m}=\vdc_m/2$,
        \begin{align*}
            A(N,t)
            & = \abs{\cbr{0\leq 2m<N\mid \frac{\vdc_m}{2}<t}} \\
            & = \absm[\Big]{\cbrm[\Big]{0\leq m<\ceilm[\Big]{\frac N2}\mid \vdc_m<2t}}
            = A(\ceil{N/2},2t).
        \end{align*}
        Applying the inductive hypothesis with $N'=\ceil{N/2}$ yields \begin{equation*}
            A(N,t)=A(\ceil{N/2},2t)\geq \ceil{N/2}\cdot 2t\geq \rbr{N/2}\cdot 2t=Nt.
        \end{equation*}
        \item $t\in (\frac12, 1]$. In this case, $\vdc_{2m}<\frac12<t$ is unconditionally true. Thus, all even indices below $N$ contribute towards $A(N,t)$; there are $\ceil{\frac N2}$ of them. For odd indices $n=2m+1<N$, or equivalently $m<\floor{\frac N2}$, the condition becomes $t>\vdc_{2m+1}=\frac12+\frac{\vdc_m}{2}$, which simplifies to $\vdc_m<2t-1$. Summing these parts:
        \begin{align*}
            A(N,t)=\ceilm[\Big]{\frac N2}+A\rbrm[\Big]{\floorm[\Big]{\frac N2},2t-1}.
        \end{align*}
        Applying the inductive hypothesis with $N'=\floor{N/2}$ yields
        \begin{align*}
            A(N,t)
            \geq \ceilm[\Big]{\frac N2}+\floorm[\Big]{\frac N2}(2t-1)
            = 2t\floorm[\Big]{\frac N2}+\rbrm[\bigg]{\ceilm[\Big]{\frac N2}-\floorm[\Big]{\frac N2}}.
        \end{align*}
        For even $N$, the term $\ceil{N/2}-\floor{N/2}=0$, so
        \begin{align*}
            A(N,t)\geq 2t\floor{N/2}\geq 2t\cdot N/2=Nt.
        \end{align*}
        For odd $N=2M+1$, the term $\ceil{N/2}-\floor{N/2}=1$, so
        \begin{align*}
            A(N,t)\geq 2tM+1\geq 2tM+t=(2M+1)t=Nt.
        \end{align*}
    \end{itemize}
    In all cases, $A(N,t)\geq Nt$, so the first claim of the lemma holds for every $N\geq 1$ by induction.

    For the second claim, we observe that all terms in the van der Corput sequence are distinct; therefore, exactly $k$ elements in the first $N$ are less than $\phi!N_k$. So, together with the first claim,
    \begin{align*}
        k=A(N,\phi!N_k)\geq N\phi!N_k,
    \end{align*}
    which rearranges to $\phi!N_k\leq \frac kN$.
\end{proof}

\subsection{Properties of the quasi-Monte Carlo approximation}

We denote as $E_n(f)$ the estimation error of approximating the integral of $f$ using the length-$n$ prefix of the van der Corput sequence, or formally,
\begin{align}
    E_n(f)=\frac{1}{n}\sum_{i=0}^{n-1} f(\vdc_i)-\int_0^1 f(x) \dd{x}.\label{eq:def-En}
\end{align}
The estimation error functional has a signature of $E_n : \calL_1\rbrm[\big]{[0, 1)}\to\dR$. It requires $f$ to be defined on $[0, 1)$, because $\vdc_i\in [0, 1)$; it also requires $f$ being integrable on $(0, 1)$, so that the integration in \eqref{eq:def-En} is defined. Therefore, $E_n(f)$ requires $f\in \calL_1\rbrm[\big]{[0, 1)}$. From the definition, $E_n$ is clearly linear.

Because the van der Corput sequence is biased to the left, the finite sum always underestimates the integral of a monotonically increasing function.
\begin{apptheorem}\label{lem:En-mono}
    If $f\in \calL_1\rbrm[\big]{[0, 1)}$ is monotonically increasing, then $E_n(f)\leq0$. Conversely, if $f$ is monotonically decreasing, then $E_n(f)\geq 0$.
\end{apptheorem}
\begin{proof}
    Assume $f$ is monotonically increasing first. Let $\phi!n_0<\cdots<\phi!n_{n-1}$ be the increasing permutation of $\{\vdc_i\}_{i\in [n]}$. From \cref{lem:vdc-bias}, we know that $\phi!n_i\leq \frac{i}{n}$. Therefore,
    \begin{align*}
        \frac{1}{n}\sum_{i=0}^{n-1} f(\vdc_i)
        & = \frac{1}{n}\sum_{i=0}^{n-1} f(\phi!n_i) \\
        & \leq \frac{1}{n}\sum_{i=0}^{n-1} f\rbrm[\Big]{\frac{i}{n}} \tag{monotonicity; \Cref{lem:vdc-bias}} \\
        & = \sum_{i=0}^{n-1} \int_{i/n}^{(i+1)/n} f\rbrm[\Big]{\frac{i}{n}} \dd{x} \\
        & \leq \sum_{i=0}^{n-1} \int_{i/n}^{(i+1)/n} f(x) \dd{x} \tag{monotonicity} \\
        & = \int_0^1 f(x)\dd{x},
    \end{align*}
    and thus $E_n(f)\leq 0$.

    If $f$ is monotonically decreasing, applying the above argument to $-f$ proves that $E_n(f)=-E_n(-f)\geq 0$.
\end{proof}

In the following, we will present a result that claims $E_n(f)=\calO(1/n)$ for all functions $f$ that satisfy a certain smoothness condition.
\begin{appdefinition}\label{def:fin-dif}
    For a function $f : [a,b]\to \dR$, we define its \emph{finite difference} $\Delta_\delta f : [a, b-\delta]\to \dR$ as a function
    \begin{align*}
        \Delta_\delta f(x) = f(x+\delta)-f(x).
    \end{align*}
    We also define its \emph{second-order finite difference} $\Delta_\delta^2 f$, or simply \emph{second-order difference}, as the $\Delta_\delta$ transform applied twice, $\Delta_\delta^2 f = \Delta_\delta(\Delta_\delta f)$, expanded as
    \begin{align*}
        \Delta_\delta^2 f(x)=\Delta_\delta f(x+\delta)-\Delta_\delta f(x)=f(x+2\delta)-2f(x+\delta)+f(x).
    \end{align*}
\end{appdefinition}
\begin{appdefinition}\label{def:total-2nd}
    We say a function $f : [0, 1]\to \dR$ has \emph{finite total dyadic second-order difference} if $\norm{f}_{\Delta^2}<\infty$, where the \emph{total dyadic second-order difference} of $f$ is defined as
    \begin{align*}
        \norm{f}_{\Delta^2} = \frac12 \sum_{m=0}^{\infty} \sum_{a=0}^{2^m-1} \absm[\Big]{\Delta_{2^{-m-1}}^2 f(a/2^m)}.
    \end{align*}
\end{appdefinition}
The notation suggests that $\norm{\cdot}_{\Delta^2}$ is a norm. For functions with $\norm{\cdot}_{\Delta^2}<\infty$, it indeed is a seminorm; if we define the linear transform $\frakD f=\{\hat{f}_{m,a}\in\dR\}_{m\in \dN,0\leq a<2^m}$ that maps a function to a sequence, where $\hat{f}_{m,a}=\frac12 \Delta_{2^{-m-1}}^2 f(a/2^m)$, then $\norm{f}_{\Delta^2}$ is just the $L_1$ norm of the sequence $\norm{\frakD f}_1$. In the following lemma, we will show that $\frakD$ is invertible, at least for continuous functions that vanish on $\{0,1\}$.

\begin{applemma}\label{lem:Faber-Schauder}
    If $f \in C[0,1]$ and $f(0)=f(1)=0$, we have
    \begin{align}
        f(x)=\sum_{m=0}^{\infty} \sum_{a=0}^{2^m-1} \hat{f}_{m,a} v_{m,a}(x),
        \label{eq:Faber-decomposition}
    \end{align}
    for a collection of ``valley'' functions (see \cref{fig:valley-function}),
    \begin{align*}
        v_{m,a}(x)
        =
        \begin{cases}
        2^{m+1}\rbrm[\Big]{\frac{a}{2^m}-x}, & \frac{a}{2^m}\leq x<\frac{a+1/2}{2^m},\\
        2^{m+1}\rbrm[\Big]{x-\frac{a+1}{2^m}}, & \frac{a+1/2}{2^m}\leq x<\frac{a+1}{2^m},\\
        0, & \text{otherwise}.
        \end{cases}
    \end{align*}
    The sum in \eqref{eq:Faber-decomposition} converges uniformly.
\end{applemma}

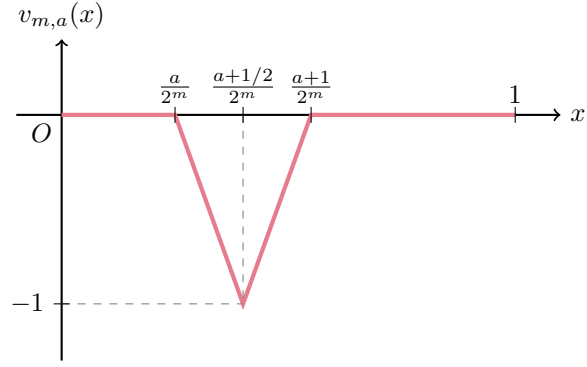
\begin{figure}[tb]
    \centering
    \begin{tikzpicture}[xscale=6, yscale=2.5]
        \draw[->, thick] (-0.1, 0) -- (1.1, 0) node[right] {$x$};
        \draw[->, thick] (0, -1.3) -- (0, 0.4) node[above] {$v_{m,a}(x)$};

        \draw[ultra thick, pastalred] 
            (0, 0) -- 
            (0.25, 0) -- 
            (0.4, -1) -- 
            (0.55, 0) -- 
            (1.0, 0);

        \draw[dashed, gray] (0.4, 0) -- (0.4, -1);
        \draw[dashed, gray] (0, -1) -- (0.4, -1);

        \node[below left] at (0,0) {$O$};
        
        \draw (0.5pt, -1) -- (-0.5pt, -1) node[left] {$-1$};
        
        \draw (0.25, 1.2pt) -- (0.25, -1.2pt) node[above=4pt] {$\frac{a}{2^m}$};
        \draw (0.4, 1.2pt) -- (0.4, -1.2pt) node[above=4pt] {$\frac{a+1/2}{2^m}$};
        \draw (0.55, 1.2pt) -- (0.55, -1.2pt) node[above=4pt] {$\frac{a+1}{2^m}$};
        \draw (1, 1.2pt) -- (1, -1.2pt) node[above=4pt] {$1$};
    \end{tikzpicture}
    \caption{Sketch of the valley function $v_{m,a}(x)$. Not to scale.}
    \label{fig:valley-function}
\end{figure}

\begin{proof}
    For notational convenience, define $L_{m,a}, M_{m,a}, R_{m,a}$ as the left endpoint, midpoint, and right endpoint of the interval $\calI_{m,a}$ (recall \cref{def:dyadic-interval}), respectively:
    \begin{align}
        L_{m,a}=\frac{a}{2^m}, \qquad M_{m,a}=\frac{a+1/2}{2^m}, \qquad R_{m,a}=\frac{a+1}{2^m}.\label{eq:def-lmr}
    \end{align}
    Also define $\calN_m$ as the $2^{-m}$-net, $\calN_m=\{0,\frac{1}{2^m},\dots,\frac{2^m-1}{2^m},1\}$.
    
    We define a family of functions $\{f_n\}_{n\in \dN}$ as follows; $f_n$ is piecewise linear with endpoints in $\calN_n$ and agrees with $f$ on exactly these points. Formally, if $L_{n,a}\leq x\leq R_{n,a}$ for $a\in \{0, 1,\dots,2^n-1\}$, we have $f_n(x)=\frac{R_{n,a}-x}{R_{n,a}-L_{n,a}}f(L_{n,a})+\frac{x-L_{n,a}}{R_{n,a}-L_{n,a}}f(R_{n,a})$.

    We first prove that 
    $f_n=\sum_{m=0}^{n-1}\sum_{a=0}^{2^m-1}\hat{f}_{m,a}v_{m,a}$. We prove this with induction. The case $n=0$ is immediate, as $f(0)=f(1)=0$ implies $f_0(x)=0$.

    For the induction step, we define $\tilde{f}_{n+1}(x)=f_n(x)+\sum_{a=0}^{2^n-1} \hat{f}_{n,a} v_{n,a}(x)$.

    We first show that $\tilde{f}_{n+1}$ is piecewise linear with endpoints in $\calN_{n+1}$. We know that $f_n$ is piecewise linear with endpoints in $\calN_n\subseteq \calN_{n+1}$, and that $v_{m,a}$ is piecewise linear with endpoints in $\{0,a/2^m,(a+1/2)/2^m, (a+1)/2^m, 1\}=\{0,2a/2^{m+1},(2a+1)/2^{m+1}, (2a+2)/2^{m+1}, 1\}\subseteq \calN_{n+1}$. A finite sum of piecewise linear functions with endpoints in $\calN_{n+1}$ must also be such a piecewise linear function.

    We then show that $\tilde{f}_{n+1}$ agrees with $f$ on points in $\calN_{n+1}$.\\
    For each point with an even index $x=\frac{2i}{2^{n+1}}$, we know that $v_{n,a}(x)=0$ for every $a$, and thus $f_{n+1}(x)=f_n(x)=f(x)$ because $f_n(x)$ agrees with $f$ on $\calN_{n}\ni \frac{i}{2^n}=x$.\\
    For each point with an odd index $x=\frac{2i+1}{2^{n+1}}=\frac{i+1/2}{2^n}=M_{n,i}$, we know that the only valley function with nonzero value is $v_{n,i}$, whose value is $-1$.  Therefore,
    \begin{align*}
        \tilde{f}_{n+1}(x)
        & = f_n(x)+\hat{f}_{n,i}v_{n,i}(x)\\
        & = \frac12 f(L_{n,i})+\frac12 f(R_{n,i})+\frac12 \Delta_{2^{-n-1}}^2 f\rbrm[\Big]{\frac{i}{2^n}}\cdot (-1) \tag{definition of $\hat{f}_{n,i}$} \\
        & = \frac12 f(L_{n,i})+\frac12 f(R_{n,i})-\frac12 \sbrm[\Big]{
            f\rbrm[\Big]{\frac{i}{2^n}}
            - 2 f\rbrm[\Big]{\frac{i}{2^n}+2^{-n-1}}
            + f\rbrm[\Big]{\frac{i}{2^n}+2\cdot 2^{-n-1}}
        }\tag{\cref{def:fin-dif}}\\
        & = \frac12 f(L_{n,i})+\frac12 f(R_{n,i})-\frac12 \sbrm[\Big]{
            f\rbrm[\big]{L_{n,i}}
            - 2 f\rbrm[\big]{M_{n,i}}
            + f\rbrm[\big]{R_{n,i}}
        } \tag{substitute \eqref{eq:def-lmr}} \\
        & = f(M_{n,i})=f(x).
    \end{align*}

    We have proved that $\tilde{f}_{n+1}(x)=f_n(x)+\sum_{a=0}^{2^n-1} \hat{f}_{n,a} v_{n,a}(x)$ is a piecewise linear function with $\calN_{n+1}$ as endpoints and agrees with $f(x)$ on $\calN_{n+1}$. Since $f_{n+1}$ is also such a function, and a piecewise linear function is uniquely determined by its endpoints and its value at those endpoints, we can conclude that $f_{n+1}=\tilde{f}_{n+1}$, finishing the induction step.
    
    Therefore, we establish by induction that $f_n$ is precisely the partial sum for the outer summation in \eqref{eq:Faber-decomposition}.
    
    It remains to prove that
    $\{f_n\}_{n\in \dN}$ converges to $f$ uniformly. As $f$ is continuous on a closed interval $[0, 1]$, $f$ is uniformly continuous. This means that for every $\eps>0$, there exists a $\delta=\delta(\eps)>0$, such that $\abs{f(x_1)-f(x_2)}\leq \eps$ whenever $\abs{x_1-x_2}\leq \delta$. 
    
    Therefore, as long as $n\geq \log_2 \delta^{-1}$, for any $x\in [0, 1]$, if $x\in \calI_{n,a}$, we know that $R_{n,a}-L_{n,a}=2^{-n}\leq \delta$, and thus $\abs{x-L_{n,a}}\leq \delta$ and $\abs{x-R_{n,a}}\leq \delta$, which implies from uniform continuity that $\abs{f(x)-f(L_{n,a})}\leq \eps$ and $\abs{f(x)-f(R_{n,a})}\leq \eps$. Because $f_n(x)$ is a convex combination of $f(L_{n,a})$ and $f(R_{n,a})$, it must satisfy $\abs{f(x)-f_n(x)}\leq \eps$ as well. Thus uniform convergence is proven.
\end{proof}

Our main result in this section is the following theorem, stating that functions with finite total dyadic second-order differences have $\calO(1/n)$ approximation error under the van der Corput sequence.

\begin{apptheorem}\label{thm:En-2nd-order}
    Let $f : [0, 1]\to \dR$ be a continuous function on $[0, 1]$ with $f(0)=f(1)=0$ and $\norm{f}_{\Delta^2}<\infty$. Then,
    \begin{align*}
        \abs{E_n(f)}\leq \frac{\norm{f}_{\Delta^2}}{n}.
    \end{align*}
\end{apptheorem}
\begin{proof}
    The condition of this theorem implies the condition of \cref{lem:Faber-Schauder}. 
    Because $\norm{f}_{\Delta^2}=\sum_{m,a} \absm[\big]{\hat{f}_{m,a}}<\infty$, we know that \eqref{eq:Faber-decomposition} holds uniformly and absolutely, and therefore it commutes with integration and finite sum; we have
    \begin{align*}
        E_n(f)=\sum_{m=0}^{\infty} \sum_{a=0}^{2^m-1} \hat{f}_{m,a} E_n(v_{m,a}).
    \end{align*}
    where the outer summation holds absolutely. This allows us to take the absolute value in both sides,
    \begin{align*}
        \abs{E_n(f)}
        & \leq \sum_{m=0}^{\infty} \sum_{a=0}^{2^m-1} \abs{\hat{f}_{m,a}} \abs{E_n(v_{m,a})} \\
        & \leq \sum_{m=0}^{\infty} \sum_{a=0}^{2^m-1} \abs{\hat{f}_{m,a}} \cdot \frac{1}{n} \\
        & = \frac{\norm{f}_{\Delta^2}}{n}.
    \end{align*}
    In the derivation, we use the fact that $\abs{E_n(v_{m,a})}\leq \frac{1}{n}$, which we will prove below in \cref{lem:En-valley}. This finishes the proof.
\end{proof}
\begin{applemma}\label{lem:En-valley}
    For any $n\geq 1$, any $m\geq 0$, and any $a=0,1,\dots,2^m-1$, we have
    \begin{align*}
        \abs{E_n(v_{m,a})}\leq\frac 1n.
    \end{align*}
\end{applemma}
\begin{proof}
    From the definition of $v_{m,a}$, we know that $\int_0^1 v_{m,a}(x)\dd{x}=-2^{-m-1}$. Each term of the finite sum $\sum_{i=0}^{n-1} v_{m,a}(\vdc_i)$ is nonzero only when $\vdc_i\in \calI_{m,a}$, which from the proof of \cref{lem:vdc-proportion} means that $i=2^m\VDC_2(a)+2^m q$. Since $\VDC_2(\cdot)$ represents reflecting bits around the radix point, we have $\vdc_i=\VDC_2(i)=2^{-m}a+2^{-m}\vdc_q$.
    
    By the definition of $v_{m,a}$, every valley is the standard valley $v_{0,0}$ on $[0, 1]$ shifted and scaled on the $x$-axis; or formally, $v_{m,a}(x 2^{-m}+a/2^m)=v_{0,0}(x)$. Therefore, the finite sum of $v_{m,a}$ becomes
    \begin{align*}
        \sum_{i=0}^{n-1}v_{m,a}(\vdc_i)
        & = \sum_{\substack{0\leq i<n\\i=2^m(\VDC_2(a)+q)}}v_{m,a}(\vdc_i) \\
        & = \sum_{\substack{0\leq i<n\\i=2^m(\VDC_2(a)+q)}}v_{m,a}\rbrm[\Big]{\frac{a}{2^m}+2^{-m}\vdc_q} \\
        & = \sum_{\substack{q\in \dN\\2^m(\VDC_2(a)+q)<n}}v_{0,0}\rbrm{\vdc_q} \\
        & = \sum_{q=0}^{\ceilm{\frac{n}{2^m}-\VDC_2(a)}-1}v_{0,0}\rbrm{\vdc_q}.
    \end{align*}
    Define $S(n)=\sum_{q=0}^{n-1} v_{0,0}(\vdc_q)$. For any $q$, we know that $\vdc_{2q+1}=\vdc_{2q}+\frac12$, so
    \begin{align*}
        v_{0,0}(\vdc_{2q})+v_{0,0}(\vdc_{2q+1}) = -2\vdc_{2q}+2(\vdc_{2q+1}-1) = -2\vdc_{2q}+2\rbrm[\Big]{\vdc_{2q}+\frac12-1}=-1.
    \end{align*}
    This means that $S(n)=-\frac{n}{2}$ for even $n$ and $S(n)=-\frac{n-1}{2}+v_{0,0}(\vdc_{n-1})$ for odd $n$; in both cases, $\abs{S(n)+\frac{1}{2}n}\leq\frac12$. Therefore,
    \begin{align*}
        \abs{E_n(v_{m,a})}
        & = \absm[\bigg]{
            \frac{1}{n}\sum_{i=0}^{n-1} v_{m,a}(\vdc_i)
            -\int_0^1 v_{m,a}(x) \dd{x}
        } \\
        & = \absm[\bigg]{
            \frac{1}{n} S\rbrm[\Big]{\ceilm[\big]{\frac{n}{2^m}-\VDC_2(a)}}
            -(-2^{-m-1})
        } \\
        & = \frac{1}{n}\absm[\bigg]{
            S\rbrm[\Big]{\ceilm[\big]{\frac{n}{2^m}-\VDC_2(a)}}
            +\frac{1}{2}\cdot \frac{n}{2^m}
        } \\
        & \leq \frac{1}{n}\absm[\bigg]{
            S\rbrm[\Big]{\ceilm[\big]{\frac{n}{2^m}-\VDC_2(a)}}
            +\frac{1}{2}\ceilm[\big]{\frac{n}{2^m}-\VDC_2(a)}
        } + \frac{1}{n}\absm[\bigg]{
            \frac{1}{2}\ceilm[\big]{\frac{n}{2^m}-\VDC_2(a)}
            -\frac{1}{2}\cdot \frac{n}{2^m}
        } \\
        & \leq \frac{1}{2n} + \frac{1}{2n} \absm[\bigg]{
            \ceilm[\big]{\frac{n}{2^m}-\VDC_2(a)}
            - \frac{n}{2^m}
        } \leq \frac1n.
    \end{align*}
    In the last step, we used the fact that if $x\in \dR,y\in [0, 1)$, then $\abs{\ceil{x-y}-x} \leq 1$ (\cref{lem:ceil-sub-1}).
\end{proof}

\subsection{Bibliographic notes}

The van der Corput sequence is first proposed by \citet{vandercorput1935verteilungsfunktionen}. For an overview of Quasi-Monte Carlo using low-discrepancy sequences like $\{\vdc_i\}_{i\in \dN}$, we refer the readers to \citet{niederreiter1992random}.
The traditional low-discrepancy theory only shows that $E_n(f)=\calO(\ln n/n)$ for functions $f$ of bounded variation; in this paper, a better $\calO(1/n)$ rate is needed.

The valley functions of \cref{lem:Faber-Schauder} are known as the Faber-Schauder system, first proposed by \citet{faber1910uber} and \cite{schauder1927zur}. It is the integral of the Haar wavelets \citep{haar1910zur}. Along with linear functions, it forms a countable basis of $C[0, 1]$. The exact form of this decomposition lemma appears in \citet{ulyanov1972representation}.
The system is useful in decomposing a function into local, simpler parts, akin to Fourier and wavelet transforms. We refer the readers to \cite{kashin1989orthogonal} for further discussions.

The condition defined in \cref{def:total-2nd} is related to the \emph{Besov space} $B_{1,1}^1([0, 1])$; we recover the exact condition for the Besov space if the summation of $\{\hat{f}_{m,a}=\frac12 \Delta_{2^{-m-1}}^2 f(a/2^m)\}_{a\in [2^m]}$ is replaced with the integral of $\Delta_\delta^2 f(x)$ for suitably chosen $\delta$. Our norm $\norm{\cdot}_{\Delta^2}$ corresponds exactly to the $b_{1,1}^1$ norm on our sequence $\{\hat{f}_{m,a}\}_{m\in \dN, a\in [2^m]}$. See, e.g., \citet{triebel2010bases} for a textbook on the topic, especially Section 3.1.
Our main result of this appendix \cref{thm:En-2nd-order} is thus a borderline case of results related to numerical integration in Chapter 5 of the same book.

\subsection{Useful lemmas for second-order difference calculation}

In this section, we present several helpful lemmas that aids the calculation of $\norm{f}_{\Delta^2}$.

For notational convenience, we write
\begin{align*}
    \norm{f}_{\Delta^2} = \sum_{m=0}^{\infty} \norm{f}_{\Delta^2,m}, \qquad \norm{f}_{\Delta^2,m}\defeq\frac12 \sum_{a=0}^{2^m-1} \absm[\Big]{\Delta_{2^{-m-1}}^2 f(a/2^m)}.
\end{align*}

Expanding the definition of second-order difference of the function, we have
\begin{align*}
    \norm{f}_{\Delta^2,m}&=
    \frac12\sum_{a=0}^{2^m-1}\absm[\Big]{
        2f\rbrm{M_a}
        -f\rbrm{L_a}
        -f\rbrm{R_a}
    }
    \\&=
    \frac12 \sum_{a=0}^{2^m-1}\absm[\bigg]{
        2f\rbrm[\Big]{\frac{2a+1}{2^{m+1}}}
        -f\rbrm[\Big]{\frac{2a}{2^{m+1}}}
        -f\rbrm[\Big]{\frac{2a+2}{2^{m+1}}}
    }.
\end{align*}
Note that we used the shorthand notation for the dyadic intervals,
\begin{align}
    L_{m,a}=\frac{a}{2^m}, \qquad M_{m,a}=\frac{a+1/2}{2^m}, \qquad R_{m,a}=\frac{a+1}{2^m}.\label{eq:def-lmr'}
\end{align}

The following lemma is useful in bounding one term in the summation of $\norm{f}_{\Delta^2,m}$.

\begin{applemma}\label{lem:finite-difference-smoothness}
    We have the following three bounds on the second-order difference of a function:
    \begin{enumerate}
        \item If $\abs{f(x)}\leq M$ on $[x, x+2\delta]$, then $\abs{\Delta_\delta^2 f(x)}\leq 4M$.
        \item If $f$ is $M$-Lipschitz continuous on $[x, x+2\delta]$, then $\abs{\Delta_\delta^2 f(x)}\leq 2M\delta$;
        as a consequence, if $f$ is continuous on $[x,x+2\delta]$ and $\abs{f'(x)}\leq M$ on $(x,x+2\delta)$, then the same bound holds.
        \item If $f'$ is continuous on $[x, x+2\delta]$ and $\abs{f''(x)}\leq M$ on $(x,x+2\delta)$, then $\abs{\Delta_\delta^2 f(x)}\leq M\delta^2$.
    \end{enumerate}
\end{applemma}
\begin{proof}
    The zeroth-order statement can be proved with the triangle inequality
    \begin{align*}
        \abs{\Delta_\delta^2 f(x)}\leq\abs{f(x)}+2\abs{f(x+\delta)}+\abs{f(x+2\delta)}\leq M+2M+M=4M.
    \end{align*}
    The first-order statement can be proved with the triangle inequality
    \begin{align*}
        \abs{\Delta_\delta^2 f(x)}
        & =\abs{\rbr{f(x+2\delta)-f(x+\delta)}-\rbr{f(x+\delta)-f(x)}} \\
        & \leq \abs{f(x+2\delta)-f(x+\delta)}+\abs{f(x+\delta)-f(x)} \\
        & \leq \abs{(x+2\delta)-(x+\delta)}M+\abs{(x+\delta)-x}M
          \leq 2M\delta.
    \end{align*}
    The second-order statement can be proved with Taylor's expansion with a Lagrange remainder,
    \begin{align*}
        f(x)&=f(x+\delta)-\delta f'(x+\delta)+\frac12 f''(\xi_1)\delta^2,\\
        f(x+2\delta)&=f(x+\delta)+\delta f'(x+\delta)+\frac12 f''(\xi_2)\delta^2,
    \end{align*}
    where $\xi_1\in (x,x+\delta)$ and $\xi_2\in (x+\delta, x+2\delta)$. Adding these two equations and rearranging gives
    \begin{align*}
        \Delta_\delta^2 f(x)=f(x)+f(x+2\delta)-2f(x+\delta)
        =\delta f'(x+\delta)-\delta f'(x+\delta)+\frac12 \delta^2\rbrm[\big]{f''(\xi_1)+f''(\xi_2)}.
    \end{align*}
    The first two terms cancel. Taking an absolute value then gives\phantom{\qedhere}
    \begin{align*}
        \absm[\big]{\Delta_\delta^2 f(x)}\leq \frac12 \delta^2\rbrm[\big]{\abs{f''(\xi_1)}+\abs{f''(\xi_2)}}
        \leq \frac12 \delta^2 \cdot 2M = M\delta^2.\mqed
    \end{align*}
\end{proof}

Using this result, we can show $\norm{\cdot}_{\Delta^2}<\infty$ is indeed a smoothness condition, as all functions with absolutely integrable second-order derivatives satisfy this condition, formally stated in the following lemma.

\begin{applemma}\label{lem:finite-difference-global-smoothness}
    If $f : [0, 1]\to \dR$ satisfies $\int_0^1 \abs{f''(x)}\,\dd{x}\leq M$, then $\norm{f}_{\Delta^2}\leq M$. Consequently, all functions that satisfy $f''\in \calL_1((0, 1))$ have $\norm{f}_{\Delta^2}<\infty$.
\end{applemma}
\begin{proof}
    Assume $f$ satisfies the condition $\int_0^1 \abs{f''(x)}\,\dd{x}\leq M$. This implies that $f$ is twice differentiable on $(0, 1)$. Let $m$ be a nonnegative integer and let $\delta=2^{-m-1}$.

    For any $x\in [0, 1-2\delta]$, the second-order finite difference is
    \begin{align*}
        \Delta_\delta^2 f(x)
        & = [f(x+2\delta)-f(x+\delta)]-[f(x+\delta)-f(x)] \\
        & = \int_0^\delta f'(x+\delta+t)\dd{t}-\int_0^\delta f'(x+t)\dd{t}\tag{fundamental theorem of calculus} \\
        & = \int_0^\delta f'(x+\delta+t)-f'(x+t)\dd{t} \\
        & = \int_0^\delta \rbrm[\Big]{\int_0^\delta f''(x+s+t)\dd{s}}\dd{t}. \tag{again, fundamental theorem of calculus}
    \end{align*}
    Taking an absolute value yields
    \begin{align*}
        \abs{\Delta_\delta^2 f(x)}
        & \leq \int_0^\delta \int_0^\delta \abs{f''(x+s+t)}\,\dd{s}\,\dd{t} \\
        & = \int_0^\delta \int_{t}^{t+\delta} \abs{f''(x+u)}\,\dd{u}\,\dd{t}\tag{substitute $u=s+t$, $\dd{u}=\dd{s}$}\\
        & \leq \int_0^\delta \int_{0}^{2\delta} \abs{f''(x+u)}\,\dd{u}\,\dd{t}\tag{$[t,t+\delta]\subseteq[0,2\delta]$}\\
        & = \delta\int_{0}^{2\delta}\abs{f''(x+u)}\,\dd{u}.
    \end{align*}
    Therefore, the total second-order difference at level $m$ is the summation 
    \begin{align*}
        \sum_{a=0}^{2^m-1}\abs{\Delta_\delta^2 f(2a\delta)}
        & \leq \sum_{a=0}^{2^m-1}\delta\int_{0}^{2\delta}\abs{f''(2a\delta+u)}\,\dd{u}\\
        & = \sum_{a=0}^{2^m-1}\delta\int_{2a\delta}^{2(a+1)\delta}\abs{f''(u)}\,\dd{u}\\
        & = \delta\int_{0}^{1}\abs{f''(u)}\,\dd{u}\leq \delta M=2^{-m-1}M.
    \end{align*}
    Finally, summing this over all levels $m$ yields
    \begin{align*}
        \norm{f}_{\Delta^2}\leq \sum_{m=0}^\infty2^{-m-1}M\leq M.
    \end{align*}
    This proves the main statement. The second statement is proved by the observation that the condition $f''\in \calL_1((0,1))$ is equivalent to $M=\int_0^1\abs{f''(x)}\,\dd{x}\leq \infty$.
\end{proof}

We also show the following lemma, which bounds \emph{one level} of the total dyadic second-order difference if a function has bounded variation.

\begin{applemma}\label{lem:finite-difference-bounded-variation}
    If a function $f$ has bounded variation on $[0, 1]$, i.e., its total variation is finite $V_0^1(f)<\infty$,
    then $\norm{f}_{\Delta^2,m}\leq \frac12 V_0^1(f)$ for any $m\in \dN$.
\end{applemma}
\begin{proof}
    Using the triangle inequality and the definition of $V_0^1(f)$, we have
    \begin{align*}
        \norm{f}_{\Delta^2,m}
        & =
        \frac12\sum_{a=0}^{2^m-1}\absm[\Big]{
            2f\rbrm{M_a}
            -f\rbrm{L_a}
            -f\rbrm{R_a}
        } \\
        & \leq \frac12\sum_{a=0}^{2^m-1}\rbrm[\Big]{\absm[\big]{
            f\rbrm{M_a}
            -f\rbrm{L_a}
        }+\absm[\big]{
            f\rbrm{R_a}
            -f\rbrm{M_a}
        }}\\
        & = \frac12\sum_{a=0}^{2^{m+1}-1}\absm[\big]{
            f((a+1)\delta)-f(a\delta)
        }\leq \frac12 V_0^1(f),
    \end{align*}
    where $\delta=2^{-m-1}$, and $\{0,\delta,2\delta,\dots,(2^{m+1}-1)\delta,2^{m+1}\delta=1\}$ forms a partition of $[0, 1]$.
\end{proof}

We note that the definition of $\norm{\cdot}_{\Delta^2}$ is symmetric on the interval $[0, 1]$, captured in the following lemma:
\begin{applemma}\label{lem:finite-difference-reflection}
    For any function $f : [0, 1]\to \dR$, let $\cev{f}(u)=f(1-u)$. We have $\normm[\big]{\cev{f}}_{\Delta^2}=\norm{f}_{\Delta^2}$.
\end{applemma}
\begin{proof}
    If we expand the norm notation into a sum of sums, we can see that the second difference in $\calI_{m,a}$ for $\cev{f}$ is the same as the second difference in $\calI_{m,2^m-1-a}$ for $f$. Reversing the order of inner summation proves the statement.
\end{proof}

\subsection{Application: Discretizing Pareto using Quantiles}\label{ah:discretization-main}

In this part of the appendix, we apply our theory on dyadic second-order differences to the function used in this paper.

Assume $\beta\in (1, \frac85]$ and $\etam\in (0, 1/\sqrt2]$. Let $\calE = p\Pareto(\etam,\beta)+(1-p)\delta_{\etam}$ be the Pareto-point-mass mixture we aim to discretize, where $p\in (0, 1]$. We know that its quantile function is
\begin{align*}
    Q_\calE(u)=\begin{cases}
        \etam, & u\leq 1-p, \\
        \etam \rbrm[\big]{\frac{1-u}{p}}^{-1/\beta}, &u>1-p.
    \end{cases}
\end{align*}
Let $f_\psi(z)=1-2\cos\psi\cdot z+z^2$, where $\psi\in (0, \frac\pi3]$. 

We aim to prove, for any $a\in [0, 1]$ and positive integer $n$ ,
\begin{align*}
    \frac1n\sum_{t=0}^{n-1} \ln f_\psi(a^2\eta_t^2)
    \leq
    \EE_{\eta\sim\calE}[\ln f_\psi(a^2\eta^2)]
    + \frac{C}{n}
\end{align*}
for some $C<\infty$, where $\eta_t=Q_\calE(\vdc_t)$. The notation assumes $L=1$; the case for a generic $L$ is achieved by scaling $\calE$ by $1/L$ and scaling $a$ by $L$. The expectation is equal to
\begin{align*}
    \EE_{\eta\sim\calE}[\ln f_\psi(a^2\eta^2 )]
    =
    \EE_{u\sim U[0, 1]}[\ln f_\psi(a^2Q_\calE^2(u))].
\end{align*}
Using the definition \eqref{eq:def-En} of the previous section, we equivalently want to prove
\begin{align*}
    C/n\geq \frac1n\sum_{t=0}^{n-1} \ln f_\psi(a^2Q_\calE^2(\vdc_t))-\EE_{u\sim U[0, 1]}[\ln f_\psi(a^2Q_\calE^2(u))]
    =E_n(h),
\end{align*}
where we define the function $h(u)=\ln f_\psi(a^2Q^2_\calE(u))$.

\begin{appremark}
    We use a general parametrization to unify all cases needed in the paper. For the results for single-stepsize EG in this paper, let $\psi=\frac\pi3$; for double-stepsize EG, let $\psi=\pi-\theta$ and $\rho=2+2\cos\theta=2-2\cos\psi$. For the results with a pure Pareto distribution, take $p=1$.
\end{appremark}

Intuitively, we want to show that our function $h$ satisfies the condition of \cref{def:total-2nd}. However, this is impossible; when $u\to 1$, $h(u)\sim 2\ln(a^2 Q_\calE^2(u))\to \infty$, and if $h$ is unbounded, $\norm{h}_{\Delta_2}$ must be infinite. To address this issue, we make use of the following monotonic function: 
\begin{align*}
    \acute{H}(u)=\begin{cases}
        0, & 0\leq u<1-\sigma,\\
        2\ln(a^2 Q_\calE^2(u)),&1-\sigma\leq u<1,
    \end{cases}
\end{align*}
where $\sigma$ is a threshold such that $a^2 Q_\calE^2(1-\sigma)=1$. The monotonicity of $\acute{H}$ is guaranteed by the monotonicity of $Q_\calE$, and \cref{lem:En-mono} shows that $E_n(\acute{H})\leq 0$. We then need to show that $h-\acute{H}$ (almost) satisfies the condition of \cref{thm:En-2nd-order}, which we will detail below.

We will present the proof of the main theorem first, which makes use of the $\norm{\cdot}_{\Delta^2}<\infty$ condition in \cref{lem:pareto-remainder-second-difference}. Both results make use of the analysis on the derivatives of the functions involved (\cref{lem:log-poly-basic} and \ref{lem:log-poly-psi}).

\begin{apptheorem}\label{thm:pareto-qmc-discretize}
    Under the definitions and assumptions of this section, there exists a constant $C=C(\psi)<\infty$ such that, for every $a\in[0,1]$ and every $n\geq1$,
    \begin{align*}
        E_n(h)\leq \frac{C}{n}.
    \end{align*}
\end{apptheorem}
\begin{proof}
    If $a=0$, then $h(u)=\ln f_\psi(a^2Q^2_\calE(u))\equiv 0$, and the claim is immediate. Assume $a>0$.

    We first handle the point-mass branch. Define the function $\hat{Q}_\calE(u)=\etam \rbrm[\big]{\frac{1-u}{p}}^{-1/\beta}$ for all $u\in [0, 1)$. (This function is the quantile function of a Pareto distribution that has the same tail as $\calE$ on $(\etam,+\infty)$.) Define
    \begin{align*}
    H(u)=\ell_\psi(a^2 \hat{Q}_\calE^2(u)),
    \end{align*}
    where, for notational convenience, we define
    \begin{align*}
        \ell_\psi(z)=\ln f_\psi(z).
    \end{align*}

    On the interval $[1-p,1)$, the functions $Q_\calE$ and $\hat{Q}_\calE$ agree, so $H(u)-h(u)=0$. On the interval $[0, 1-p]$, $\hat{Q}_\calE(u)$ is monotonically increasing, and $a^2\hat{Q}_\calE^2(u)\leq a^2\etam^2\leq \frac12$; as the function $f_\psi$ is monotonically decreasing on $[0, \cos\psi]\supseteq[0, \frac12]$, $H(u)=\ln f_\psi(a^2\hat{Q}_\calE^2(u))$ is monotonically decreasing as well, and $H(u)\leq H(0)\leq 0$.
    
    Therefore, $h-H$ is continuous and monotonically increasing (or constant) on both branches, so it is monotonically increasing on $[0, 1)$, and \cref{lem:En-mono} guarantees that $E_n(h-H)\leq 0$, or $E_n(h)\leq E_n(H)$.

    In order to investigate $H(u)=\ln f_\psi(a^2\hat{Q}_\calE^2(u))$, we expand the inner expression $a^2\hat{Q}_\calE^2(u))=a^2\etam^2\rbrm[\big]{\frac{1-u}{p}}^{-2/\beta}=\rbrm[\big]{\frac{1-u}{p a^\beta \etam^\beta}}^{-2/\beta}=\rbrm[\big]{\frac{1-u}{\sigma}}^{-2/\beta}$, where we define $\sigma=pa^\beta \etam^\beta$; this number also satisfies $a^2 \hat{Q}_\calE^2(1-\sigma)=a^2 \rbrm[\big]{\etam (\sigma/p)^{-1/\beta}}^2=1$.

    We decompose $H$ into a monotone part and a bounded part. Define
    \begin{align*}
        \acute{H}(u)=\begin{cases}
            0, & 0\leq u<1-\sigma,\\
            2\ln(a^2 \hat{Q}_\calE^2(u)),&1-\sigma\leq u<1,
        \end{cases}
    \end{align*}
    and let $\mathsf{H}=H-\acute{H}$. We know that $\acute{H}(1-\sigma)=2\ln(a^2 \hat{Q}_\calE^2(1-\sigma))=2\ln1=0$, so $\acute{H}$ is continuous; $\acute{H}$ is monotonically increasing on both branches (the second being guaranteed by the monotonicity of $\hat{Q}_\calE$), so $\acute{H}$ is monotonically increasing, and \cref{lem:En-mono} again guarantees that $E_n(\acute{H})\leq 0$, and $E_n(H)\leq E_n(\mathsf{H})$.
    
    If we expand and simplify $\mathsf{H}$, we know that it agrees with $H$ when $0\leq u<1-\sigma$; on $1-\sigma\leq u<1$, we have
    \begin{align*}
        \mathsf{H}(u)=\ell_\psi\rbrm[\big]{a^2 \hat{Q}_\calE^2(u)} - 2\ln\rbrm[\big]{a^2\hat{Q}_\calE^2(u)}
        =\ell_\psi\rbrm[\big]{[a^2 \hat{Q}_\calE^2(u)]^{-1}},
    \end{align*}
    due to \eqref{eq:l_theta-self} of \cref{lem:log-poly-basic} below. We additionally define $\mathsf{H}(1)=\lim_{u\to 1^-} \mathsf{H}(u)=\ell_\psi(0)=0$. This means %
    \begin{align*}
        \mathsf{H}(u)
        & =\begin{cases}
            \ell_\psi\rbrm[\big]{[a^2 \hat{Q}_\calE^2(u)]}, & 0\leq u<1-\sigma, \\
            \ell_\psi\rbrm[\big]{[a^2 \hat{Q}_\calE^2(u)]^{-1}}, & 1-\sigma\leq u\leq 1.
        \end{cases}\\
        & =\begin{cases}
            \ell_\psi\rbrm[\Big]{\rbrm[\big]{\frac{1-u}{\sigma}}^{-2/\beta}}, & 0\leq u<1-\sigma, \\
            \ell_\psi\rbrm[\Big]{\rbrm[\big]{\frac{1-u}{\sigma}}^{2/\beta}}, & 1-\sigma\leq u\leq 1.
        \end{cases}
    \end{align*}
    Equivalently, 
    \begin{align*}
        \mathsf{H}(1-u)
        & =\begin{cases}
            \ell_\psi\rbrm[\big]{(\sigma/u)^{2/\beta}}, & \sigma<u\leq 1, \\
            \ell_\psi\rbrm[\big]{(u/\sigma)^{2/\beta}}, & 0\leq u\leq \sigma.
        \end{cases}
    \end{align*}

    If we define $\cev{\mathsf{H}}(u)=\mathsf{H}(1-u)$, we notice that the function $\cev{\mathsf{H}}$ has finite total dyadic second-order difference $\normm[\big]{\cev{\mathsf{H}}}_{\Delta^2}\leq C(\psi)$, where $C(\psi)$ is a constant that only depends on $\psi$ and grows as $\calO((\ln \psi^{-1})^2)$. The analysis is the main technical difficulty of this section, and we defer it to \cref{lem:pareto-remainder-second-difference} below. From \cref{lem:finite-difference-reflection}, we know that $\normm{\mathsf{H}}_{\Delta^2}\leq C(\psi)$ as well.

    We are very close to applying \cref{thm:En-2nd-order} on $\mathsf{H}$; however, the theorem requires the function to evaluate to $0$ on both endpoints $\{0, 1\}$, which $\mathsf{H}$ does not satisfy. We define an affine function $\bar{\mathsf{H}}(u)$ that satisfies
    \begin{align*}
        \bar{\mathsf{H}}(0)=\mathsf{H}(0)=H(0),
        \qquad
        \bar{\mathsf{H}}(1)=\mathsf{H}(1)=0,
    \end{align*}
    or $\bar{\mathsf{H}}(u)=H(0)(1-u)$. Since $H(0)\leq 0$, the function $\bar{\mathsf{H}}$ is monotonically increasing, and yet again \cref{lem:En-mono} implies that $E_n(\bar{\mathsf{H}})\leq 0$, or equivalently $E_n(\mathsf{H})\leq E_n(\mathsf{H}-\bar{\mathsf{H}})$.

    Now, the function $\mathsf{H}-\bar{\mathsf{H}}$:
    \begin{itemize}
        \item equals 0 on $\{0, 1\}$, by the definition of $\bar{\mathsf{H}}$;
        \item has the same second-order differences as $\mathsf{H}$, because $\bar{\mathsf{H}}$ is linear and satisfies $\Delta^2_\delta\bar{\mathsf{H}}\equiv 0$, so $\normm[\big]{\mathsf{H}-\bar{\mathsf{H}}}_{\Delta^2}=\normm{\mathsf{H}}_{\Delta^2}\leq C(\psi)$.
    \end{itemize}
    So, \cref{thm:En-2nd-order} proves that $\absm[\big]{E_n(\mathsf{H}-\bar{\mathsf{H}})}\leq \frac{C(\psi)}{n}$ for all positive integer $n$. Therefore,
    \begin{align*}
        E_n(h)\leq E_n(H)\leq E_n(\mathsf{H})\leq E_n(\mathsf{H}-\bar{\mathsf{H}})\leq \frac{C(\psi)}{n},
    \end{align*}
    where the first three inequalities hold from monotonicity of their differences, and the last inequality holds from the bounded $\Delta^2$ norm.
    This proves the theorem with $C=C(\psi)=12(\ln(\sin\psi)^{-1})^2+53\ln(\sin\psi)^{-1}+2$, as established in \cref{lem:pareto-remainder-second-difference}.
\end{proof}

\begin{applemma}\label{lem:log-poly-basic}
    Let $\ell_\psi(z)=\ln f_\psi(z)$. Then:
    \begin{enumerate}
        \item For $z>0$,
            \begin{align}
                \ell_\psi(z)=2\ln z+\ell_\psi(1/z).\label{eq:l_theta-self}
            \end{align}
        \item For $z\in[0,1]$,
            \begin{align}
                -2\ln(\sin\psi)^{-1} \leq{} &\ell_\psi(z)\leq 0,\label{eq:l_theta-0th}
                \\
                \frac{-4\ln(\sin\psi)^{-1}}{\cos\psi}z\leq{}&\ell_\psi(z).\label{eq:l_theta-1st-app}
            \end{align}
        \item For $z\in (0, 1]$,
            \begin{align}
                \abs{\ell'_\psi(z)}&\leq \frac{2}{\sin\psi},\label{eq:l_theta-1st}\\
                \abs{\ell''_\psi(z)}&\leq \frac{2}{(\sin\psi)^2}.\label{eq:l_theta-2nd}
            \end{align}
        \item For $z\in (0, \frac12]$,
            \begin{align}
                \abs{\ell'_\psi(z)}&\leq 4,\label{eq:l_theta-1st-half}\\
                \abs{\ell''_\psi(z)}&\leq 8,\label{eq:l_theta-2nd-half}
            \end{align}
        \item The function $\ell_\psi$ has a bounded total variation
        \begin{align}
            V_0^1(\ell_\psi)\leq4\ln(\sin\psi)^{-1}.\label{eq:l_tehta-bv}
        \end{align}
    \end{enumerate}
\end{applemma}
\begin{proof}
    The first statement follows from $f_\psi(z)=z^2 f_\psi(1/z)$.
    
    For \eqref{eq:l_theta-0th}, observe that $f_\psi$ is convex and is maximized over $[0,1]$ at an endpoint. Since $f_\psi(0)=1$ and $f_\psi(1)=2-2\cos\psi\leq 1$ , we have $f_\psi(z)\leq 1$ for $z\in[0,1]$; solving the first-order condition $f_\psi'(z)=0$ yields $z^*=\cos\psi$, so the minimum of $f_\psi$ is $f_\psi(z^*)=1-(\cos\psi)^2=(\sin\psi)^2$; thus $\ell_\psi(z)\geq \ln(\sin\psi)^2=-2\ln(\sin\psi)^{-1}$. (This also implies that $z^*=\cos\psi\geq \frac12$, and the function $f_\psi$ is monotonically decreasing on $[0, \frac12]$, which we will use later.)
    
    For \eqref{eq:l_theta-1st-app}, we notice that $\frac{\ell_\psi(z)}{z}=\frac{\ln f_\psi(z)}{1-f_\psi(z)}\frac{1-f_\psi(z)}{z}$. The first term is of the form $t(x)=\frac{\ln x}{1-x}$ and negative; the gradient of this function is $t'(x)=\frac{1-x+x\ln x}{x(1-x)^2}$, which is positive for $x\in (0, 1)$, thus $t(f(x))$ is minimized whenever $f(x)$ is minimized, and $t(f(x))\geq t((\sin\psi)^2)=\frac{-2\ln(\sin\psi)^{-1}}{(\cos\psi)^2}$; the second term is $\frac{1-f_\psi(z)}{z}=\frac{2\cos\psi\cdot z-z^2}{z}=2\cos\psi-z\leq 2\cos\psi$ is positive. Therefore, $\frac{\ell_\psi(z)}{z}\geq \frac{-2\ln(\sin\psi)^{-1}}{(\cos\psi)^2}\cdot(2\cos\psi)=\frac{-4\ln(\sin\psi)^{-1}}{\cos\psi}$.

    For the derivative bounds, we notice that
    \begin{align*}
        \ell'_\psi(z)&=\frac{f'_\psi(z)}{f_\psi(z)}=\frac{2z-2\cos\psi}{1-2\cos\psi z+z^2}=\frac{1}{z-e^{i\psi}}+\frac{1}{z-e^{-i\psi}},\\
        \ell''_\psi(z)&=-\frac{1}{(z-e^{i\psi})^2}-\frac{1}{(z-e^{-i\psi})^2}.
    \end{align*}
    We have $\abs{z-e^{\pm i\psi}}^2=(z-\cos\psi)^2+(\mp \sin\psi)^2=1-2z\cos\psi+z^2=f_\psi(z)$. Thus, 
    $\absm[\big]{\ell'_\psi(z)}
     \leq 
      \frac{1}{\abs{z-e^{i\psi}}}
     +\frac{1}{\abs{z-e^{-i\psi}}}
     =\frac{2}{\sqrt{f_\psi(z)}}$. We can prove \eqref{eq:l_theta-1st} with the fact that $f_\psi(z)\geq(\sin\psi)^2$.
    Furthermore, when $z\in [0, \frac12]$, we know that $f_\psi(z)$ is monotonically decreasing, and thus $f_\psi(z)\geq f_\psi(\frac12)=\frac54-\cos\psi\geq \frac14$; this proves \eqref{eq:l_theta-1st-half}.

    Similarly, we have $\absm[\big]{\ell''_\psi(z)}
     \leq 
      \frac{1}{\abs{z-e^{i\psi}}^2}
     +\frac{1}{\abs{z-e^{-i\psi}}^2}
     =\frac{2}{f_\psi(z)}$, and \eqref{eq:l_theta-2nd} and \eqref{eq:l_theta-2nd-half} can be proved with again plugging in $f_\psi(z)\geq(\sin\psi)^2$ on $[0, 1]$ and $f_\psi(z)\geq \frac14$ on $[0, \frac12]$ respectively.
    
    To prove the total variation, we notice that $\ell_\psi$ is decreasing on the left of $z^*$ and increasing on the right. Because its range is $[-2\ln(\sin\psi)^{-1}, 0]$, it varies at most twice through this range, making the total variation bounded by $4\ln(\sin\psi)^{-1}$.
\end{proof}

\begin{applemma}\label{lem:log-poly-psi}
    Let
    \begin{align*}
        \Psi_{\beta,\psi}(t)=
        \begin{cases}
            \ell_\psi(t^{2/\beta}),&0\leq t<1,\\
            \ell_\psi(t^{-2/\beta}),&t\geq 1.
        \end{cases}
    \end{align*}
    Then,
    \begin{align}
        \Psi_{\beta,\psi}(t)&\leq 0,\label{eq:psi-bound-nonpositive}\\
        \abs{\Psi_{\beta,\psi}(t)}&\leq \frac{4\ln(\sin\psi)^{-1}}{\cos\psi} \min\{t^{2/\beta},t^{-2/\beta}\},&&\textrm{\upshape if } t>0,\label{eq:psi-bound-0th}\\
        \absm[\big]{\Psi'_{\beta,\psi}(t)}&\leq \frac{4}{\sin\psi},
        &&\textrm{\upshape if }t\in (0, 1)\cup(1,+\infty),
        \label{eq:psi-bound-1st} \\
        \absm[\big]{\Psi''_{\beta,\psi}(t)}&\leq 24 t^{2/\beta-2},
        &&\textrm{\upshape if }t\in (0, 1/2]
        \label{eq:psi-bound-2nd-left},\\
        \absm[\big]{\Psi''_{\beta,\psi}(t)}&\leq 40 t^{-2/\beta-2},
        &&\textrm{\upshape if }t\in [2,+\infty)
        \label{eq:psi-bound-2nd-right},\\
        \absm[\big]{\Psi''_{\beta,\psi}(t)}&\leq \frac{19}{(\sin\psi)^2},&&\textrm{\upshape if }t\in [1/3, 1)\cup(1,3].\label{eq:psi-bound-2nd-near}
    \end{align}
    Also, the function $\Psi_{\beta,\psi}$ has total variation
    \begin{align}        V_0^\infty(\Psi_{\beta,\psi})\leq8\ln(\sin\psi)^{-1}.\label{eq:psi-bv}
    \end{align}
\end{applemma}
\begin{proof}
    The pointwise bounds \eqref{eq:psi-bound-nonpositive} and \eqref{eq:psi-bound-0th} follow from \eqref{eq:l_theta-0th} and \eqref{eq:l_theta-1st-app}.

    By differentiating the definition of $\Psi$, we have, for $t\in (0, 1)$,
    \begin{align}
        \Psi'_{\beta,\psi}(t)&=\frac{2}{\beta}t^{2/\beta-1}\ell'_\psi(t^{2/\beta}),\label{eq:step-psi-left-1st}\\
        \Psi''_{\beta,\psi}(t)&=\frac{2}{\beta}\rbrm[\Big]{\frac{2}{\beta}-1}t^{2/\beta-2}\ell'_\psi(t^{2/\beta})+\frac{4}{\beta^2}t^{4/\beta-2}\ell''_\psi(t^{2/\beta}),
        \label{eq:step-psi-left-2nd}
    \end{align}
    and for $t\in (1, +\infty)$,
    \begin{align}
        \Psi'_{\beta,\psi}(t)&=-\frac{2}{\beta}t^{-2/\beta-1}\ell'_\psi(t^{-2/\beta}),\label{eq:step-psi-right-1st}\\
        \Psi''_{\beta,\psi}(t)&=-\frac{2}{\beta}\rbrm[\Big]{-\frac{2}{\beta}-1}t^{-2/\beta-2}\ell'_\psi(t^{-2/\beta})+\frac{4}{\beta^2}t^{-4/\beta-2}\ell''_\psi(t^{-2/\beta}).
        \label{eq:step-psi-right-2nd}
    \end{align}

    To prove \eqref{eq:psi-bound-1st}, we notice that $t^{2/\beta-1}<1$ when $t\in (0, 1)$ and $t^{-2/\beta-1}<1$ when $t\in (1,+\infty)$, so
    \begin{align*}
        \absm[\big]{\Psi'_{\beta,\psi}(t)}\leq\frac{2}{\beta}\absm[\big]{\ell'_\psi\rbrm[\big]{\min\{t^{2/\beta},t^{-2/\beta}\}}}\leq \frac{2}{\beta}\frac{2}{\sin\psi}\leq \frac{4}{\sin\psi}.
    \end{align*}

    When $0<t\leq\frac12$, the argument $z=t^{2/\beta}$ is in the range $(0, \frac12]$ because $2/\beta>1$; using \eqref{eq:step-psi-left-2nd}, \eqref{eq:l_theta-1st-half}, \eqref{eq:l_theta-2nd-half} and $\frac{2}{\beta}\in [5/4, 2]$, we have
    \begin{align*}
        \absm[\big]{\Psi''_{\beta,\psi}(t)}\leq 8t^{2/\beta-2}+32t^{4/\beta-2}=8t^{2/\beta-2}\rbrm[\big]{1+4t^{2/\beta}}\leq 24t^{2/\beta-2}.
    \end{align*}
    When $t\geq2$, we similarly use \eqref{eq:step-psi-right-2nd}, \eqref{eq:l_theta-1st-half}, and \eqref{eq:l_theta-2nd-half} to derive
    \begin{align*}
        \absm[\big]{\Psi''_{\beta,\psi}(t)}\leq 24t^{-2/\beta-2}+32t^{-4/\beta-2}=8t^{-2/\beta-2}\rbrm[\big]{3+4t^{-2/\beta}}\leq 40t^{-2/\beta-2}.
    \end{align*}
    Thus \eqref{eq:psi-bound-2nd-left} and \eqref{eq:psi-bound-2nd-right} are proved.

    When $\frac13\leq t<1$, we have $t^{2/\beta-1}\in \sbrm[\big]{\frac13, 1}$, $t^{2/\beta-2}\in \sbrm[\big]{1, 3^{3/4}}$, and $t^{4/\beta-2}\in \sbrm[\big]{\frac19, 1}$. In this case, together with \eqref{eq:l_theta-1st} and \eqref{eq:l_theta-2nd}, we have
    \begin{align*}
         \absm[\big]{\Psi''_{\beta,\psi}(t)}&\leq
            2\cdot1\cdot 3^{3/4}\cdot \frac{2}{\sin\psi}
            + 4\cdot 1\cdot \frac{2}{(\sin\psi)^2}
            =\frac{1}{(\sin\psi)^2}(4\cdot 3^{3/4}\sin\psi+8)
            \leq \frac{16}{(\sin\psi)^2}.
    \end{align*}
    When $1<t\leq 3$, we have $t^{-2/\beta-1}\in \sbrm[\big]{\frac1{27}, 1}$, $t^{-2/\beta-2}\in \sbrm[\big]{\frac1{81}, 1}$, and $t^{-4/\beta-2}\in \sbrm[\big]{\frac1{729}, 1}$. Again, together with \eqref{eq:l_theta-1st} and \eqref{eq:l_theta-2nd}, we have
    \begin{align*}
         \absm[\big]{\Psi''_{\beta,\psi}(t)}&\leq
            2\cdot3\cdot 1\cdot \frac{2}{\sin\psi}
            + 4\cdot 1\cdot \frac{2}{(\sin\psi)^2}
            =\frac{1}{(\sin\psi)^2}(12\sin\psi+8)
            \leq \frac{19}{(\sin\psi)^2}.
    \end{align*}
    These together prove \eqref{eq:psi-bound-2nd-near}.

    To prove that $\Psi_{\beta,\psi}$ has bounded total variation, we see that the value $\begin{cases}
        t^{2/\beta},&0\leq t<1,\\
        t^{-2/\beta},&t\geq 1.
    \end{cases}$
    increases from 0 to 1 and then decreases to 0, so the total variation of $\Psi_{\beta,\psi}$ is twice the total variation of $\ell_\psi$, which is bounded as \eqref{eq:l_tehta-bv}.
\end{proof}
\begin{applemma}\label{lem:pareto-remainder-second-difference}
    Suppose $0\leq \sigma<1$. Define the function %
    \begin{align*}
        f(u)\defeq
        \Psi_{\beta,\psi}(u/\sigma)
        =
        \begin{cases}
            \ell_\psi\rbrm[\big]{(u/\sigma)^{2/\beta}},&0\leq u<\sigma,\\
            \ell_\psi\rbrm[\big]{(\sigma/u)^{2/\beta}},&\sigma\leq u\leq1.
        \end{cases}
    \end{align*}
    If $\sigma=0$, let $f\equiv 0$. Then, we have
    \begin{align*}
        \norm{f}_{\Delta^2}\leq 12(\ln(\sin\psi)^{-1})^2+53\ln(\sin\psi)^{-1}+2,
    \end{align*}
    for every $\sigma \in [0, 1)$.
\end{applemma}
\begin{proof}
    The proof makes heavy use of \cref{lem:log-poly-basic} and \ref{lem:log-poly-psi} below.

    The case $\sigma=0$ is immediate, so assume $\sigma>0$.
    We will prove it by bounding $\norm{f}_{\Delta^2,m}$ for each $m$. Let $m$ be a nonnegative integer and set $\delta=2^{-m-1}$. Choose an integer $m_0\geq 0$ such that $2^{-m_0-1}<\sigma\leq 2^{-m_0}$.

    \paragraph{Bounding all levels.} We first see that $\norm{f}_{\Delta^2,m}\leq 4\ln(\sin\psi)^{-1}$ for any $m$. This is because the total variation of $f$ on $[0, 1]$ is equal to the total variation of $\Psi_{\beta,\psi}$ on $[0, \frac{1}{\sigma}]$, and is at most the total variation of $\Psi_{\beta,\psi}$ on $[0, \infty)$; the bound is then implied by \cref{lem:finite-difference-bounded-variation}.

    \paragraph{Bounding coarse levels.} We then are going to bound $\norm{f}_{\Delta^2,m}$ when $m \leq m_0 - 1$.

    Let $\eta=\frac\delta\sigma$. Since $m \leq m_0 - 1$, we know that $\eta\geq 1$. 
    Since $\Psi(t)\leq 0$ for all $t$, we have $f(u)\leq 0$ for all $u\in [0, 1]$, so
    \begin{align*}
        \norm{f}_{\Delta^2,m} 
        & = \frac12\sum_{a=0}^{2^m-1}\absm[\big]{
            2f\rbrm{M_a}
            -f\rbrm{L_a}
            -f\rbrm{R_a}
        } \\
        & \leq \frac12\sum_{a=0}^{2^m-1} \rbrm[\big]{
            \abs{2f\rbrm{M_a}}
            +\abs{f\rbrm{L_a}}
            +\abs{f\rbrm{R_a}}
        } = -\frac12\sum_{a=0}^{2^m-1} \rbrm[\big]{
            {2f\rbrm{M_a}}
            +{f\rbrm{L_a}}
            +{f\rbrm{R_a}}
        } \\
        & \leq -\sum_{k=0}^{2^{m+1}} 2 f(k\delta)
          = -2\sum_{k=0}^{2^{m+1}}\Psi_{\beta,\psi}(k\eta)
          \leq -2 \sum_{k=0}^{\infty} \Psi_{\beta,\psi}(k\eta) \\
        & \leq \frac{8\ln(\sin\psi)^{-1}}{\cos\psi} \sum_{k=1}^\infty \min\cbrm[\big]{
            (k\eta)^{2/\beta},(k\eta)^{-2/\beta}
        } \\
        & \leq 16\ln(\sin\psi)^{-1} \sum_{k=1}^\infty (k\eta)^{-2/\beta},
    \end{align*}
    where we used \eqref{eq:psi-bound-0th} in \cref{lem:log-poly-psi}. The term $k=0$ can be discarded because $\Psi_{\beta,\psi}(0)=0$. The last step is due to $k\eta\geq 1$ for any $k\geq 1$ and the condition $\cos\psi\geq\frac12$. Taking the last summation over $k=1,2,\dots$ yields $\eta^{-2/\beta} \zeta(2/\beta)$, where $\zeta(\cdot)$ is the Riemann zeta function. As $\beta\in (1, \frac85]$, this implies
    \begin{align*}
        \norm{f}_{\Delta^2,m} \leq {} & 16\ln(\sin\psi)^{-1} \eta^{-1} \zeta(5/4)\leq 80 \ln(\sin\psi)^{-1} \eta^{-1} \\
        & = 160\ln(\sin\psi)^{-1} \frac{\sigma}{2^{-m}}\leq 160\ln(\sin\psi)^{-1} 2^{m-m_0},
    \end{align*}
    where we used \cref{lem:zeta-bound} to bound $\zeta(2/\beta)$.

    \paragraph{Bounding fine levels.} We next bound $\norm{f}_{\Delta^2,m}$ when $m \geq m_0 + 2$.

    Again, let $\eta=\frac\delta\sigma$. Since $m \geq m_0 + 2$, we know that $\eta< 2^{m_0-m}\leq\frac14$. We split all intervals of the form $[L_a, R_a]$ for $a=0,1,\dots,2^m-1$ into four cases.
    \begin{itemize}
        \item Consider all intervals such that $\{0,\frac\sigma 2,\sigma,2\sigma\}\cap[L_a,R_a]\neq \varnothing$. There are at most 7 such intervals; for each such interval, we apply the second condition from \cref{lem:finite-difference-smoothness}, and using the fact that $f$ is $\frac{4}{\sigma\sin\psi}$-Lipschitz (from \eqref{eq:psi-bound-1st} and the definition of $f$), we can bound $\frac12\absm[\big]{\Delta^2_\delta f(2a\delta)}\leq \frac12 2\cdot \frac{4}{\sigma\sin\psi}\cdot \delta=\frac{4}{\sin\psi} \eta$. All 7 intervals together contribute at most $\frac{28}{\sin\psi}\eta$.
        \item Consider any interval with $0<L_a<R_a<\frac\sigma2$. On the interval, if $u\in [L_a, R_a]$, from \eqref{eq:psi-bound-2nd-left},
        \begin{align*}
            \abs{f''(u)}
            & \leq \frac{1}{\sigma^2}\absm[\big]{\Psi_{\beta,\psi}''(u/\sigma)}
              \leq \frac{1}{\sigma^2}\cdot 24\rbrm{u/\sigma}^{2/\beta-2}
              = \frac{24}{\sigma^{2/\beta}}u^{2/\beta-2}\\
            & \leq \frac{24}{\sigma^{2/\beta}}(2a\delta)^{2/\beta-2}.
        \end{align*}
        The third statement of \cref{lem:finite-difference-smoothness} implies that
        \begin{align*}
            \frac12 \absm[\big]{\Delta^2_\delta f(2a\delta)}
            & \leq \frac12 \frac{24}{\sigma^{2/\beta}}(2a\delta)^{2/\beta-2} \delta^2
            \leq 12\frac{\delta^{2/\beta}}{\sigma^{2/\beta}} a^{2/\beta-2}=12\eta^{2/\beta}a^{2/\beta-2}.
        \end{align*}
        Summing the RHS over $1\leq a<\frac{\sigma}{4\delta}-1=\frac1{4\eta}-1$ yields
        \begin{align*}
            \frac12\sum_{0<L_a<R_a<\frac\sigma2}\absm[\big]{\Delta^2_\delta f(2a\delta)}
            & \leq 12\eta^{2/\beta}\sum_{a=1}^{\floorm{\frac{1}{4\eta}}-1}a^{2/\beta-2} \\
            & \leq 12\eta^{2/\beta}\int_{0}^{1/(4\eta)}a^{2/\beta-2}\dd{a} 
              = 12\eta^{2/\beta}\frac{\rbrm[\big]{\frac{1}{4\eta}}^{2/\beta-1}}{2/\beta-1} \\
            & = \frac{12\cdot 4^{1-2/\beta}}{2/\beta-1}\eta
              \leq 24\sqrt2 \eta \leq 34\eta.
        \end{align*}
        \item Consider any interval with $L_a>2\sigma$. Similarly, if $u\in [L_a, R_a]$, from \eqref{eq:psi-bound-2nd-right},
        \begin{align*}
            \abs{f''(u)}
            & \leq \frac{1}{\sigma^2}\absm[\big]{\Psi_{\beta,\psi}''(u/\sigma)}
              \leq \frac{1}{\sigma^2}\cdot 40\rbrm{u/\sigma}^{-2/\beta-2}
              \leq 40\sigma^{2/\beta} (2a\delta)^{-2/\beta-2}.
        \end{align*}
        Again, using the third statement of \cref{lem:finite-difference-smoothness}, we have
        \begin{align*}
            \frac12 \absm[\big]{\Delta^2_\delta f(2a\delta)}
            & \leq \frac12 40\sigma^{2/\beta} (2a\delta)^{-2/\beta-2} \delta^2
            \leq \frac{5}{2^{5/4}}\frac{\sigma^{2/\beta}}{\delta^{2/\beta}} a^{-2/\beta-2}\leq\frac52\eta^{-2/\beta}a^{-2/\beta-2}.
        \end{align*}
        Summing the RHS over $a>\frac{2\sigma}{2\delta}=\frac1{\eta}$ yields
        \begin{align*}
            \frac12\sum_{L_a>2\sigma}\absm[\big]{\Delta^2_\delta f(2a\delta)}
            & \leq \frac52\eta^{-2/\beta}\sum_{a=\ceilm{1/\eta}}^{2^m-1}a^{-2/\beta-2} \\
            & \leq \frac52\eta^{-2/\beta}\int_{1/\eta-1}^{+\infty}a^{-2/\beta-2}\dd{a}
              = \frac52\eta^{-2/\beta}\frac{\rbrm[\big]{\frac{1}{\eta}-1}^{-2/\beta-1}}{2/\beta+1} \\
            & \leq \frac52\eta^{-2/\beta}\frac{\rbrm[\big]{\frac{3}{4\eta}}^{-2/\beta-1}}{2/\beta+1} 
              = \frac{(5/2)\cdot (4/3)^{2/\beta+1}}{2/\beta+1}\eta
              \leq \frac{160\sqrt2}{3^{17/4}} \eta \leq 3\eta.
        \end{align*}
        \item Finally, consider all intervals with $\frac\sigma2<L_a<R_a<2\sigma$ such that $\sigma\not\in [L_a,R_a]$. On these intervals, we have $\abs{f''(u)}\leq \frac{1}{\sigma^2}\absm[\big]{\Psi_{\beta,\psi}''(u/\sigma)} \leq \frac{19}{\sigma^2 (\sin\psi)^2}$, and the third statement of \cref{lem:finite-difference-smoothness} guarantees that $\frac12 \absm[\big]{\Delta^2_\delta f(2a\delta)}\leq \frac12\frac{19}{\sigma^2 (\sin\psi)^2}\delta^2\leq \frac{19}{2(\sin\psi)^2}\eta^2$. There are at most $\frac{2\sigma-\frac\sigma2}{2\delta}=\frac34 \eta^{-1}$ such intervals, and their total contribution bounded by $\frac{19}{2(\sin\psi)^2}\cdot \frac34\eta\leq\frac{8}{(\sin\psi)^2}\eta$.
    \end{itemize}

    Totaling all four cases together, we have
    \begin{align*}
        \norm{f}_{\Delta^2,m} 
        & = \frac12\sum_{a=0}^{2^m-1} \absm[\big]{\Delta^2_\delta f(2a\delta)} 
         \leq \frac{28}{\sin\psi}\eta+34\eta+3\eta+\frac{8}{(\sin\psi)^2}\eta \\
        & \leq \rbrm[\Big]{37+\frac{28}{\sin\psi}+\frac{8}{(\sin\psi)^2}}\eta
          < \rbrm[\Big]{37+\frac{28}{\sin\psi}+\frac{8}{(\sin\psi)^2}} 2^{m_0-m}.
    \end{align*}

    \paragraph{Putting everything together.}
    Let $K=2\log_2(\sin \psi)^{-1}+6$. We have $2^K=\frac{64}{(\sin\psi)^2}\geq 37+\frac{28}{\sin\psi}+\frac{8}{(\sin\psi)^2}$ for all $(\sin\psi)^{-1}\geq \frac{2}{\sqrt3}$, so $\norm{f}_{\Delta^2,m}\leq \rbrm[\big]{37+\frac{28}{\sin\psi}+\frac{8}{(\sin\psi)^2}} 2^{m_0-m}\leq 2^{m_0+K-m}$ whenever $m\geq m_0+K$. In the sums below, a sum with upper limit below its lower limit is zero. Therefore,
    \begin{align*}
        \norm{f}_{\Delta^2} 
        & = \sum_{m=0}^\infty \norm{f}_{\Delta^2,m} 
          = \sum_{m=0}^{m_0-6} \norm{f}_{\Delta^2,m} + \sum_{m=\max\{0,m_0-5\}}^{m_0+\ceil{K}-1} \norm{f}_{\Delta^2,m}+\sum_{m=m_0+\ceil{K}}^{\infty} \norm{f}_{\Delta^2,m} \\
        & \leq \sum_{m=0}^{m_0-6} 160\ln(\sin\psi)^{-1} 2^{m-m_0} + (\ceil{K}+5)\cdot 4\ln(\sin\psi)^{-1} + \sum_{m=m_0+\ceil{K}}^{\infty} 2^{m_0+K-m} \\
        & \leq 5\ln(\sin\psi)^{-1}  + (\floor{K}+6)\cdot 4\ln(\sin\psi)^{-1} + 2 
          \leq (4K+29)\ln(\sin\psi)^{-1}+2 \\
        & = (8\log_2(\sin \psi)^{-1}+53)\ln(\sin\psi)^{-1}+2 \\
        & \leq 12(\ln(\sin\psi)^{-1})^2+53\ln(\sin\psi)^{-1}+2,
    \end{align*}
    where we used $\floor{K}\leq K$ and $8 \log_2 e<12$. This proves the final claim.
\end{proof}

\section{Power-law stepsizes in OG}\label{ah:og}

In this section, we provide preliminary theoretical results when our power-law stepsize schedule is applied to the Optimistic Gradient (OG) algorithm. We only derive weaker, randomized geometric-mean results for the single-stepsize version of OG, as the purpose of this section is to demonstrate that our methodology has the potential of generalizing to other algorithms.

\paragraph{Algorithm definition} We use the standard Optimistic Gradient (OG) algorithm:
\begin{align}
    \bigg\{\begin{aligned}
        x_{t+\frac 12} &= x_t-\eta_t\nabla_x \ell(x_{t-\frac 12},y_{t-\frac 12}),\\
        y_{t+\frac 12} &= y_t+\eta_t\nabla_y \ell(x_{t-\frac 12},y_{t-\frac 12}),
    \end{aligned}\quad
    \bigg\{\begin{aligned}
        x_{t+1} &= x_t-\eta_t\nabla_x \ell(x_{t+\frac 12},y_{t+\frac 12}),\\
        y_{t+1} &= y_t+\eta_t\nabla_y \ell(x_{t+\frac 12},y_{t+\frac 12}),
    \end{aligned} \quad t=0,1,2,...
    \label{eq:OG}
\end{align}
with the initialization $(x_{-\frac12},y_{-\frac12})=(x_0,y_0)$. Note that we use the gradient at the last update step in place of the gradient of the extrapolation step. If we write $z_t=(x_t,y_t)$, OG can be concisely represented with the gradient operator $G_\ell$ as
\begin{align}
    z_{t+\frac12}&=z_t-\eta_t G_\ell(z_{t-\frac12}), & %
    z_{t+1}&=z_t-\eta_t G_\ell(z_{t+\frac 12}).\label{eq:og-def}
\end{align}

\paragraph{Reformulation as optimization}
Similar to EG, OG also makes use of only affine operations and calls to the gradient operator $G_\ell$. We claim that OG also satisfies \cref{lem:decompose-translation} and \cref{lem:decompose}; the formal proofs are omitted. They allow us to focus on the $1\times 1$ case, which, similar to \cref{lem:opt-1x1}, leads to the following lemma.
\begin{applemma}\label{lem:og-opt-1x1}
    Let $(\dR, \dR, \ell(x,y)=axy)$ be a bilinear min-max optimization problem. The trajectory of OG on this problem satisfies the following:
    \begin{align}
        \GN(z_T)
        \leq
            \norm{z_0}
                \cdot \abs{a} \prod_{t=0}^{T-1} \sqrt{\frac{4\eta_t^4 a^4+1+\sqrt{(4\eta_t^4 a^4+1)^2-4\eta_t^2 a^2}}{2}}.
        \label{eq:og-1x1}
    \end{align}
\end{applemma}
\begin{proof}
    Following the proof of \cref{lem:opt-1x1}, we know that $G_\ell$ has the matrix $\smimatrix{0&a\\-a&0}=\iota(ai)$. From \eqref{eq:og-def}, we know that
    \begin{align*}
        \begin{cases}
            z_{t+\frac12} = z_t-\eta_t G_\ell(z_{t-\frac12}),\\
            z_{t+1} = z_t-\eta_t G_\ell(z_{t+\frac12})=z_t-\eta_t G_\ell\rbrm[\big]{
                z_t-\eta_t G_\ell(z_{t-\frac12})
            }=(I-\eta_t G_\ell)z_t+\eta_t^2 G_\ell^2 z_{t-\frac12},
        \end{cases}
    \end{align*}
    which can be concisely written as
    \begin{align*}
        \begin{pmatrix}
            z_{t+\frac12} \\ z_{t+1}
        \end{pmatrix}
        =
        \begin{pmatrix}
            -\eta_t G_\ell     & I \\
             \eta_t^2 G_\ell^2 & I-\eta_t G_\ell
        \end{pmatrix}
        \begin{pmatrix}
            z_{t-\frac12} \\ z_{t}
        \end{pmatrix}.
    \end{align*}
    As $\smimatrix{z_{t-\frac12}-z_t\\z_t}=\smimatrix{1&-1\\0&1}{z_{t-\frac12}\\z_t}$, we have equivalently
    \begin{align*}
        \begin{pmatrix}
            z_{t+\frac12}-z_{t+1} \\ z_{t+1}
        \end{pmatrix}
        & =
        \begin{pmatrix}
            1&-1\\0&1
        \end{pmatrix}
        \begin{pmatrix}
            -\eta_t G_\ell     & I \\
             \eta_t^2 G_\ell^2 & I-\eta_t G_\ell
        \end{pmatrix}
        \begin{pmatrix}
            1&-1\\0&1
        \end{pmatrix}^{-1}
        \begin{pmatrix}
            z_{t-\frac12}-z_{t} \\ z_{t}
        \end{pmatrix} \\
        & = 
        \begin{pmatrix}
            -\eta_t G_\ell-\eta_t^2 G_\ell^2 & -\eta_t^2 G_\ell^2 \\
             \eta_t^2 G_\ell^2               & I-\eta_t G_\ell+\eta_t^2 G_\ell^2
        \end{pmatrix}
        \begin{pmatrix}
            z_{t-\frac12}-z_{t} \\ z_{t}
        \end{pmatrix} \\
        & = 
        \underbrace{\begin{pmatrix}
            \iota(-xi-(xi)^2) & \iota(-(xi)^2) \\
            \iota((xi)^2)               & \iota(1-xi+(xi)^2)
        \end{pmatrix}
        }_{\defeq M_t}
        \begin{pmatrix}
            z_{t-\frac12}-z_{t} \\ z_{t}
        \end{pmatrix},
    \end{align*}
    where we write $x=a\eta_t\in \dR$, and $\iota(xi)=\eta_t G_\ell$. This implies that
    \begin{align*}
        \norm{
            \begin{pmatrix}
                z_{t+\frac12}-z_{t+1} \\ z_{t+1}
            \end{pmatrix}
        } \leq \norm{M_t} \norm{
            \begin{pmatrix}
                z_{t-\frac12}-z_{t} \\ z_{t}
            \end{pmatrix}
        }
    \end{align*}
    
    Without proving, we claim that a block real matrix where each block is $\iota(a_{ij})$ has the same singular values as the complex matrix $\{a_{ij}\}_{i,j}$, with each multiplicity doubled. As a matrix's induced $L_2$ norm is its largest singular value, we have $\norm{M_t}=\norm{m_t}$, where $m_t=\smimatrix{-xi-(xi)^2&-(xi)^2\\(xi)^2&1-xi+(xi)^2}=\smimatrix{x^2-xi&x^2\\-x^2&-x^2-xi+1}$. To calculate $\norm{m_t}$, we consider the square of its singular values, which are the eigenvalues of $m_t^*m_t$. We have
    \begin{align*}
        \lambda I - m_t^* m_t
        & = 
        \lambda I - \begin{pmatrix}
            x^2+xi&-x^2\\x^2&-x^2+xi+1
        \end{pmatrix}
        \begin{pmatrix}
            x^2-xi&x^2\\-x^2&-x^2-xi+1
        \end{pmatrix}\\
        & =
        \begin{pmatrix}
            \lambda-2x^4-x^2&-2x^4-2ix^3+x^2\\-2x^4+2ix^3+x^2&\lambda-2x^4+x^2-1
        \end{pmatrix}.
    \end{align*}
    The determinant of the matrix in the equation simplifies to be
    \begin{align*}
        \det(\lambda I - m_t^* m_t)=\lambda^2-(4x^4+1)\lambda+x^2.
    \end{align*}
    The equation $\det(\lambda I - m_t^* m_t)=0$ is quadratic, and solves to
    \begin{align*}
        \lambda=\frac{4x^4+1\pm\sqrt{(4x^4+1)^2-4x^2}}{2}.
    \end{align*}
    Both roots of $\lambda$ are real, and the larger root takes the ``$+$'' in the ``$\pm$'' sign. As the singular values of $m_t$ are the square roots of the roots $\lambda$, we have
    \begin{align*}
        \norm{M_t}=\norm{m_t}&\leq \sqrt{\frac{4x^4+1+\sqrt{(4x^4+1)^2-4x^2}}{2}} \\
        & = \sqrt{\frac{4(\eta_t a)^4+1+\sqrt{(4(\eta_t a)^4+1)^2-4(\eta_t a)^2}}{2}}.
    \end{align*}
    Because the derivation above works for all $t$, we have
    \begin{align*}
        \norm{z_T} \leq \norm{
            \begin{pmatrix}
                z_{T-\frac12}-z_{T} \\ z_{T}
            \end{pmatrix}
        } & \leq
            \norm{M_{T-1}} \norm{M_{T-2}}\cdots \norm{M_{0}} \norm{
            \begin{pmatrix}
                z_{-\frac12}-z_{0} \\ z_{0}
            \end{pmatrix}
        } \\
        & = \norm{M_{T-1}} \norm{M_{T-2}}\cdots \norm{M_{0}} \norm{z_0}.
    \end{align*}
    Together with $\norm{G_\ell(z_T)}\leq \abs{a}\norm{z_T}$, this proves the lemma.
\end{proof}
\begin{remark}
    Unlike \cref{lem:opt-1x1}, this lemma is not tight, as \eqref{eq:og-1x1} is just an inequality. The matrices $\{M_t\}$ (or $\{m_t\}$) in general do not commute with each other, so the exact product $\smimatrix{z_{T-1/2}-z_T\\z_T}=M_{T-1} M_{T-2}\cdots M_0\smimatrix{0\\z_0}$ is not permutation-invariant, which prevents us from having an exact contraction formula. Hence, this section only provides a weak analysis on OG; \cref{ah:experiment-more} shows that OG empirically has a faster convergence rate than predicted in this section.
\end{remark}

The rest of the analysis follows the same framework as the EG cases before. Minimizing \eqref{eq:og-1x1} can be formulated as the following optimization problem:

\begin{align}
    \minimize_{\{\eta_t\}_{t\in [T]}}\; \maximize_{a\in [0,L]} \;
    \calROG_{T}(\{\eta_t\}_{t\in [T]},a)\defeq
    {
        a \prod_{t=0}^{T-1} \sqrt{\frac{4\eta_t^4 a^4+1+\sqrt{(4\eta_t^4 a^4+1)^2-4\eta_t^2 a^2}}{2}}
    }.
    \label{eq:og-minmax}
\end{align}

If $\eta_0,\eta_1,\dots$ are sampled from a distribution $\calE$, we again define $\calROG_T(\calE, a)$ as the geometric mean, or formally,
\begin{align}
    \ln \calROG_{T}(\calE, a)
    & \defeq
    \EE_{\eta_0,\dots,\eta_{T-1}\sim \calE}
    \sbrm[\Bigg]{
    \ln\rbrm[\bigg]{
        a \prod_{t=0}^{T-1} \sqrt{\frac{4\eta_t^4 a^4+1+\sqrt{(4\eta_t^4 a^4+1)^2-4\eta_t^2 a^2}}{2}}
    }} \\
    & =
    {
        \ln a + \frac T2\EE_{\eta\sim \calE}\sbrm[\bigg]{\ln\frac{4\eta^4 a^4+1+\sqrt{(4\eta^4 a^4+1)^2-4\eta^2 a^2}}{2}}
    },\label{eq:calL-og}
\end{align}
and, when $\calE=\Pareto(\etam/L, \beta)$,
\begin{align}
    & \EE_{\eta\sim \calE}\sbrm[\bigg]{\ln\frac{4\eta^4 a^4+1+\sqrt{(4\eta^4 a^4+1)^2-4\eta^2 a^2}}{2}}\notag\\
    {}={}& \frac{\beta}{2}\rbrm[\Big]{ \frac{\etam a}{L} }^{\beta} \int_{(\etam a/L)^2}^{\infty} \ln\frac{4t^2+1+\sqrt{(4t^2+1)^2-4t}}{2}t^{-\beta/2-1}\dd{t}.\label{eq:pareto-contraction-og}
\end{align}
In the last step, we use the substitution $t=\eta^2a^2$. These two equations are OG versions of \eqref{eq:calL-def} and \eqref{eq:pareto-contraction-int}, respectively.

\paragraph{Solving for a Pareto distribution}
Similar to the EG case, we define the integration in the RHS of \eqref{eq:pareto-contraction-og} as
\begin{align*}
    \IOG_{\beta}(s)=\int_{s}^{\infty} \ln\frac{4t^2+1+\sqrt{(4t^2+1)^2-4t}}{2}t^{-\beta/2-1}\dd{t},
\end{align*}
and we need to find $\IOG_{\beta}(s)<0$ for sufficiently small $s$. The logarithmic term $\ln\frac{4t^2+1+\sqrt{(4t^2+1)^2-4t}}{2}$ is negative on $(0, \frac14)$ and positive on $(\frac14,+\infty)$. Similar to the EG cases, the integral $\IOG_{\beta}$ can be evaluated into a closed form when $1<\beta<2$; however, this proof is significantly more complicated and requires special functions; see \cref{lem:m-func-og}, where $\IOG_\beta(0)$ is obtained with the substitution $s=\beta/2$:
\begin{align*}
    \IOG_\beta(0)=\begin{aligned}[t]
        &\frac{2^{\beta/2+1}\pi}{\beta\sin(\pi \beta/4)}
        \ghg43\rbrm[\Big]{
            -\frac{\beta}{4},\frac{\beta}{12},\frac{\beta+4}{12},\frac{\beta+8}{12};
            1, \frac12, \frac12;
            \frac{27}{64}
        }\\
        &{}-
        \frac{2^{\beta/2-4}(\beta+2)\pi}{\cos(\pi \beta/4)}
        \ghg43\rbrm[\Big]{
            \frac{2-\beta}{4},\frac{\beta+6}{12},\frac{\beta+10}{12},\frac{\beta+14}{12};
            1, \frac32, \frac32;
            \frac{27}{64}
        },
        \end{aligned}
\end{align*}
where $\ghg43$ is the generalized hypergeometric function. Using numerical calculation, we can solve for $\IOG_\beta(0)<0$; the result is $\beta>\beta^\star\approx 1.732$. Similar to the single- and double-stepsize EG analysis before, we have
\begin{applemma}\label{lem:og-cont}
    For any $\beta\in (\beta^\star, 2)$, there exists $\etam, C\in \dR_+$, and distribution $\calE=\Pareto(\etam/L, \beta)$, such that for all $a\in (0, L]$,
    \begin{align*}
        \calROG_{T}(\calE,a)\leq CLT^{-1/\beta}.
    \end{align*}
    As a corollary, for any $L$-smooth biaffine min-max optimization instance $(\dR^n, \dR^m, \ell)$ with an NE $z^\star$, if OG is initiated at $z_0$ with $\norm{z_0-z^\star}\leq R$ and samples its stepsizes $\eta_0,\eta_1,\dots$ from $\calE$, its gradient norm will satisfy
    \begin{align*}
        \EE[\ln \GN(z_T)]\leq \ln\sbrm[\big]{CLRT^{-1/\beta}}.
    \end{align*}
\end{applemma}
\begin{proof}
    Assume $\beta\in(\beta^\star, 2)$; from the definition of $\beta^\star$ above, we know that $\IOG_\beta(0)<0$. We have
    \begin{align*}
        \IOG_\beta(s)
        & = \IOG_\beta(0)-\int_{0}^{s} \ln\frac{4t^2+1+\sqrt{(4t^2+1)^2-4t}}{2}t^{-\beta/2-1}\dd{t} \\
        & \leq \IOG_\beta(0) - \int_{0}^{s} (-t) t^{-\beta/2-1}\dd{t}\tag{\cref{lem:og-first-order-0}} \\
        & \leq \IOG_\beta(0) + \int_{0}^{s} t^{-\beta/2}\dd{t}
          = \IOG_\beta(0)+\frac{s^{1-\beta/2}}{1-\beta/2}.
    \end{align*}
    Therefore, if we let $\etam=\rbrm[\big]{-(1-\frac\beta2)\frac{\IOG_\beta(0)}{2}}^{1/(2-\beta)}>0$, everything follows identically to the proof of \cref{lem:eg-cont} (from \eqref{eq:step:eg-cont-proof-first} onwards). We omit the rest of the proof.
\end{proof}

This lemma demonstrates that our randomized analysis generalizes to OG as well, achieving a $\calO(T^{-1/\beta^\star+\eps})\approx \calO(T^{-0.577})$ rate, which beats the $\calO(T^{-1/2})$ rate achieved by OG with a fixed stepsize. We believe that our other EG results, including the constant factor optimization by using a point-mass-Pareto-mixture and the discretization, are applicable for OG as well; however, we decide not to include them in the paper due to the analysis being suboptimal, and also to avoid adding tediousness to the already long appendices.

The suboptimality of OG appears to be an artifact of our analysis. This section only predicts a $\calO(T^{-1/\beta})$ for $\beta>\beta^\star\approx 1.732$; in comparison, in our numerical experiment (see \cref{fig:gradient-norm-compare-og} of \cref{ah:experiment-more} below), OG converges at a $\calO(T^{-0.66})$ rate with a $\beta=100/66\approx 1.515$ distribution, suggesting that there is still room for improvement for the theory of stepsize tuning for OG.

\section{Instance-dependent linear convergence}\label{ah:linear-rate}

 Traditional analysis of EG in bilinear cases focuses on \emph{instance-dependent linear convergence} in addition to worst-case $\order(T^{-1/2})$ convergence. In this section, we will show that our schedule also improves EG in this aspect; in particular, we achieve an improved, near-optimal instance-dependent rate \emph{without knowledge of the instance-dependent parameters}.

In this section, we have the following assumption.

\begin{appassumption}\label{ass:strongly}
    Assume $A\neq0$, and let $\sigma_{\min}(A)$ denote its smallest positive singular value. We assume $\sigma_{\min}(A)\geq \mu$ for some $\mu \in (0, L]$.
    The condition number is defined as $\kappa=\frac{L}{\mu}<\infty$.
\end{appassumption}

The case $A=0$ is trivial when a saddle point exists, since then $G_\ell\equiv0$. We assume $R>0$ in the logarithmic complexity statements; if $R=0$, the initial point is already a saddle point. We write $\ln_+x=\max\{0,\ln x\}$.

\Cref{lem:opt} can be rewritten as

\begin{applemma}\label{lem:opt-strongly}
    Let $(\dR^n, \dR^m, \ell)$ be an $L$-smooth biaffine min-max optimization problem with a saddle point $z^\star=(x^\star, y^\star)$, and $\ell(x,y)=x^\trans A y+p^\trans y+x^\trans q+c$ satisfies $\sigma_{\min}(A)\geq \mu$. The trajectory of EG in this problem starting from $z_0$ satisfies the following:
    \begin{align}
        \GN(z_T)
        \leq
            \norm{z_0-z^\star}
            \max_{a\in [\mu,L]} \cbrm[\bigg]{
                a \prod_{t=0}^{T-1} \sqrt{1+\eta_t(\eta_t-2\gamma_t) a^2+\eta_t^2\gamma_t^2 a^4}
            }.
        \label{eq:lem-opt-strongly}
    \end{align}
\end{applemma}

The lemma can be proved in the same way as \cref{lem:opt}, as $a$ represents a single decomposed game in \cref{lem:decompose}, which takes its values as positive singular values of $A$, within the range $[\mu,L]$. The full proof is omitted.

Armed with this lemma, we can extend the theoretical analysis of the paper. We illustrate this by extending the randomized stepsize schedule for single-stepsize EG.
\begin{applemma}\label{lem:eg-cont-strongly}
    For any $\beta\in (\frac32,2)$, there exists $\etam,C\in \dR_+$, such that the distribution $\calE=\Pareto(\etam/L, \beta)$ satisfies
    $
        \calR_T(\calE, a)\leq L\exp(-C\kappa^{-\beta}T)
    $ for all $a\in [\mu, L]$.
    As a corollary, for any $L$-smooth biaffine min-max optimization $(\dR^n, \dR^m, \ell)$ with a saddle point $z^\star$ that satisfies \cref{ass:strongly}, %
    if EG is initiated at $z_0$ with $\norm{z_0-z^\star}\leq R$ and samples its stepsizes from $\calE$, we have for all $T$,
    $
        \EE_{\eta_0,\dots,\eta_{T-1}\sim \calE}[\ln \GN(z_T)]\leq \ln\sbrm[\big]{LR\exp(-C\kappa^{-\beta}T)}.
    $
\end{applemma}
\begin{proof}
    We adapt the proof of \cref{lem:eg-cont}. Let $\etam=\rbrm[\big]{-(1-\frac\beta2)\frac{I_\beta(0)}{2}}^{1/(2-\beta)}>0$ as defined in the proof of \cref{lem:eg-cont}; we have $I_\beta((\etam a/L)^2)\leq I_\beta(\etam^2)\leq I_\beta(0)-\frac{I_\beta(0)}{2}=\frac12 I_\beta(0)<0$ from \eqref{eq:Ibeta-bound}. Then, from (\hyperref[eq:pareto-contraction-int']{\ref*{eq:pareto-contraction-int}}), we know that
    \begin{align*}
        \EE_{\eta\sim \calE}\sbrm[\big]{\ln(1-\eta^2a^2+\eta^4a^4)}
        \leq
        \frac{\beta}{2}\rbrm[\Big]{ \frac{\etam a}{L} }^{\beta}
        \rbrm[\Big]{\frac{I_\beta(0)}{2}}.
    \end{align*}
    Combining this with (\hyperref[eq:calL-def']{\ref*{eq:calL-def}}), we have
    \begin{align*}
        \ln \calR_T(\calE, a)
        & \leq {
            \ln L + \ln \frac{a}{L}
            +
            \frac T2 
            \frac{\beta}{2}\rbrm[\Big]{ \frac{\etam a}{L} }^{\beta}
            \rbrm[\Big]{\frac{I_\beta(0)}{2}}
        } \\
        & = \ln L + {
            \ln b
            +
            \frac T2 
            \frac{\beta}{2}\rbrm{ \etam b }^{\beta}
            \rbrm[\Big]{\frac{I_\beta(0)}{2}}
        } \tag{substitute $b=\frac aL$}\\
        & \leq \ln L +
            \frac {T \beta \etam^\beta I_\beta(0)}{8}
            b^\beta
    \end{align*}
    In the above, we substitute $b=\frac aL$, and use the fact $\ln b\leq 0$ in the last step. We note that $b=\frac aL\geq\frac \mu L=\frac{1}{\kappa}$. Therefore,
    \begin{align*}
        \ln \calR_T(\calE, a)
        & \leq 
            \ln L -
            \rbrm[\Big]{-\frac {\beta \etam^\beta I_\beta(0)}{8}}
            b^\beta T \\
        & \leq 
            \ln L -
            \rbrm[\Big]{-\frac {\beta \etam^\beta I_\beta(0)}{8}}
            \kappa^{-\beta} T
    \end{align*}
    Defining $C=-\frac {\beta \etam^\beta I_\beta(0)}8$ finishes the proof for the first statement.
    The second statement is an implication of \cref{lem:opt-strongly} and the definition of $\calR_T(\cdot, \cdot)$.
\end{proof}

Since the derandomization process \cref{lem:discretization} applies to all values of $a$, it still applies when $a\in [\mu,L]$. Therefore, we have the following result:

\begin{apptheorem}\label{thm:main-eg-strongly}
    For any $\eps\in (0, \frac{1}{10})$, there exist constants $\etam,C,C'>0$ and $\beta=\frac32+\eps$ such that the following holds:
    For any $L$-smooth biaffine min-max optimization $(\dR^n, \dR^m, \ell)$ with a saddle point $z^\star = (x^\star, y^\star)$ that satisfies \cref{ass:strongly}, if EG is initiated at $z_0=(x_0,y_0)$ with $\norm{z_0-z^\star}\leq R$ and uses stepsizes $\gamma_t = \eta_t = (1-\vdc_t)^{-1/\beta}\etam/L$, the $\vdc_t$-quantile of $\Pareto(\etam/L, \beta)$, then for all $T$,
    \begin{align*}
        \GN(z_T)\leq LR\exp(-C\kappa^{-\beta}T)\exp(C'/2),
    \end{align*}
    where $C,C'$ depend only on $\eps$. In other words, for any $\delta>0$, a sufficient iteration count to achieve $\GN(z_T)\leq \delta$ is
    $T=\order_\eps(\kappa^{3/2+\eps}[1+\ln_+(LR/\delta)])$, where $\order_\eps(\cdot)$ hides $\eps$-dependent factors.
\end{apptheorem}
\begin{proof}
    The first claim follows by combining the scalar bound in \cref{lem:eg-cont-strongly} with \cref{lem:discretization,lem:opt-strongly}. The second claim follows by setting
    \begin{align*}
        T=\max\cbrm[\Big]{0,\ceilm[\Big]{\frac1C\kappa^\beta\rbrm[\Big]{\frac{C'}{2}+\ln \frac{LR}{\delta}}}}
    \end{align*}
    and observing that the first claim implies $\GN(z_T)\leq \delta$.
\end{proof}

This theorem shows that, in addition to worst-case $\order_\eps(T^{-2/3+\eps})$ sublinear convergence, our stepsize schedule for single-stepsize EG also achieves linear convergence with an $\order_\eps(\kappa^{3/2+\eps}[1+\ln_+(LR/\delta)])$ iteration complexity, improving the quadratic dependence on $\kappa$ of fixed-stepsize EG.

An analogous theorem for double-stepsize EG can also be obtained, achieving an $\order_\eps(\kappa^{1+\eps}[1+\ln_+(LR/\delta)])$ iteration complexity.

\begin{apptheorem}\label{thm:main-dseg-strongly}
    For any $\eps\in (0, \frac{1}{9})$, there exist constants $\beta=1+\eps, \rho, \etam, C, C'>0$, such that for any $L$-smooth biaffine min-max optimization instance $(\dR^n, \dR^m, \ell)$ with a saddle point $z^\star$ that satisfies \cref{ass:strongly}, if EG is initiated at $z_0$ with $\norm{z_0-z^\star}\leq R$, and uses stepsizes $\eta_t = (1-\vdc_t)^{-1/\beta} \etam/L$, the $\vdc_t$-th quantile of $\calE=\Pareto(\etam/L, \beta)$, and $\gamma_t=\eta_t/\rho$, where $\{\vdc_i\}_{i\in \dN}$ is the base-2 van der Corput sequence,
    then we have for all $T$,
    \begin{align*}
        \GN(z_T)\leq LR\exp(-C\kappa^{-1-\eps}T)\exp(C'/2).
    \end{align*}
    In other words, for any $\delta>0$, a sufficient iteration count to achieve $\GN(z_T)\leq \delta$ is
    $T=\order_\eps(\kappa^{1+\eps}[1+\ln_+(LR/\delta)])$, where $\order_\eps(\cdot)$ hides $\eps$-dependent factors.
\end{apptheorem}

\begin{proof}
    Choose $\rho,\lambdam$ as in the proof of \cref{lem:dseg-cont}, and set $\etam=\lambdam\sqrt\rho$ and $C=-\beta\lambdam^\beta I_{\rho,\beta}(0)/8>0$. The same calculation as in the proof of \cref{lem:eg-cont-strongly} gives, for all $a\in[\mu,L]$,
    \begin{align*}
        \calR_{\rho,T}(\Pareto(\lambdam/L,\beta),a)
        \leq L\exp(-C\kappa^{-\beta}T).
    \end{align*}
    Applying \cref{lem:dseg-discretization} and \cref{lem:opt-strongly} proves the first claim. The tolerance bound follows by the same choice of integer horizon as in the proof of \cref{thm:main-eg-strongly}.
\end{proof}

For comparison, \citet[Theorem 3 and Corollary 4]{ibrahim2020linear} give an $\Omega(\kappa\ln(1/\delta))$ lower bound on the iterations needed to reduce relative distance to a solution to $\delta\in(0,1)$, under mild assumptions that EG satisfies. Our double-stepsize bound is this near-optimal in terms of the instance-dependent linear rate, up to an $\order_\eps(\kappa^\eps)$ factor. Further, this instance-dependent convergence is achieved in a \emph{uniform and instance-oblivious} way; our stepsize schedule does not depend on the value of $\mu$ or $\kappa$.
Linear dependence on $\kappa$ is attained in prior work by algorithms proposed by \citet{azizian2020accelerating} and \citet{shugart2025negative/preprint}\footnote{\citet{shugart2025negative/preprint} define $\kappa$ as $\frac{L^2}{\mu^2}$, so their $\order(\sqrt{\kappa})$ dependence is $\order(\kappa)$ in our notation.}, who both require knowledge of $\mu$ in their stepsize tuning; our result achieves near-linear dependence on $\kappa$ without this knowledge.

\section{Experiment Setup and Results}\label{ah:experiment-more}

We implement all stepsize schedules and run numerical experiments on them. The experiments are designed to mainly verify the polynomial growth rate of the gradient norm, i.e., the slope in the log-log plots.
The source code for the experiments is available in the GitHub repository \url{https://github.com/EtaoinWu/powerlaw-eg}.

\subsection{Experiment setup}

\paragraph{Instance generation} For $1\times 1$ problems with $\ell(x,y)=axy$, we randomly sample $a\in [a_{\mathrm{min}} L, L]$ for $a_{\mathrm{min}}=\frac{1}{100T}$. The distribution is uniform in log, i.e., $\ln a \sim U[\ln(a_{\mathrm{min}} L), \ln L]$. The data is generated this way because our theory predicts the hard instance to be $a\approx \Theta(T^{-2/3})$ for single-stepsize power-law EG and $a\approx\Theta(T^{-1})$ for double-stepsize power-law EG; sampling $a$ from a uniform distribution in $[0, 1]$ will make hard instances unlikely, resulting in overly optimistic performances of the algorithms.

For the same reason, for $n\times m$ problems, we generate the matrix $A$ by randomly sampling $r=\min\{n,m\}$ singular values independently from $\ln\sigma_i\sim U[\ln(a_{\mathrm{min}} L), \ln L]$, and sample the singular vector directions by randomly generating orthogonal vectors using the Haar measure. We sample a saddle point $x^\star, y^\star$ uniformly among all points with a distance $R$ from the origin. A prescribed saddle point is necessary, because otherwise the instance may have no solution (e.g., $\min_x \max_y y$ is biaffine and has no solution). 
Due to the affine nature of the step, scaling isn't important, and we choose $L=1, R=1$ throughout.

\paragraph{Stepsize schedule} For our algorithms, we use the stepsize distributions specified in \cref{lem:eg-mixture} and \cref{lem:dseg-mixture}, i.e., the mixtures of a Pareto and a point mass, to optimize for the constant factors. We set $\etam=\frac{1}{\sqrt2}$ for both. The actual stepsize schedule in the experiment can be seen in \cref{fig:vdc-schedule-mixture}; we can still clearly see the dyadic structure. We also set $\eta=\frac{1}{\sqrt2 L}$ for our baseline algorithms with fixed stepsizes (EG, and EAG of \citet{yoon2021accelerated}). All setups of single-stepsize EG will have $\beta=100/66$, and all setups of double-stepsize EG will have $\beta=100/99$, corresponding to $\calO(T^{-0.66})$ and $\calO(T^{-0.99})$ convergence rates respectively.

\paragraph{Result compilation} For each experiment, $K$ individual random instances $\{\ell_i\}_{i\in [K]}$ will be generated. For deterministic stepsize schedules and (deterministic) baseline algorithms, the worst-case gradient norm is reported --- for each $t$, we show $\max_{i\in [K]} \GN_{\ell_i}(z_t)$ for $t=1,\dots, T$. For randomized i.i.d. stepsize schedules, each instance is repeated $Q$ times, and their geometric mean is counted; this is then maximized over $K$ instances. The geometric mean is chosen because it corresponds to our convergence theorems on $\EE[\ln\GN(z_T)]$.

The dashed lines are added to represent different slopes visually, which correspond to different values of $k$ in the growth rate $\calO(T^{-k})$. The intercepts of the lines have no specific meaning; they are chosen such that dashed lines are close to, but do not overlap with, the curves.

\subsection{Experimental Results}
\paragraph{Validating main theorems} To validate the claims of \cref{thm:main-eg} and \cref{thm:main-dseg}, we compared EG with constant stepsize, single-stepsize power-law EG, double-stepsize power-law EG, and EAG, on $K=128$ random games of $n\times m=4\times 4$ for a time horizon of $T=2\times 10^6$. The results are shown in \cref{fig:gradient-norm-compare}. The dotted lines are added for reference; their slope in the log-log plot represents the polynomial growth rate. The results confirm that each of our algorithms behaves closely to our prediction.

\paragraph{Validating dimension independence} Our stepsize schedules and analysis are completely agnostic to the dimension of the vectors. To show that the growth rate is irrelevant to the dimensions, we repeated the same setup, but with $n\times m=100\times 128$ and $T=10^5$. The results in \cref{fig:gradient-norm-compare-large} show that the last-iterate convergence trends are very similar, confirming that our analysis works across multiple dimensions.

\paragraph{Testing generalization to OG} Due to the similarity of EG and OG algorithms, we also applied our stepsize schedule to the OG family of algorithms. The base stepsize is changed to $\frac{1}{2L}$ for all three algorithms (OG with a fixed stepsize, OG with a power-law stepsize with $\beta=100/66$, and the AOG algorithm of \citet{cai2023doubly}); all other parameters are identical to our first experiment. The results are visualized in \cref{fig:gradient-norm-compare-og}. The AOG algorithm by \citet{cai2023doubly} is a variant of OG with anchoring, achieving also a last-iterate convergence of $\calO(T^{-1})$.

As analyzed in \cref{ah:og}, the OG algorithm (with a single gradient call and thus a single stepsize per iteration) can be accelerated to $\calO(T^{-0.577})$.
Empirically, OG achieves a better rate of $\calO(T^{-0.66})$, matching the convergence of EG.
This result verifies and improves upon the theory of the previous section, showing potential room for improvement in the theoretical analysis of OG.

\begin{figure}[tbh]
    \centering
    \includegraphics[width=.875\linewidth]{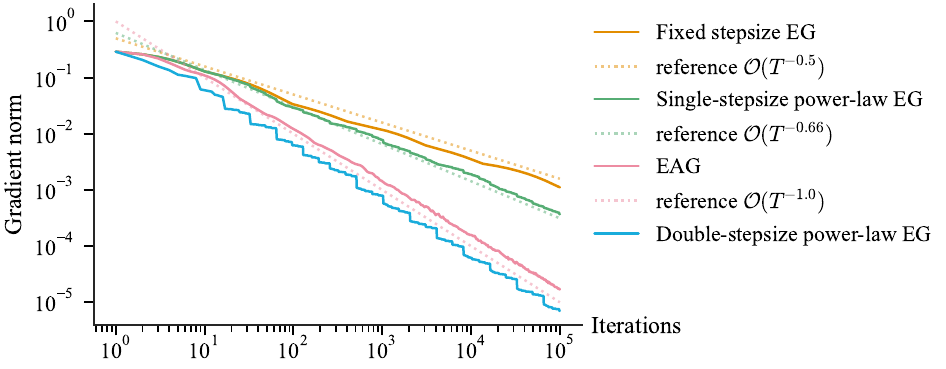}
    \caption{A repetition of \cref{fig:gradient-norm-compare} with larger matrices.
    }
    \label{fig:gradient-norm-compare-large}
\end{figure}

\begin{figure}[tbh]
    \centering
    \includegraphics[width=.875\linewidth]{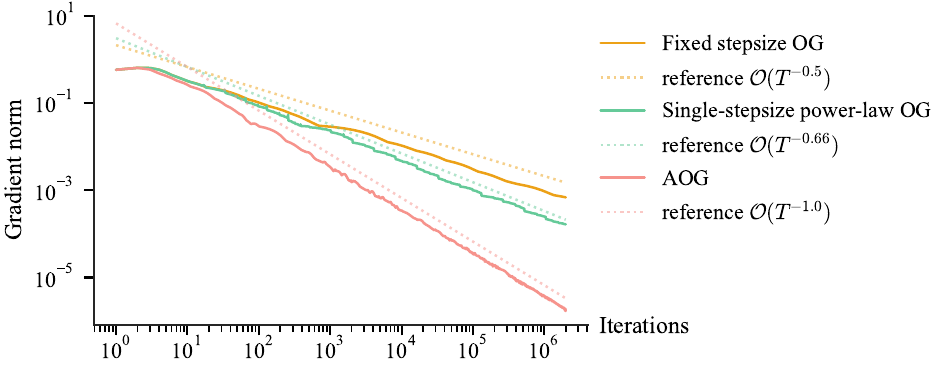}
    \caption{Last iterate convergence of OG with a fixed stepsize (orange), OG with power-law stepsize with $\beta=1/0.66$ (green), and the AOG algorithm of \cite{cai2023doubly} (pink). X-axis: iteration number $T$; Y-axis: $\GN(z_T)$. Both axes are displayed in log scale.
    The three dashed lines are the theoretical growth rates of the 3 algorithms for reference.
    }
    \label{fig:gradient-norm-compare-og}
\end{figure}

\end{document}